\documentclass[12pt]{amsart}
\usepackage[utf8]{inputenc}
\usepackage{natbib}

\usepackage[T1]{fontenc}
\usepackage{lmodern}
\usepackage{orcidlink}

\usepackage{amsmath,amstext,amssymb,amsopn,amsthm}
\usepackage{algorithm, algpseudocode}
\usepackage{url,verbatim}
\usepackage{mathtools}
\usepackage{enumerate}

\usepackage{color,graphicx}

\usepackage[margin=30mm]{geometry}

\allowdisplaybreaks

\usepackage{longtable,booktabs,array}

\usepackage{epstopdf}

\numberwithin{equation}{section}

\usepackage{tikz}
\usetikzlibrary{decorations.pathreplacing}
\usetikzlibrary{intersections,calc,arrows.meta,patterns}
\usepackage{hyperref}

\makeatletter
\g@addto@macro\th@plain{\thm@headpunct{}}
\makeatother

\numberwithin{equation}{section}

\theoremstyle{plain}
\newtheorem{thm}{Theorem}
\newtheorem{lemma}{Lemma}

\theoremstyle{definition}
\newtheorem{example}{Example}
\newtheorem{fact}{Fact}
\newtheorem{defin}{Definition}
\newtheorem{remark}{Remark}
\usepackage{subcaption}
\DeclareMathOperator*{\argmin}{\arg\min}
\DeclareMathOperator*{\argmax}{\arg\max}
\DeclareMathOperator*{\Argmin}{\mathrm{Arg}\min}
\newcommand{\tr}{\mathrm{tr}}
\newcommand{\diag}[1]{\mathrm{diag}(#1)}
\newcommand{\odiag}[1]{\mathrm{odiag}(#1)}
\newcommand{\Spp}{\mathrm{S}_{++}}

\newcommand{\Sppone}{\mathrm{S}_{++}^{(1)}}

\newcommand{\Szero}{\mathrm{Sym}^{(0)}}
\newcommand{\R}{\mathbb{R}}
\newcommand{\Sym}{\mathrm{Sym}}
\newcommand{\sign}{\mathrm{sign}}
\newcommand{\Diag}{\mathrm{Diag}}
\newcommand{\Diagp}{\mathrm{Diag}_+}
\newcommand{\D}{\mathcal{D}}
\newcommand{\Supp}{\mathbf{s}_\ast}
\renewcommand{\v}[1]{\mathrm{vec}(#1)}
\makeatletter
\newcommand{\opnorm}{\@ifstar\@opnorms\@opnorm}
\newcommand{\@opnorms}[1]{%
	\left|\mkern-1.5mu\left|\mkern-1.5mu\left|
	#1
	\right|\mkern-1.5mu\right|\mkern-1.5mu\right|
}
\newcommand{\@opnorm}[2][]{%
	\mathopen{#1|\mkern-1.5mu#1|\mkern-1.5mu#1|}
	#2
	\mathclose{#1|\mkern-1.5mu#1|\mkern-1.5mu#1|}
}
\makeatother

\title[Identifying Network Hubs with the Partial Correlation Graphical LASSO]{Identifying Network Hubs with the Partial Correlation Graphical LASSO}

\author[M. Bogdan]{Ma{\l}gorzata Bogdan}
\email{malgorzata.bogdan@math.uni.wroc.pl}
\address{Mathematical Institute, University of Wroc{\l}aw, pl. Grunwaldzki 2/4, 50-384, Wroc{\l}aw}

\author[A. Chojecki]{Adam Chojecki
\orcidlink{0009-0008-2902-4096}}
\email{adam.chojecki@pw.edu.pl}
\address{Faculty of Mathematics and Information Sciences, Warsaw University of Technology, Koszykowa 75, \mbox{00-662} Warsaw, Poland}

\author[I. Hejný]{Ivan Hejný}
\email{ivan.hejny@stat.lu.se}
\address{Department of Statistics, Lund University, Box 743, SE-220 07 Lund, Sweden}

\author[B. Kołodziejek]{Bartosz Kołodziejek \orcidlink{0000-0002-5220-9012}}
\email{bartosz.kolodziejek@pw.edu.pl}
\address{Faculty of Mathematics and Information Sciences, Warsaw University of Technology, Koszykowa 75, \mbox{00-662} Warsaw, Poland}

\author[J. Wallin]{Jonas Wallin}
\email{jonas.wallin@stat.lu.se}
\address{Department of Statistics, Lund University, Box 743, SE-220 07 Lund, Sweden}

\thanks{For the purpose of Open Access, the authors have applied a CC-BY public copyright licence to any Author Accepted Manuscript (AAM) version arising from this submission.
}
\begin{document}

\begin{abstract}


Graphical LASSO (GLASSO) is a widely used method for estimating sparse precision matrices and learning undirected graphical models in high-dimensional settings. Because GLASSO penalizes entries of the precision matrix directly, however, it is not scale-invariant. Partial Correlation Graphical LASSO (PCGLASSO), introduced by Carter et al. (2024), addresses this limitation by penalizing partial correlations, which directly characterize conditional dependence. In this paper, we study both statistical and computational properties of the PCGLASSO estimator. Our main contribution is the introduction of a scale-invariant irrepresentability condition for PCGLASSO and the proof that this condition is sufficient for consistent model selection. We further show that this condition is weaker than the corresponding irrepresentability condition for GLASSO, helping to explain the improved empirical behavior of PCGLASSO in settings such as hub-structured graphs. In addition, we develop two efficient algorithms for computing the estimator and analyze the nonconvex optimization problem underlying PCGLASSO, deriving conditions for global uniqueness and showing consistency of all minimizers.
\end{abstract}


\maketitle


\section{Introduction}\label{sec-intro}

Estimating a sparse precision matrix is a cornerstone of modern high-dimensional statistics, providing a powerful tool for uncovering conditional independence structures in Gaussian graphical models. These models are widely applied in fields ranging from genomics to finance, where understanding the underlying network of relationships between variables is of paramount importance. The classical approach for this task is the Graphical LASSO (GLASSO), which has become a standard due to its computational tractability and theoretical guarantees \cite{friedman2008sparse, YuanLi07}. The GLASSO estimator is defined as the solution to a convex optimization problem:
\begin{align}\label{eq:glasso}
\hat{K}_{\mathrm{GLASSO}} = \argmin_{K\in\Spp} \left\{-\log\det(K)+\tr(S K) + \lambda \|K\|_{1,\mathrm{off}}\right\},
\end{align}
where $S$ is the sample covariance matrix from $n$ independent copies of a $p$-dimensional random vector $X$, $\|K\|_{1,\mathrm{off}} = \sum_{i\neq j} |K_{ij}|$
is the $\ell_1$-penalty on the off-diagonal entries, and $\lambda\geq 0$ is a tuning parameter.

Despite its success, the GLASSO suffers from a notable limitation: it is not scale-invariant. Because the penalty is applied to the raw entries of the precision matrix $K$, simply rescaling the variables can alter the estimated graph structure. This makes the results sensitive to data preprocessing choices, such as standardization. A more robust and often more interpretable approach is to enforce sparsity directly on the partial correlations,
\[        
P(K)_{ij}=-\frac{K_{ij}}{\sqrt{K_{ii}K_{jj}}},
\]
which are naturally normalized measures of conditional dependence. This motivates penalizing the likelihood based on the partial correlations:
\begin{align}\label{eq:pcg0}
\Argmin_{K\in\Spp} \left\{ -\log\det(K) + \tr(SK) + \lambda\| P(K)\|_{1,\mathrm{off}} \right\}.
\end{align}

We note that penalizing the raw off-diagonal elements of $K$ can work against the goal of attenuating strong conditional dependencies, since their magnitudes need not track those of the corresponding partial correlations. Indeed, one can construct a positive definite $3\times 3$ matrix $K$ such that $K_{12} < K_{13} < K_{23}$ but $|P(K)_{12}|>|P(K)_{13}|>|P(K)_{23}|$,
 e.g.,
 \[
K = \begin{pmatrix}
    1 & 1 & 2 \\ 
    1 & 4 & 3 \\
    2 & 3 & 25
\end{pmatrix}\implies
P(K) =  \begin{pmatrix}
    1 & -0.5 & -0.4 \\ 
    -0.5 & 1 & -0.3 \\
    -0.4 & -0.3 & 1
\end{pmatrix}.
\]
Thus, the largest raw off-diagonal entry may correspond to the weakest conditional dependence, so penalizing raw entries can produce the opposite of the intended sparsity effect.

However, directly penalizing the partial correlation matrix renders the problem non-convex in the precision matrix $K$. This paper focuses on a systematic study of an estimator based on this principle, known as the Partial Correlation Graphical LASSO (PCGLASSO).

\subsection{Problem setup}

Let $X=(X_1,\ldots,X_p)^\top$ be a zero-mean random vector with covariance matrix $\Sigma^\ast$ and precision matrix $K^\ast=(\Sigma^\ast)^{-1}$. Suppose we observe $n$ independent copies $(X^{(i)})_{i=1}^n$ of $X$ and $S$ is the sample covariance matrix.

Following \cite{Carter}, we leverage a natural factorization to handle the non-convex penalty in \eqref{eq:pcg0}. Any positive definite matrix $K$ admits a unique factorization
\[
K=DRD,
\]
where $R$ is a positive definite matrix with unit diagonal entries and $D$ is a diagonal matrix with positive entries. Here, $R_{ij} = - P(K)_{ij}$ for $i\neq j$ and $D^2 = \diag{K}$ is the diagonal matrix whose $(i,i)$ entry is $K_{ii}$. 

Rewriting the problem \eqref{eq:pcg0} in terms of $(R,D)$  makes the $\ell_1$-penalty convex in $R$, though it introduces a non-convex coupling $\tr(S D R D)$ between $R$ and $D$ in the likelihood term.

We define the PCGLASSO estimator as
\[
\hat{K}_{\mathrm{PCG}} =  \hat{D}\hat{R}\hat{D},
\]
with $(\hat{R},\hat{D})$ obtained by solving 
\begin{align}\label{eq:main_problem_s}
    (\hat{R},\hat{D}) \in \Argmin_{R,D} \Big\{ -\log\det(DRD) + \tr(S D R D) +  \lambda\| R\|_{1,\mathrm{off}} +2\,\alpha\log\det(D) \Big\}.
\end{align}
The optimization is over matrices $R$ in the set $\Sppone$ of positive definite matrices with unit diagonal and over diagonal matrices $D$ with positive diagonal entries. The parameters $\lambda\geq0$ and $\alpha<1$ serve as the hyperparameters of the method. 

Note that compared to \eqref{eq:pcg0}, we introduced a logarithmic penalty on the diagonal elements. In \cite{Carter}, $\alpha = 4/n$ is recommended based on univariate MSE arguments, but here we treat $\alpha$ as a free parameter.

Finally, it is worth noting that problem \eqref{eq:main_problem_s} is not only the $\ell_1$-penalized Gaussian log-likelihood but also coincides with the minimization of the penalized log-determinant Bregman divergence  \cite{ravikumar2011graphical,zwiernik2023entropic}. Unlike \eqref{eq:glasso}, whose objective is convex (and coercive when $S$ has positive diagonals), \eqref{eq:main_problem_s} remains non-convex even at $\lambda=\alpha=0$ due to the mixed term $\tr(S D R D)$.

A key property of the PCGLASSO estimator defined by \eqref{eq:main_problem_s} is its scale invariance. An estimator $\hat{K}(S)$ is scale-invariant if 
\[
\hat{K}(H S H) = H^{-1} \hat{K}(S) H^{-1}
\]
for every diagonal matrix $H$ with positive entries.
The PCGLASSO estimator satisfies this property \cite[Proposition 2]{Carter}, which allows us to reformulate the problem entirely in terms of the sample correlation matrix $C$, i.e.,
\[
C = H S H \quad \text{with} \quad H = \diag{S}^{-1/2}.
\]
Henceforth we work with the equivalent formulation
\begin{align}\label{eq:main_problem}
    (\hat{R},\hat{D}) \in \Argmin_{R,D} \left\{-\log\det(R)-2(1-\alpha)\log\det(D) +\tr( C DRD)+ \lambda \|R\|_{1,\mathrm{off}} \right\}.
\end{align}
Note that $(\hat{R},\hat{D})$ solves \eqref{eq:main_problem} if and only if $(\hat{R}, \diag{S}^{-1/2} \hat{D})$ solves \eqref{eq:main_problem_s}. 

Finally, note that the support (or sign pattern) of the PCGLASSO estimator \(\hat{R}\) depends only on the sample correlation matrix \(C\), rather than on the raw sample covariance matrix \(S\). For any fixed pair \((i,j)\) with \(C_{ij}^\ast \neq 0\), standardization removes nuisance marginal scales. Indeed, assuming \(X \sim \mathcal{N}_p(0,\Sigma^\ast)\), we have
\[
\mathrm{Var}\left( \frac{C_{ij}}{C_{ij}^\ast}\right) = 
\frac{1}{n} \frac{(1-(C_{ij}^\ast)^2)^2}{(C_{ij}^\ast)^2} \leq  \frac{1}{n}\left(1+\frac{1}{(C_{ij}^\ast)^2}\right)  = \mathrm{Var}\left( \frac{S_{ij}}{\Sigma^\ast_{ij}}\right).
\]
Thus, on a relative-error scale, the sample correlation is asymptotically less variable than the sample covariance. This suggests that correlation-based procedures may have an advantage for support recovery over covariance-based analogues such as the classical unstandardized GLASSO.

\subsection{Literature review}

Estimating a sparse precision matrix is a cornerstone of statistical learning, particularly for uncovering conditional independence structures in Gaussian graphical models. The seminal work on the GLASSO provided a tractable convex framework for this task by penalizing the Gaussian log-likelihood with an $\ell_1$-norm on the precision matrix entries \cite{friedman2008sparse, YuanLi07}. Despite its widespread adoption, a well-known limitation of the GLASSO is its lack of scale invariance. Since the penalty is applied to the raw precision matrix entries, rescaling variables can alter the estimated graph structure, making the results dependent on data preprocessing choices such as standardization.

This limitation motivated a rich line of research focused on estimators that are either inherently scale-invariant or directly target the partial correlations, which are naturally normalized measures of conditional dependence. Early work in this direction includes the Sparse Permutation Invariant Covariance Estimation (SPICE) method, which achieves scale invariance by penalizing only the off-diagonal elements of the precision matrix \cite{SPICE08}. Another major family of methods reframes the problem as a series of sparse regressions. The neighborhood selection framework of  \cite{meinshausen2006high} and the Sparse Partial Correlation Estimation (SPACE) method \cite{SPACE09} estimate the graph structure by regressing each variable against all others using the LASSO. These approaches are particularly effective at identifying hub structures but may yield asymmetric estimates of the precision matrix.

To address the symmetry issue while retaining the benefits of a regression-based formulation, subsequent methods have focused on jointly convex objectives. The CONCORD algorithm  \cite{CONCORD15}, for example, maximizes a convex surrogate likelihood composed of node-wise conditional likelihoods, ensuring a symmetric and positive-definite estimate with the same asymptotic guarantees as SPACE. These methods successfully provide scale-invariant estimation with the computational and theoretical advantages of convexity, including convergence to a unique global minimizer.

A more direct approach to penalizing partial correlations was proposed by \cite{Carter} with the PCGLASSO, the focus of our work. Unlike the aforementioned methods, PCGLASSO incorporates an $\ell_1$-penalty directly on the partial correlation values within the Gaussian log-likelihood. This formulation is arguably the most natural and interpretable way to enforce sparsity on conditional dependencies. However, this directness comes at a cost: the objective function is no longer convex due to the coupling of diagonal and off-diagonal elements in the likelihood. In their original work, \cite{Carter} proposed a simple numerical algorithm and provided compelling empirical evidence of PCGLASSO's superior performance, especially in recovering networks with heterogeneous variable scales. Yet, its practical implementation and theoretical underpinnings (including the characterization of the solution landscape, conditions for a unique solution, and formal model selection guarantees) remained largely unexplored.

Recent advances have sought to circumvent the non-convexity of direct partial correlation penalization. Two-stage approaches, for instance, first estimate the diagonal elements of the precision matrix and then solve a convex GLASSO-like problem for the off-diagonal elements, effectively turning the problem back into a convex one \cite{SPACE23}. Other work has focused on computational scalability for ultra-high-dimensional data through screening techniques that break the problem into smaller, parallelizable subproblems \cite{SCREEN16}.

While these alternative strategies are valuable, they sidestep the original non-convex problem posed by PCGLASSO. The central challenge, and the primary gap in the literature, is the lack of a comprehensive framework for understanding and solving the PCGLASSO problem as originally formulated. This paper aims to fill this gap by providing the first systematic treatment of the PCGLASSO estimator, including a highly efficient algorithm, a rigorous analysis of its theoretical properties in the non-convex setting, and novel results on model selection consistency that theoretically justify its empirical advantages.


A complementary strategy targets salient structure-such as hub nodes-without estimating the entire graph. The Inverse Principal Components for Hub Detection (IPC–HD) links hubness to the spectrum of the precision matrix, enabling direct hub estimation \cite{HUB25}. Because IPC-HD and related criteria operate on the precision (rather than the scale-free partial-correlation) matrix, they can be highly sensitive to near-collinearity and variable scaling, allowing a few strongly correlated variables to dominate hub scores; we examine this further in Section~\ref{sec:real_ex2}.

\subsection{Contribution of the paper}

While the PCGLASSO estimator was first defined in \cite{Carter}, its practical implementation and theoretical underpinnings remained largely unexplored. This paper provides a comprehensive framework for the PCGLASSO method, featuring a highly efficient algorithm and a systematic study of its theoretical properties. Our main contributions are:

\textbf{A novel and efficient algorithm}: We introduce a block coordinate descent algorithm that is substantially more efficient than previously suggested approaches. Our key algorithmic innovations include:
\begin{enumerate}
    \item A solution rooted in classical matrix theory for the $D$-subproblem. We reveal and exploit a surprising connection between the optimization over the diagonal matrix $D$ and the classical problem of scaling positive definite matrices, first studied by \cite{MO68}. By leveraging established results from this literature, we develop an efficient modified Newton-Raphson solver and derive crucial theoretical bounds on the solution.
    \item Two coordinate-descent solvers for the $R$-subproblem. We develop (i) a \emph{primal} coordinate-descent method that updates $R$ directly via closed-form element-wise updates, and (ii) an adapted GLASSO routine: an efficient \emph{dual} block-coordinate descent method in $W=R^{-1}$ that optimizes $R$ under the unit-diagonal constraint by modifying the classical GLASSO dual for this correlation-matrix structure.
\end{enumerate}


\textbf{A systematic study of theoretical properties}: We address the challenges arising from the non-convexity of the PCGLASSO objective function:
\begin{enumerate}
    \item Characterization of the solution landscape: We demonstrate that the objective function, while biconvex, is not globally convex and may admit multiple local and global minima.
    \item Conditions for a unique solution: We identify two practical and verifiable scenarios under which the problem has a unique global minimizer: when the regularization penalty $\lambda$ is small, and when the sample correlations are close to zero (i.e., the data correlation matrix $C$ is close to identity).
    \item Consistency of the estimator: We establish consistency results (Lemma \ref{lem:cons}), showing that all coordinate-wise minimizers converge to the true precision matrix as the sample size increases. This guarantees that despite the potential for multiple solutions in finite samples, the estimator is reliable in the asymptotic regime.
\end{enumerate}

\textbf{Asymptotic analysis and superior model selection}: We derive the low-dimen\-sional asymptotic distribution of the estimator and provide theoretical guarantees for model selection consistency (sparsistency). We introduce a novel, scale-invariant irrepresentability condition and show it is often significantly weaker than the corresponding condition for the standard GLASSO, providing a theoretical explanation for PCGLASSO's superior performance in recovering sparse networks, especially those with hub structures.

\textbf{Empirical validation}: Our theoretical findings are supported by extensive simulations and a real-data application, which confirm the computational efficiency and statistical accuracy of the proposed method, particularly in identifying hub-like structures where other methods falter.

The code for all the experiments is available at \url{https://github.com/PrzeChoj/pcglasso_article_code}.

\subsection{Structure of the paper}

The remainder of this paper is organized as follows. Section~\ref{sec:alg} introduces our efficient block coordinate descent algorithm, detailing the novel solvers for the diagonal scaling matrix $D$ and the partial correlation matrix $R$. In Section~\ref{sec:prop}, we conduct a thorough theoretical investigation of the PCGLASSO estimator. We analyze the non-convex objective function, establish conditions for the uniqueness of the solution, and derive consistency results. Furthermore, we study the estimator's asymptotic properties and introduce a new, weaker irrepresentability condition that guarantees model selection consistency. Section~\ref{sec:num_exp} provides empirical validation of the method through extensive simulations and a real-data analysis of a gene expression dataset. 

Additional computational, empirical, and theoretical material is provided in the Appendix. 
In Section~\ref{app:comparison_carter}, we present simulation studies comparing the performance of our proposed algorithms with the approach of \cite{carter_arxiv_2025}. Section~\ref{sec:appendix_applied} provides additional empirical results, including supplementary analyses of the prostate cancer RNA-seq data and a non-hub simulation study based on a covariance structure estimated from the same gene dataset.
  Section~\ref{appendix:algs} collects detailed pseudocode for the algorithms developed in the paper.
 Section~\ref{appendix_D_optim} provides theoretical and empirical justification for the diagonal Hessian approximation used in the optimization over $D$ in Section~\ref{sec:algD}. 
Finally, Section~\ref{sec:proofs} contains the proofs of all theoretical results.

\subsection{Notation}\label{sec:notation}
Fix $p\in\mathbb{N}$. Denote by $\Sym$ the set of symmetric $p\times p$ matrices,  $\Szero\subset\Sym$ consists of symmetric matrices with zero diagonal, and by $\Diag$ the set of $p\times p$ diagonal matrices. Let $\Sppone$ be the collection of positive definite matrices with unit diagonal, and $\Diagp$ the set of diagonal matrices with strictly positive diagonal entries. 

Let $\odot$ denote the Hadamard (entry-wise) product. For any $p\times p$ matrix $X$, define $\diag{X}=X\odot I_p$, which is the diagonal matrix whose entries are the diagonal elements of $X$, and $\odiag{X}=X-\diag{X}$. 
Let $e=(1,\ldots,1)^\top \in \R^p$ and define $J_p = e  e^\top$, which is the $p\times p$ matrix with all entries equal to $1$. Moreover, set $J_p' = J_p - I_p = \odiag{J_p}$.  

For a function $f\colon \Omega\to \R$, define $\Argmin_{x\in\Omega} \{ f(x) \} = \{ x\in\Omega \colon f(x)\leq f(y)\,\mbox{for all } y\in\Omega\}$. 
In particular, we write $\hat{x} = \argmin_{x\in\Omega} \{f(x)\}$ if the minimizer is unique. 

We define two norms on $\R^{p\times p}$ by
\begin{align*}
    \|A\|_\infty = \max_{i,j} |a_{ij}|\quad\mbox{and}\quad
\opnorm{A} = \max_{i=1,\ldots,p} \sum_{j=1}^p | A_{ij}|.     
\end{align*}
Note that $\opnorm{\cdot}$ is the operator norm induced by the $\ell_\infty$ vector norm on $\R^p$. 

\section{Algorithm}\label{sec:alg}
We present an optimization framework for estimating the regularized precision matrix model defined by \eqref{eq:main_problem}. 
Our approach combines coordinate descent with specialized convex optimization techniques, as detailed in the following subsections. 

While the problem \eqref{eq:main_problem} is not globally convex, it is biconvex (see Lemma \ref{lem:biconvex} in Section \ref{sec:conv}). Therefore, we employ a coordinate descent approach, alternating between:
\begin{enumerate}
  \item Optimizing in  \(D\) holding \(R\) fixed.
  \item Optimizing in  \(R\) holding \(D\) fixed.
\end{enumerate}
While such an alternating algorithm was proposed in \cite{Carter}, details for solving the individual subproblems were not provided, and a different numerical approach was ultimately implemented. Such details were later provided in \cite{carter_arxiv_2025}.  

We take advantage of the fact that optimization in $D$ is related to the classical problem of scaling positive definite matrices, first studied by \cite{MO68}. The algorithm for updating $R$ is a modification of the GLASSO algorithm \cite{friedman2008sparse}.

\subsection{Optimization in \texorpdfstring{$D$ given $R$}{D given R}}\label{sec:algD}

We note that all terms involving $D$ in \eqref{eq:main_problem} can be written as
\[
\tr(CDRD) -2(1-\alpha)\log\det(D)= d^\top (R\odot C) d - 2(1-\alpha)\sum_{i=1}^p \log(d_i),
\]
where $d=(D_{ii})_{i=1}^p\in(0,\infty)^p$ and $\odot$ denotes the Hadamard product. Thus, minimization in $D$ is equivalent to minimizing the function $f(d)= \tfrac12 d^\top A d -\sum_{i=1}^p\log(d_i)$, where $A = (R\odot C)/(1-\alpha)$ is positive definite (see Lemma \ref{lem:A>0}). The unique stationary point $d$ of this logarithmic barrier function is characterized by the vector equation $A\, d = d^{-1}$,
where $d^{-1} = (1/d_{i})_{i=1}^p$ is the component-wise inverse of $d$. 
This system can be equivalently written in the form
\begin{align}\label{eq:scaling}
D A D e = e,\qquad \mbox{where } e=(1,\ldots,1)^\top\in\R^p.
\end{align}
The problem of finding a solution to \eqref{eq:scaling} for a given positive definite matrix $A$ was considered by \cite{MO68}. When $A$ has nonnegative entries, such a problem originally arose in estimating the transition matrix of a Markov chain known to be doubly stochastic; see  \cite{Sin64}.

Building on the results of \cite{Kal92}, we prove the following result:
\begin{thm}\label{thm:Dregion}
For any $R\in\Sppone$, correlation matrix $C$ and $\alpha<1$, \eqref{eq:scaling} has a unique solution $D\in\Diagp$. 

Moreover, if $C$ is positive definite, then all diagonal entries of $D$ belong to the interval
\[
\left[ \frac{\sqrt{(1-\alpha)\lambda_{\min}(C)}}{p} ,\sqrt{\frac{p (1-\alpha)}{\lambda_{\min}(C)}}\right].
\]
\end{thm}
This theorem underpins our uniqueness and consistency results. Indeed, it implies that if $C\in\Spp$, then it is enough to consider $D$ in \eqref{eq:main_problem} to belong to a compact subset defined above. Note that the non-convexity of \eqref{eq:main_problem} comes mainly due to large values of the diagonal $D$.

In \cite{Kal92},  the Newton-Raphson method is used to solve \eqref{eq:scaling}. Let $d_n=D_ne\in\R^p$.  The $n$-th iteration is given by
\begin{align}\label{eq:NR}
d_n = d_{n-1}+ H_n^{-1} (d_{n-1}^{-1}-A\, d_{n-1}),
\end{align}
where $H_n = (D_{n-1}^{-2}+A)$ is the Hessian of the objective function $f$ evaluated at $d_{n-1}$. Once a good initialization is found (within 
$O(p^{1/2+\varepsilon})$ iterations), the optimal solution to tolerance $\tau$ is obtained in $O(\log(1/\tau))$ additional iterations \cite{Kal92}. However, each iteration requires solving the linear system $H_n \delta_n = d_{n-1}^{-1}-Ad_{n-1}$ for the Newton direction $\delta_n=d_n-d_{n-1}$, which has a computational cost of $O(p^3)$. To reduce this cost, we approximate the Hessian with its diagonal part:
\[
H_n \approx  D_{n-1}^{-2}+\diag{A},
\]
reducing the per-iteration cost to $O(p^2)$. Justification for this diagonal approximation is provided in Appendix \ref{appendix_D_optim}. 
To guarantee convergence, we use the Line Search Algorithm that ensures the Wolfe conditions for $0 < c_1 = 10^{-4} < c_2 = 0.9 < 1$. We present the pseudocode describing the algorithm in the Appendix  \ref{sec:optD}.

\subsection{Optimization in \texorpdfstring{$R$ given $D$}{R given D}}\label{sec:algR}

The $R$-update solves the GLASSO convex problem over the
unit-diagonal cone $\Sppone$, 
\begin{align*}
\hat{R} \in \argmax_{R\in\Sppone}
\Big\{ \log\det(R) - \tr(RS) - \lambda\|R\|_{1,\mathrm{off}} \Big\},
\end{align*}
where $S$ is a positive semidefinite matrix.
We consider two alternative coordinate descent solvers.

\textbf{Primal coordinate descent in $R$.}
Algorithm~\ref{alg:updateR} updates $R$ directly; we refer to this implementation as \texttt{pcglassoFast\_Primal}.
It cycles over columns $i=1,\dots,p$; in the $i$th step, it freezes the principal block
$R_{-i,-i}$ and updates the off-diagonal vector $r=R_{-i,i}$.
The inner update of $r$ is performed element-wise (cycling $j=1,\dots,p-1$) using the closed-form
coordinate maximizer from Theorem~\ref{thm:element_update}; see Algorithm~\ref{alg:updater}.

\textbf{Dual coordinate descent in $W=R^{-1}$.}
Algorithm~\ref{alg:R} solves the dual problem; we refer to this implementation as \texttt{pcglassoFast\_Dual}. From Lemma~\ref{lem:dual},
\[
\hat{R}^{-1}=\argmax_{W\in\Spp}
\left\{\log\det(W)-\tr(W):\ |W_{ij}-S_{ij}|\le\lambda\ \ \forall i\neq j\right\}.
\]
This mirrors the classical GLASSO dual with the key difference that the unit-diagonal constraint $\diag{R}=I_p$ introduces the additional
$-\tr(W)$ term. As in \cite{Banerjee,friedman2008sparse}, updating a single column/row of $W$ can be reduced
to a LASSO regression, solved efficiently by coordinate descent with soft-thresholding. The resulting routine
is a minor adaptation of \textsc{glassoFast} \cite{glassoFAST}.

We compare both algorithms in Appendix~\ref{app:comparison_carter}, together with reference implementation \cite{carter_arxiv_2025}. Our experiments indicate that the dual coordinate descent algorithm performs best in most considered settings, while for the \texttt{hub\_1} structure no uniform ordering holds and the relative performance depends strongly on $(\lambda,\alpha)$ and the initialization.

\section{Theoretical properties of the estimator}\label{sec:prop}

\subsection{Convexity issues}\label{sec:conv}

A function $f\colon \mathcal{R}\times\mathcal{D}\to \R$ is called biconvex if, for every fixed $R\in\mathcal{R}$, the map $D\mapsto f(R,D)$ is convex, and for every fixed $D\in\mathcal{D}$, the map $R\mapsto f(R,D)$ is convex. If these maps are strictly convex in each argument, we say $f$ is strictly biconvex. A thorough introduction to biconvex functions can be found in \cite{biconvex}. 

As noted in \cite[Proposition 4]{Carter}, the objective is convex in $R$; the next lemma strengthens this to strict biconvexity and clarifies that global convexity holds only for $C=I_p$. 
\begin{lemma}\label{lem:biconvex}
    The objective function in \eqref{eq:main_problem} is strictly biconvex, but not globally convex unless $C=I_p$. 
\end{lemma}

Biconvexity does not imply global convexity. As a result, biconvex problems can admit multiple local minima, and standard global convexity guarantees (such as a unique global minimum) fail to apply in general.

In Section \ref{sec:alg} we proposed a coordinate descent algorithm for solving \eqref{eq:main_problem}. 
The algorithm stops at a coordinate-wise minimizer (also called a partial optimum in \cite{biconvex}) for the objective function $f$ in \eqref{eq:main_problem}, i.e., at a point $(\hat{R},\hat{D})$ such that for every $R$ and $D$,
\[
f(\hat{R},D)\geq f(\hat{R},\hat{D}) \leq f(R,\hat{D}).
\] 
However, it is well known in biconvex optimization that a coordinate-wise minimizer need not be a local minimum when both variables are perturbed simultaneously.
Each coordinate-wise minimizer corresponds to a critical point of the objective function \cite[Corollary 4.3]{biconvex}.

\begin{lemma}\label{lem:charact}
Any coordinate-wise minimizer $(\hat{R},\hat{D})$ of the objective \eqref{eq:main_problem} 
is defined by 
\begin{align}\label{eq:R-DCD}
\hat{R}^{-1}-\hat{D} C \hat{D} =  \lambda\Pi + \alpha I_p  - \lambda\,\diag{J_p' |\hat{R}|},
\end{align}
where $\Pi\in\partial\|\hat{R}\|_{1,\mathrm{off}}$ and $|\hat{R}|=(|\hat{R}_{ij}|)_{ij}$.
\end{lemma}

\begin{fact}
     The problem \eqref{eq:main_problem} may admit multiple minimizers.
\end{fact}
We illustrate this with a simple $2\times 2$ example. 
 \begin{example}\label{ex:2on2}
Consider $p=2$ with $\alpha=0$ and $\lambda=1$, and choose $\rho=C_{12}=\frac{e^{r_0}\sqrt{1-r_0^2}-1}{r_0}\approx 0.91$, where $r_0\approx -0.85$ is the unique negative solution to 
$\sqrt{1-r_0^2} = e^{r_0} (1-r_0+r_0^3)$. Let $d=(1+r_0 \rho)^{-1/2}$. 
Then the objective in \eqref{eq:main_problem} has two global minima: at $(\hat{R},\hat{D})=(I_2,I_2)$ and at 
\(
(\hat{R},\hat{D})=\left( \begin{pmatrix}
    1 & r_0 \\ r_0 & 1
\end{pmatrix}, \begin{pmatrix} d & 0 \\ 0 & d \end{pmatrix}
\right).
\)

Furthermore, if we vary $\lambda$,  one can show that \eqref{eq:main_problem} has:
\begin{itemize}
\item unique global minimum for $\lambda\in[0,\rho)$,
\item two local minima for $\lambda \in[\rho, 1.168]$,
\item unique global minimum at $(R,D)=(I_2,I_2)$ for $\lambda> 1.168$.
\end{itemize} 
More generally, we can show that in the case $p=2$ with $\alpha=0$, problem \eqref{eq:main_problem} has a unique solution for all $\lambda \geq 0$ if and only if $|\rho|\leq \frac{\sqrt{3+2\sqrt{3}}}{3}\approx 0.85$.
 \end{example}

Even though multiple solutions may exist, the following consistency result states that they are not far from each other.
\begin{lemma}\label{lem:cons}
If $C$ is positive definite, then each coordinate-wise minimizer $\hat{K}$ of \eqref{eq:main_problem} satisfies the following bound:
\[
\| \hat{K}^{-1}-C\|_\infty  \leq  \frac{(\lambda p +|\alpha|)p^2}{(1-\alpha)\lambda_{\min}(C)}.
\]
\end{lemma}

\begin{remark}\label{rem:localI}
\begin{enumerate}
    \item We note that if $\|C-I_p\|_{\infty}\leq\frac{\lambda}{1-\alpha}$, then $(\hat{R},\hat{D}) = (I_p,\sqrt{1-\alpha}I_p)$ is a local minimum of \eqref{eq:main_problem}. Indeed, it is easy to verify that \eqref{eq:R-DCD} holds in such a case.
    \item Since $\diag{\hat{D}C\hat{D}} = \hat{D}\,\diag{C}\hat{D}=\hat{D}^2$, by \eqref{eq:R-DCD}, we obtain $\hat{D} = d(\hat{R})$, where
\begin{align}\label{eq:Dexplicit}
d(R)^2 = \lambda \,\diag{J_p' |R|} + \diag{R^{-1}}-\alpha I_p.
\end{align}
Thus, very surprisingly, $\hat{D}$ is expressed as an explicit function of $\hat{R}$, even though the minimizer in $D$ does not offer an explicit formulation beyond the $p=2$ case. 

 We note that it is natural to substitute the optimization in $D$ (which is based on solving \eqref{eq:scaling}) by  \eqref{eq:Dexplicit}. However, our numerical simulations show that the benefit of a faster update of $D$ is offset by the increased number of steps in the main coordinate descent iteration. Moreover, since the update \eqref{eq:Dexplicit} is not optimal, ensuring the algorithm's theoretical convergence would require us to know that it does not increase the loss function - and that does not seem easy to prove.

Substituting $D=d(R)$ into \eqref{eq:scaling} shows that $\hat{R}$ lies on a smooth manifold described by a system of $p$ equations in $p(p-1)/2$ variables $(R_{ij})_{i>j}$, 
\[
d(R) (R\odot C) d(R) e = (1-\alpha)e
\]
in contrast to the non‑smooth constraint \eqref{eq:R-DCD}. It would be interesting to exploit this observation by reformulating the original problem \eqref{eq:main_problem} as a manifold-constrained programme.
\end{enumerate}

\end{remark}

\subsubsection{Uniqueness of the solution}\label{sec:unique}

In the Example \ref{ex:2on2}, we saw that \eqref{eq:main_problem} has a unique solution in two scenarios:  small $\lambda$ or small correlations.
Below, we generalize these observations to arbitrary dimensions.

\begin{thm}\label{thm:unique}\ 
\begin{enumerate}
    \item[(i)] 
    If $\|C-I_p\|_{\infty} \leq (2(1-\alpha) p^3)^{-1/2}$, 
    then for any $\lambda\geq 0$, \eqref{eq:main_problem} admits a unique local minimum. 
    \item[(ii)] For any $C\in\Sppone$, there exist $\lambda_0>0$ and $\alpha_0>0$ such that, for every $\lambda\in(0,\lambda_0)$ and $\alpha\in(-\infty,\alpha_0)$, \eqref{eq:main_problem} admits a unique local minimum.
    \end{enumerate}
\end{thm}

\subsection{Low dimensional asymptotics and sign recovery}\label{sec:low_dim}

In this subsection, we consider the classical asymptotic regime with $p$ fixed and let $n\to\infty$. Recall the setup in which we observe $n$ independent copies $X^{(1)},\dots, X^{(n)}$ of a centered random vector $X=(X_1,\dots,X_p)^\top\in\mathbb{R}^p$ with covariance matrix $\Sigma^\ast=(K^\ast)^{-1}$. Throughout, we shall assume that the fourth moments $\mathbb{E}[X_j^4]<\infty$ exist for every $j\in\{1,\dots,p\}$.
Suppose that $\lambda_n=\gamma n^{-1/2}$ and $\alpha_n=o(n^{-1/2})$ for fixed $\gamma>0$. Then, by Theorem \ref{thm:unique} the PCGLASSO estimator is unique for sufficiently large $n$ and by Lemma \ref{lem:cons} it is strongly consistent (since $\|C-C^\ast\|_\infty \to 0$ a.s.). 
We reformulate PCGLASSO optimization problem in a way consistent with the general asymptotic results obtained in \cite{hejny2025asymptotic}. We assume that
\begin{align}\label{eq:tuning}
\lambda_n = \gamma\,n^{-1/2}\quad\mbox{ and } \quad\alpha_n = o(n^{-1/2})
\end{align}
Then, the PCGLASSO estimator \eqref{eq:main_problem_s} can be written in the form 
\begin{align*}
       \hat{K}_n&= \argmin_{K\in \Spp}\left\{ n^{-1}\sum_{i=1}^n\ell(X^{(i)},K)+n^{-1/2}\gamma\operatorname{Pen}_n(K)\right\},
\end{align*}
where $\ell(X,K)=-\log\det(K)+\tr(KXX^\top)$ is the negative log-likelihood of the Gaussian model, and $\operatorname{Pen}_n(K) = \| P(K)\|_{1,\mathrm{off}} + o(1) \log\det(\diag{K})$.

We shall also define
\[
f'(K;U)=\lim_{\varepsilon\to 0}\frac{1}{\varepsilon}\left(f(K+\varepsilon U)-f(K)\right),
\]
the directional derivative of $f$ at $K$ in direction $U\in\Sym$. 

Using the results of \cite{hejny2025asymptotic}, we have the following.
\begin{thm}\label{thm:conv_in_dist}
Assume that $X$ has a finite fourth moment. The error $\sqrt{n}(\hat{K}_n-K^\ast)$ converges in distribution to the random variable $\hat{U}$, defined as the minimizer of
\begin{align}\label{vec(U) objective}
\hat{U}=\argmin_{U\in\Sym}\left\{ \frac{1}{2}\operatorname{vec}(U)^\top \Gamma^\ast\, \operatorname{vec}(U) - W^\top \operatorname{vec}(U) + \gamma \operatorname{Pen}'(K^\ast; U)\right\},
\end{align}
where $\operatorname{Pen}(K) = \| P(K)\|_{1,\mathrm{off}}$, $\Gamma^\ast=\Sigma^\ast\otimes\Sigma^\ast$, $W\sim\mathcal{N}_{p^2}(0,C_{\triangle})$, $C_{\triangle}= \operatorname{Cov}( \operatorname{vec}(XX^\top))$.  Moreover, \begin{align*}
\lim_{n\to\infty}\mathbb{P}\left(\sign(\sqrt{n}(\hat{K}_n-K^\ast))=\mathcal{S}\right)=\mathbb{P}\left(\sign(\hat{U})=\mathcal{S}\right),
\end{align*}
for every sign pattern $\mathcal{S}\in\{\sign(U)\colon U\in \Sym\}$.
\end{thm}

\subsection{Sign recovery}\label{sec:recovery}
For $X\in \R^{p\times p}$, define the vectorization operator $\v{X}\in\R^{p^2}$ obtained by stacking the columns of $X$ into a single column vector. Let $\mathrm{P}_{\mathrm{diag}}$ be the orthogonal projection matrix satisfying $\mathrm{P}_{\mathrm{diag}} \v{X}=\v{\diag{X}}$ for all $X\in\R^{p\times p}$. Denote $\mathrm{P}_{\mathrm{diag}}^\bot = I_{p^2}-\mathrm{P}_{\mathrm{diag}}$. 

\begin{defin}
We decompose the true precision matrix as 
$K^\ast=D^\ast R^\ast D^\ast$.
Let
  \[
    \tilde{\Gamma}=\mathrm{P}_{\mathrm{diag}}^\bot((R^\ast)^{-1}\otimes(R^\ast)^{-1})+\frac12\mathrm{P}_{\mathrm{diag}}(((R^\ast)^{-1}\otimes I_p)+(I_p\otimes (R^\ast)^{-1}))
\]
and let $\Supp$ be the support of $K^\ast$ (equivalently the support of $R^\ast$), i.e., 
\[
\Supp = \{ (i,j)\in \{1,\ldots,p\}^2\colon K^\ast_{ij}\neq 0\}.
\]
    The irrepresentability condition for PCGLASSO is given by
\begin{align}\label{new irrepresentability condition}
  \mathrm{IRR}_{\mathrm{PCG}}(K^\ast) = \| \tilde{\Gamma}_{\Supp^c\Supp} (\tilde{\Gamma}_{\Supp\Supp})^{-1}\v{\Pi}_{\Supp}\|_{\infty}<1,
\end{align}
where $\Pi_{ij} = \sign(K^\ast_{ij})$ if $i\neq j$ and $\Pi_{ii}=0$. 
\end{defin}
Note that the scale invariance of the PCGLASSO method implies the scale invariance of the irrepresentability condition, which is manifested by its lack of dependence on the $D^\ast$ matrix.

We are now ready to present the main result in this section. It establishes model selection consistency for the PCGLASSO estimator under the irrepresentability condition. 

\begin{thm}\label{thm:pattern_recovery} 
Assume that \eqref{eq:tuning} holds with $\gamma>0$ and let $\hat{K}_n$ denote the solution to \eqref{eq:main_problem_s} with $(\lambda,\alpha)=(\lambda_n,\alpha_n)$. Under the irrepresentability condition \eqref{new irrepresentability condition}, there exists $c>0$, independent of $\gamma$, such that
\begin{align*}
     \lim_{n\rightarrow \infty}\mathbb{P}\left(\sign(\hat{K}_n)=\sign(K^\ast)\right)\geq 1-e^{-c\,\gamma^2}.
\end{align*}
Conversely, when $\mathrm{IRR}_{\mathrm{PCG}}(K^\ast)\geq 1$, the limiting probability is bounded from above by $1/2$.
\end{thm}

\subsubsection{Comparison with GLASSO}\label{sec:recoveryGLASSO}

\cite{Carter} observed empirically that the PCGLASSO estimator, partly due to its scale invariance, possesses better sign recovery properties than the GLASSO estimator. This is a direct consequence of the irrepresentability condition for PCGLASSO being generally much weaker than the corresponding condition for the GLASSO, which we recall below. 

Let $\Gamma^\ast = \Sigma^\ast\otimes \Sigma^\ast$. Then, the GLASSO irrepresentability condition is 
\[
 \mathrm{IRR}_{\mathrm{GLASSO}}(K^\ast)= \| \Gamma^\ast_{\Supp^c\Supp} (\Gamma^\ast_{\Supp\Supp})^{-1} \mathrm{vec}(\Pi)_{\Supp} \|_\infty < 1,
\]
where the set $\Supp$ and the matrix $\Pi$ are the same as in \eqref{new irrepresentability condition}. The GLASSO irrepresentability condition is necessary for the sign recovery by the GLASSO estimator in the sense of Theorem \ref{thm:pattern_recovery}.

The main feature is that \eqref{new irrepresentability condition} depends only on the partial correlation matrix
$R^\ast$, making it inherently scale-invariant. In contrast, the GLASSO irrepresentability condition depends on the entire matrix $\Sigma^\ast$, and is therefore not scale-invariant.

\begin{example}
For the hub example, the irrepresentability condition is more favorable for PCGLASSO than for GLASSO. 
Consider the matrix $K^\ast$ representing a hub graph, defined by
\[
K^\ast_{11} = a,\quad K^\ast_{ii} = b \,\ (i\ge2),\quad K^\ast_{1i}=K^\ast_{i1}=c \,\ (i\ge2),\quad K^\ast_{ij}=0 \ \text{otherwise}.
\]
For PCGLASSO, the irrepresentability value can be shown to be:
\begin{align*}
    \mathrm{IRR}_{\mathrm{PCG}}(K^\ast) =  \frac{|c|}{\sqrt{a b}}\left(2-(p-1)\frac{c^2}{ab}\right). 
\end{align*}
Since the matrix $K^\ast$ is positive definite if and only if $c^2/(ab)<(p-1)^{-1}$, it can be easily verified that 
\[
\mathrm{IRR}_{\mathrm{PCG}}(K^\ast) \leq \frac{4\sqrt{2}}{3\sqrt{3}}\frac{1}{\sqrt{p-1}}= O(p^{-1/2}), 
\]
which implies that the PCGLASSO irrepresentability condition \eqref{new irrepresentability condition} is satisfied for all such matrices for $p\geq 3$. By contrast, the irrepresentability value for GLASSO is $\mathrm{IRR}_{\mathrm{GLASSO}}(K^\ast) = 2 |c|/b$, which implies that the GLASSO irrepresentability condition is very restrictive.

Figure \ref{Fig:IRR} displays the heatmaps of the values $\mathrm{IRR}_{\mathrm{GLASSO}}(K^\ast)$ (top, for GLASSO) and $\mathrm{IRR}_{\mathrm{PCG}}(K^\ast)$ (bottom, for PCGLASSO) for $b=1$ and $p=15$. 

\begin{figure}[h]
\centering {
\includegraphics[width=12cm]{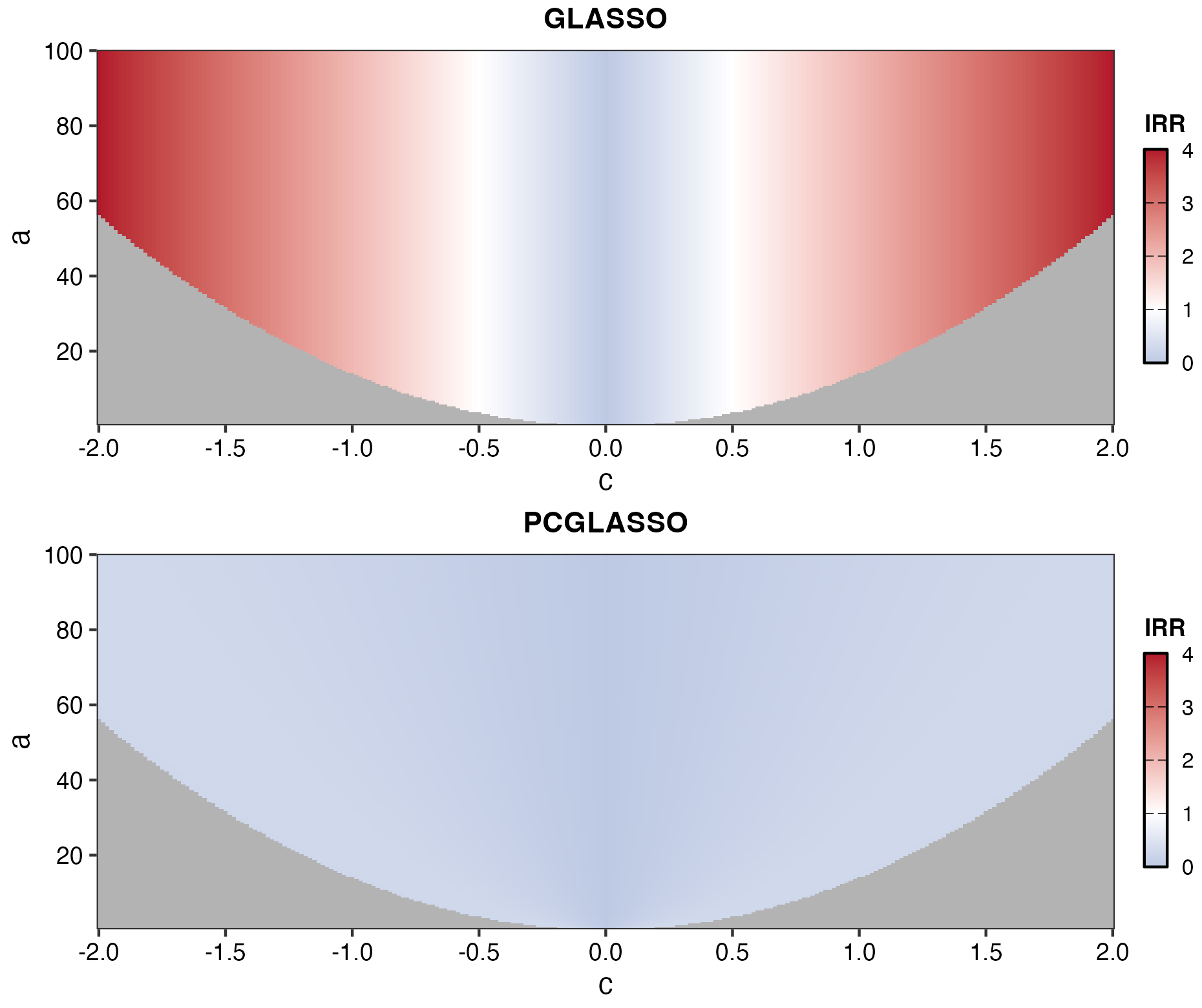}
}
\caption{Heatmaps of the $\mathrm{IRR}$ values for a hub graph on $p=15$ vertices. Top: GLASSO; bottom: PCGLASSO. The matrix is defined by $K^\ast_{1,1}=a$, $K^\ast_{i,i}=1$ for $i\ge2$, and $K^\ast_{1,i}=K^\ast_{i,1}=c$ (with all other entries zero). Green indicates regions where the IRR condition is satisfied (i.e., the value is below $1$), while gray marks regions where $K^\ast$ is not positive definite (i.e., $a \le (p-1)c^2$).\label{Fig:IRR}}
\end{figure}
\end{example}

The bottom heatmap is uniformly green, indicating that the \eqref{new irrepresentability condition} is satisfied for all tested values of $a$ and $c$.  In contrast, the top heatmap displays only a narrow green strip, revealing that the GLASSO condition is far more restrictive and holds only when the conditional dependence between the hub and spoke nodes is weak.

When applied to chain-graph models, PCGLASSO again surpasses GLASSO, but the advantage is considerably less pronounced than in the case of hub models.

\section{Real data analysis and simulations}\label{sec:num_exp}

\subsection{Gene Expression Omnibus}\label{sec:real_ex}
In this section, we compare different versions of GLASSO for identifying the graphical model behind the genome-wide gene expression data from lymphoblastoid cell lines of HapMap individuals, made publicly available by \cite{GSE6536} through the NCBI Gene Expression Omnibus (GEO accession: GSE6536). We used the data of 210 unrelated individuals from four distinct populations (60 Utah residents with ancestry from northern and western Europe, 45 Han Chinese in
Beijing, 45 Japanese in Tokyo, 60 Yoruba in Ibadan, Nigeria), which was previously studied, e.g., in \cite{bradic2011, fanfanbarut14, Wojtek, frommlet22}.  
The major goal of the analysis in \cite{frommlet22} was to identify genes whose expression levels can be used  to predict the expression level of the gene CCT8, which appears within the Down syndrome critical region on human chromosome 21. Such analyses can be used to identify genes whose expression is associated with (and may regulate) CCT8. In this work, we perform this task using the graphical model tools, which can provide additional information about structure of partial correlations among the selected genes.

The original dataset contains expression levels measured for 47\,293 probes. Following the procedure described in \cite{Wojtek}, we pre-processed the data by removing probes that met either of the following two criteria: (i) the maximum expression level across the 210 individuals was below the 25th percentile of all measured expression levels, or (ii) the range of expression levels across individuals was less than 2. After this filtering step, we retained
$p=3\,220$ probes.

We then applied LASSO to select 124 probes that best predict the expression of CCT8. One probe exhibiting an unusually high variance was removed as an outlier. Consequently, the final set of variables used to construct the graphical model includes CCT8 and the 123 LASSO-selected probes.

Figures~\ref{Fig:XXX} and~\ref{fig:matrices} compare the performance of PCGLASSO with two variants of GLASSO, as well as with the SPACE method of~\cite{SPACE09}, on this dataset. The Cor-GLASSO approach applies the standard GLASSO algorithm to standardized data (i.e., the sample correlation matrix), and subsequently retransforms the estimates back to the original scale.

Figure~\ref{Fig:XXX} presents the values of the Extended BIC criterion of \citet{Foygel} as a function of the number of edges, for graphs obtained along the solution paths of the considered methods. The pronounced differences between the EBIC curves reflect substantial structural discrepancies among these paths. In particular, the EBIC values along the PCGLASSO path are consistently lower than those obtained for the GLASSO variants and the SPACE method, indicating that PCGLASSO achieves superior likelihood maximization for models of a given size.

The comparison further demonstrates that applying GLASSO to standardized data (i.e., the correlation matrix) yields improved performance relative to its direct application to the raw gene expression data. Nevertheless, both GLASSO-based approaches and SPACE are markedly outperformed by PCGLASSO in terms of likelihood values across their respective solution paths.

\begin{figure}
\centering{
\includegraphics[width=0.7\textwidth]{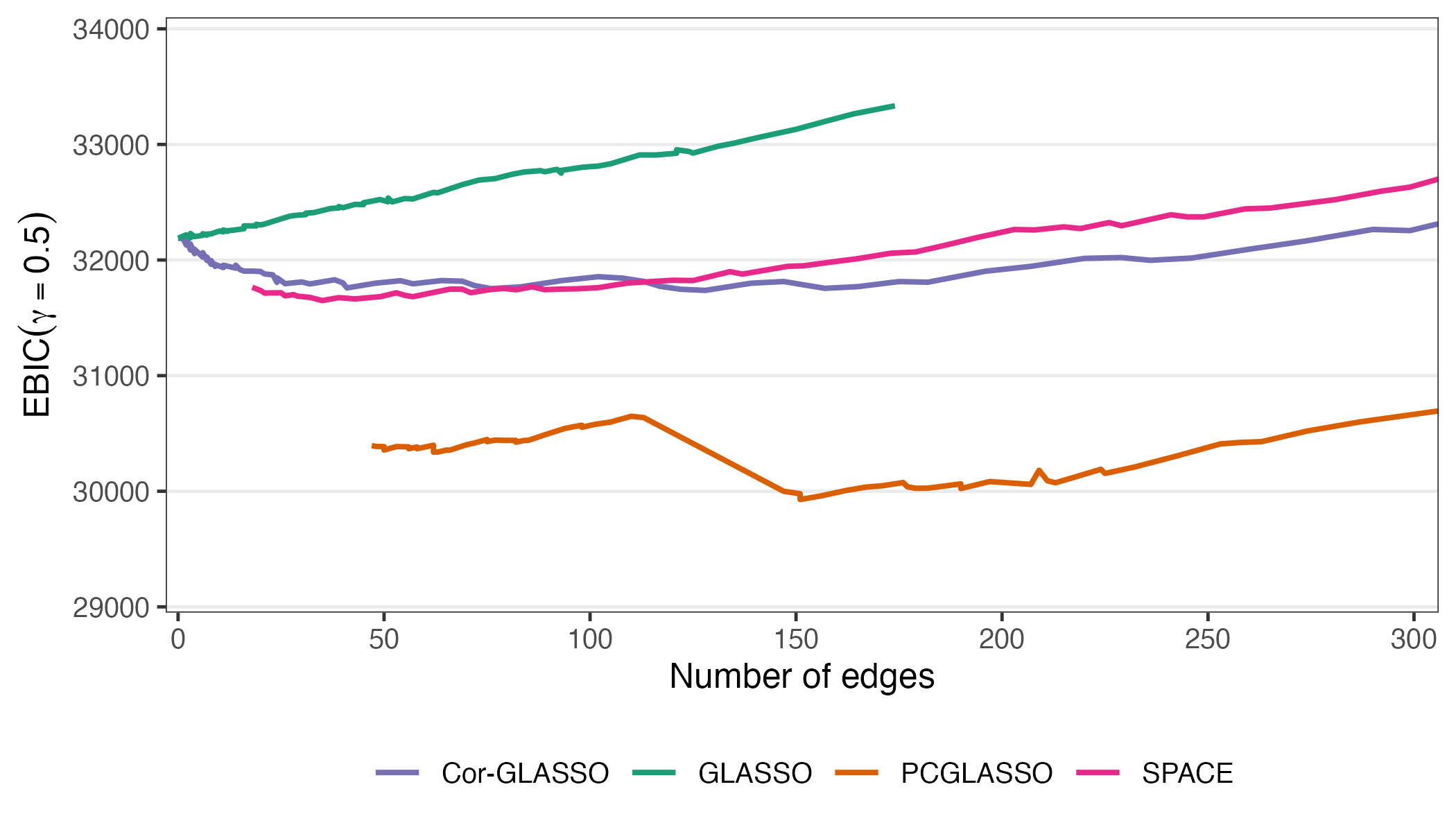}
}
	\caption{Values of EBIC along the paths of different methods.
    \label{Fig:XXX}}
\end{figure}

\begin{figure}[htbp]
    \centering
    \includegraphics[width=0.5\textwidth,trim=100 0 100 0,clip]{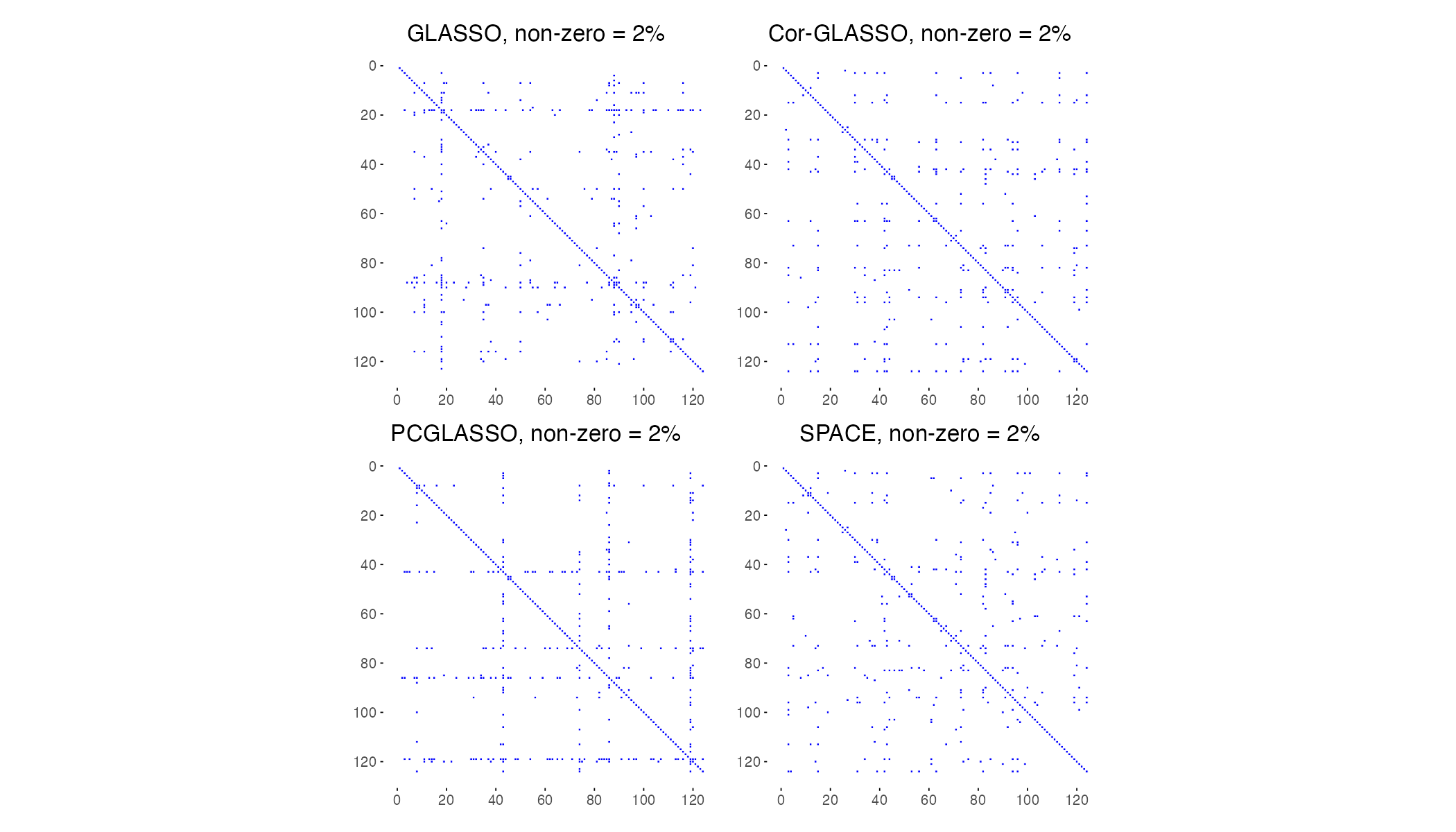}    
    \caption{Comparison of estimated sparse precision matrices: GLASSO, Cor-GLASSO, PCGLASSO, and SPACE (non-zero = 2\%).}
    \label{fig:matrices}
    \end{figure}
	
\begin{figure}
    \centering
    \includegraphics[width=0.4\linewidth]{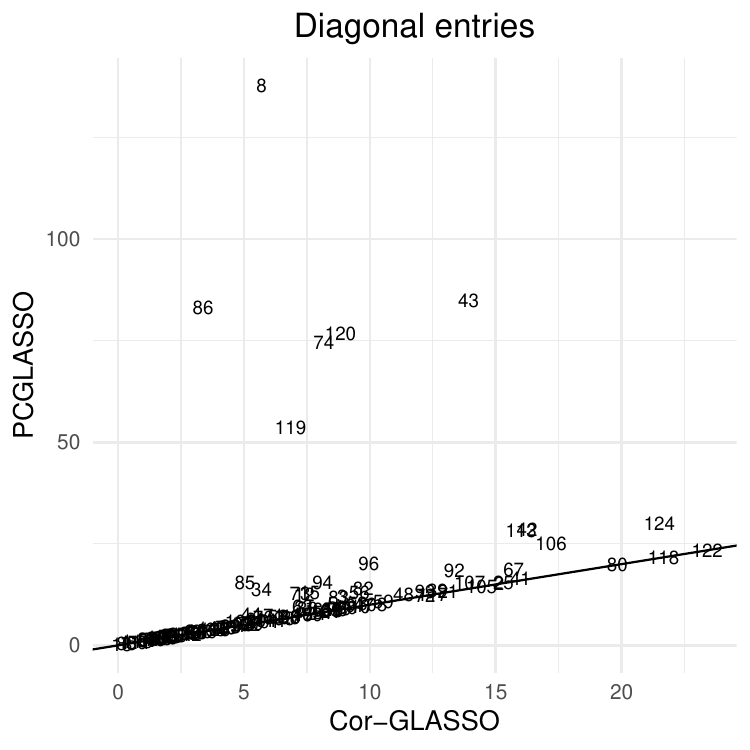}
\caption{Comparison of diagonal precision estimates under PCGLASSO and Cor-GLASSO. PCGLASSO produces larger diagonal values for hub nodes, while Cor-GLASSO exhibits substantial shrinkage.}
    \label{fig:placeholder2}
\end{figure}

\begin{figure}
    \centering
    \includegraphics[width=0.48\linewidth]{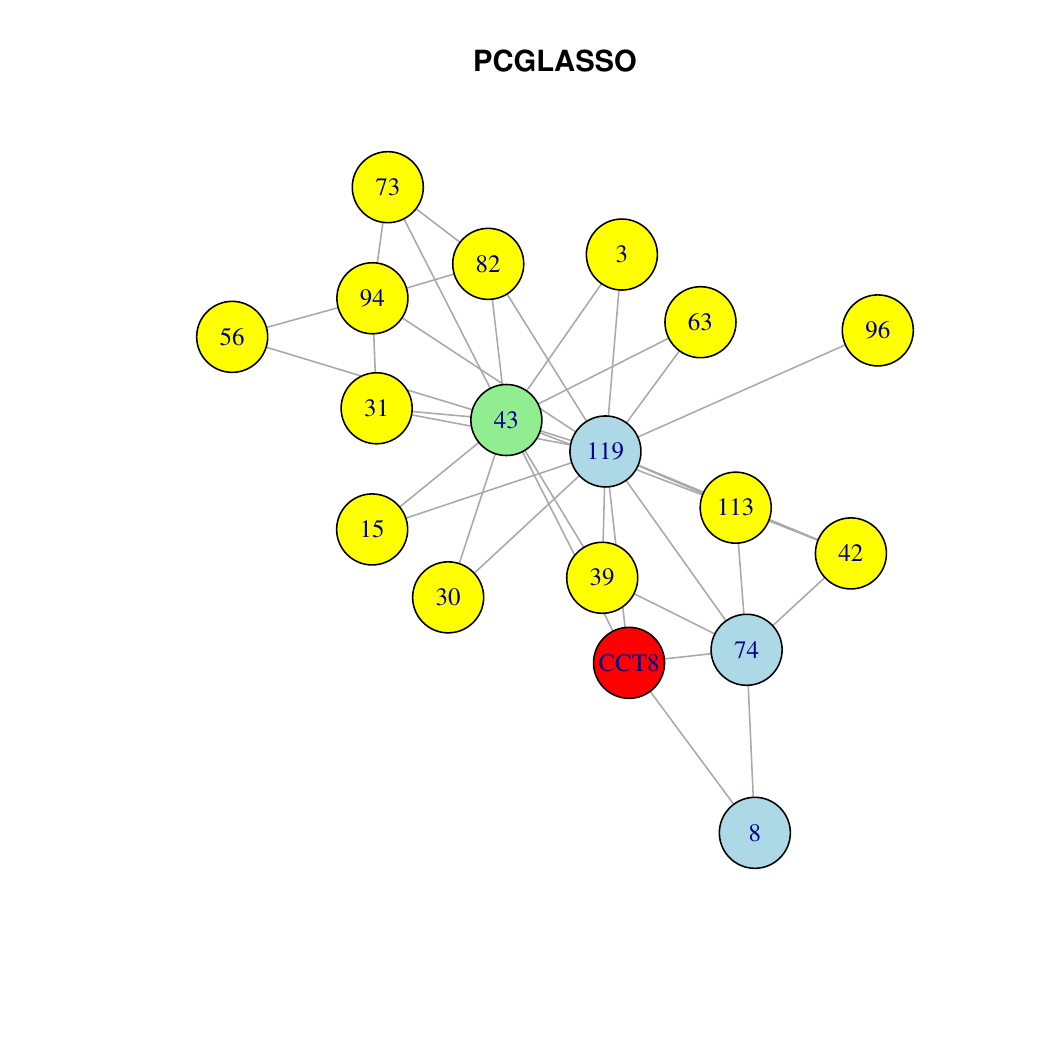}
    \includegraphics[width=0.48\linewidth]{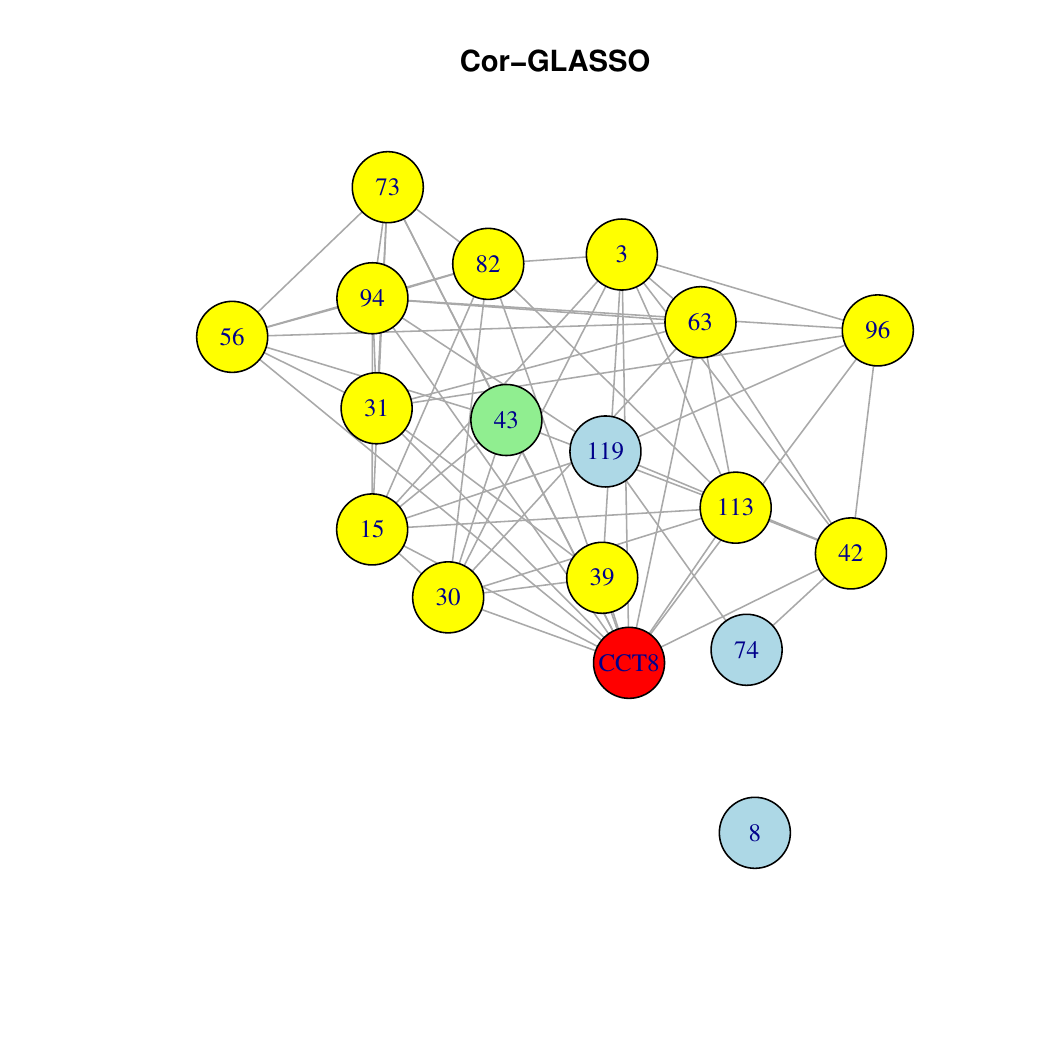}
\caption{Genes directly connected to CCT8 in the estimated graphs of PCGLASSO (left) and Cor-GLASSO (right). The two methods share only one common neighbor, gene 43 (green). Additional PCGLASSO neighbors are shown in blue, whereas genes selected only by Cor-GLASSO are shown in yellow.}
    \label{fig:placeholder1}
\end{figure}

Figure \ref{fig:matrices} highlights clear differences between the graphs produced by the various methods, even though they contain the same number of edges (2\% of the total number of gene pairs). The PCGLASSO model exhibits the most structured topology, with four prominent hubs corresponding to genes 43, 74, 86, and 119. In contrast, the GLASSO solution is more diffuse, featuring a single dominant hub associated with gene 18. The graphs obtained from Cor-GLASSO and SPACE are even more diffuse and do not display a clear structural organization.

Figure \ref{fig:placeholder2} provides insight into this phenomenon. It shows that Cor-GLASSO substantially shrinks the diagonal elements of the precision matrix, even though these elements are not directly penalized. This shrinkage limits the number and magnitude of the off-diagonal entries within each row, effectively restricting the formation of hubs. In contrast, PCGLASSO produces much larger diagonal estimates for hub nodes, thereby enabling the identification of the hub structure.

Figure \ref{fig:placeholder1} graphically compares the network of genes directly connected to CCT8 in the adjacency matrices obtained using PCGLASSO and Cor-GLASSO, as illustrated in Figure~\ref{fig:placeholder2}. Both methods identify only one common gene, 43 (marked in green), as a direct predictor of CCT8. This gene forms a strong hub in the PCGLASSO model, but not in the GLASSO-based models. 

Other genes connected to CCT8 by the PCGLASSO model are marked in blue and include two additional hubs, corresponding to genes 74 and 119, as well as gene 8, which itself is highly correlated (with correlation exceeding 0.977) with hub 86. This suggests that the hubs identified by PCGLASSO constitute the primary direct predictors of CCT8.

In contrast, the GLASSO model connects CCT8 to 13 additional genes (marked in yellow), resulting in a network that is considerably denser and less structured than that obtained with PCGLASSO.

Based on the shapes of the likelihood functions shown in Figure \ref{Fig:XXX}, together with our theoretical results demonstrating the superiority of PCGLASSO in accurately identifying hub structures, we believe that the model selected by PCGLASSO provides a more faithful representation of the dependencies between genes than the models obtained using the other methods.



\subsection{Prostate cancer patients}\label{sec:real_ex2}
We revisit the prostate cancer RNA-seq dataset of \citet{HUB25}. Among GLASSO, Cor-GLASSO, SPACE, and PCGLASSO, PCGLASSO attains the lowest EBIC (paths reported in the Appendix~\ref{sec:prostateCont}). Using partial-correlation diagnostics, our hub findings differ from \citet{HUB25} because their hub score is defined on the precision matrix and has a few nearly collinear columns. Specifically, \citet{HUB25} score nodes via the squared $\ell_2$-norm of the $i$th row of the precision matrix,
\[
\alpha_i^{(2)}(\hat K)=\sum_{j\neq i}\hat K_{ij}^{2}.
\]
 To mitigate scaling, we also rank by the absolute row sum of the partial-correlation matrix,
\[
\alpha_i^{(1)}(\hat R)=\sum_{j\neq i}\lvert \hat R_{ij}\rvert .
\]
The sorted $\alpha^{(1)}(\hat R)$ shows one outlier (indices, $53{=}\texttt{EEF2}$), see the top-right panel in Figure~\ref{fig:cancer_panels}, none highlighted in \citet{HUB25}. By contrast, $\alpha_i^{(2)}(\hat K)$ has five clear outliers which are the five hub genes reported by \citet{HUB25} ($151=\texttt{MIR3609}$, $174=\texttt{SCARNA7}$, $1=\texttt{SEMG1}$, $12=\texttt{SEMG2}$, $6=\texttt{RN7SK}$); Figure~\ref{fig:cancer_panels} bottom-left panel. The same five genes also appear among the largest diagonal precision entries, where we write $\hat D=\mathrm{diag}(\hat K)^{1/2}$; Figure~\ref{fig:cancer_panels} bottom-right panel. The large entries of $\hat{D}$, for these data, are due to collinearity between a few variables. For instance \texttt{SEMG1} and \texttt{SEMG2} are almost collinear (empirical correlation $r=0.999$), removing \texttt{SEMG2} drops $\alpha^{(2)}_{\texttt{SEMG1}}$ from $3{\,}777{\,}822$ (rank $3$) to $1{\,}010$ (rank $21$). Neither \texttt{SEMG1} nor \texttt{SEMG2} can reasonably be considered a hub; this observation reveals a limitation of precision-matrix-derived scores compared with partial-correlation–based criteria.
See Appendix~\ref{sec:prostateCont} for more details.
\begin{figure}[t]
  \centering
  \includegraphics[width=.4\linewidth,height=.3\linewidth]{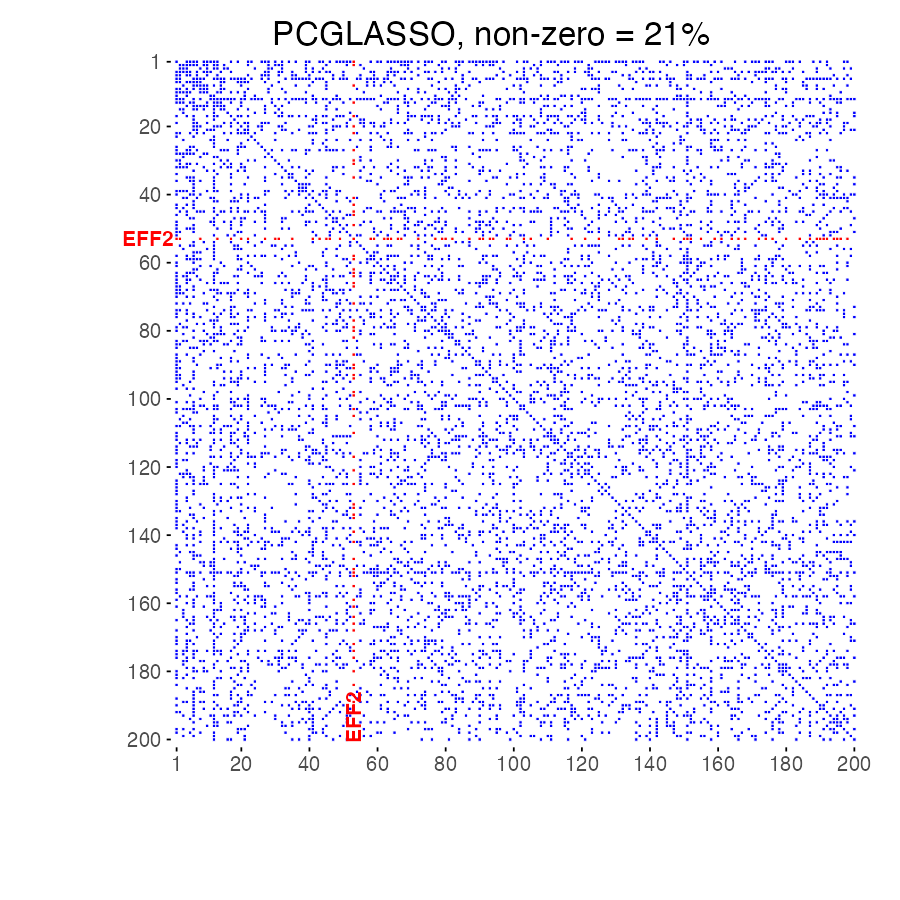}
  \includegraphics[width=.4\linewidth]{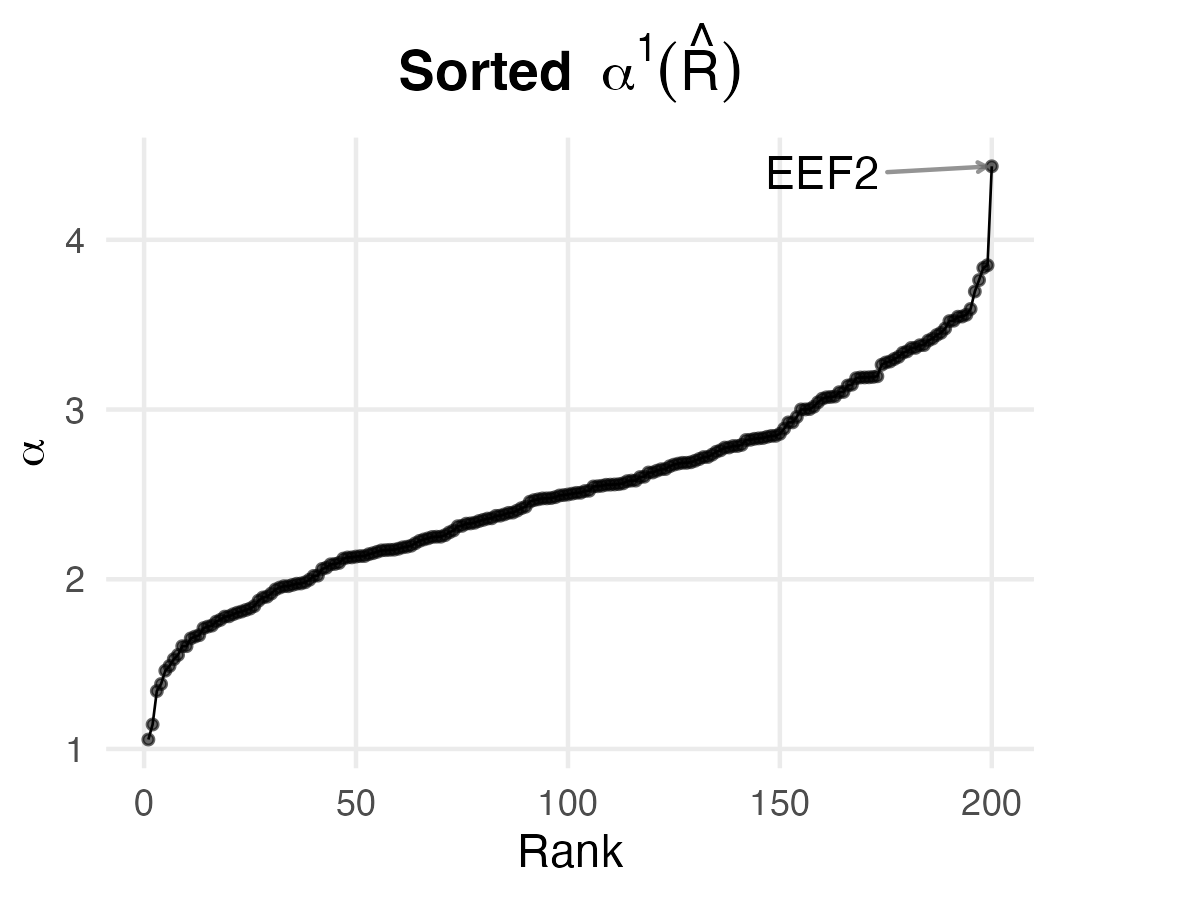}\\
  \includegraphics[width=.4\linewidth]{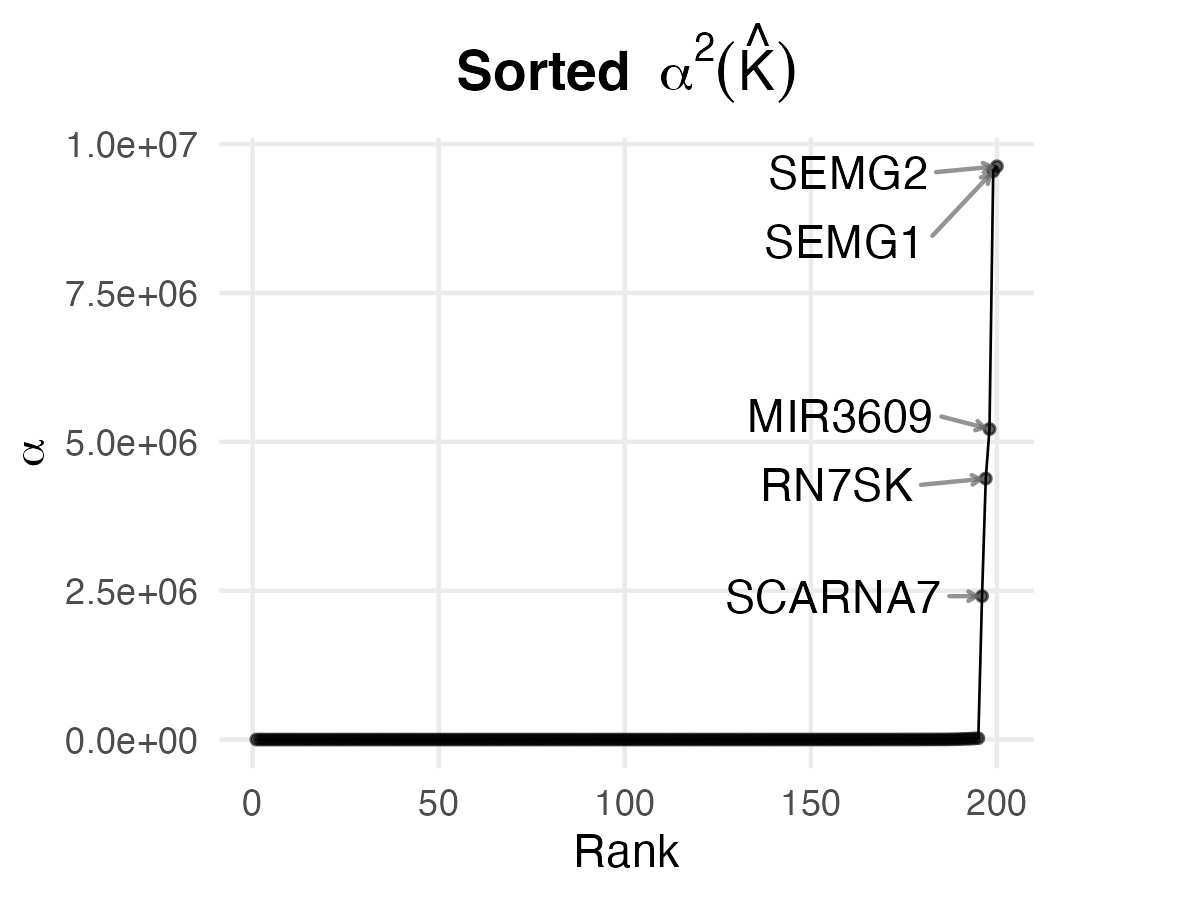}
  \includegraphics[width=.4\linewidth]{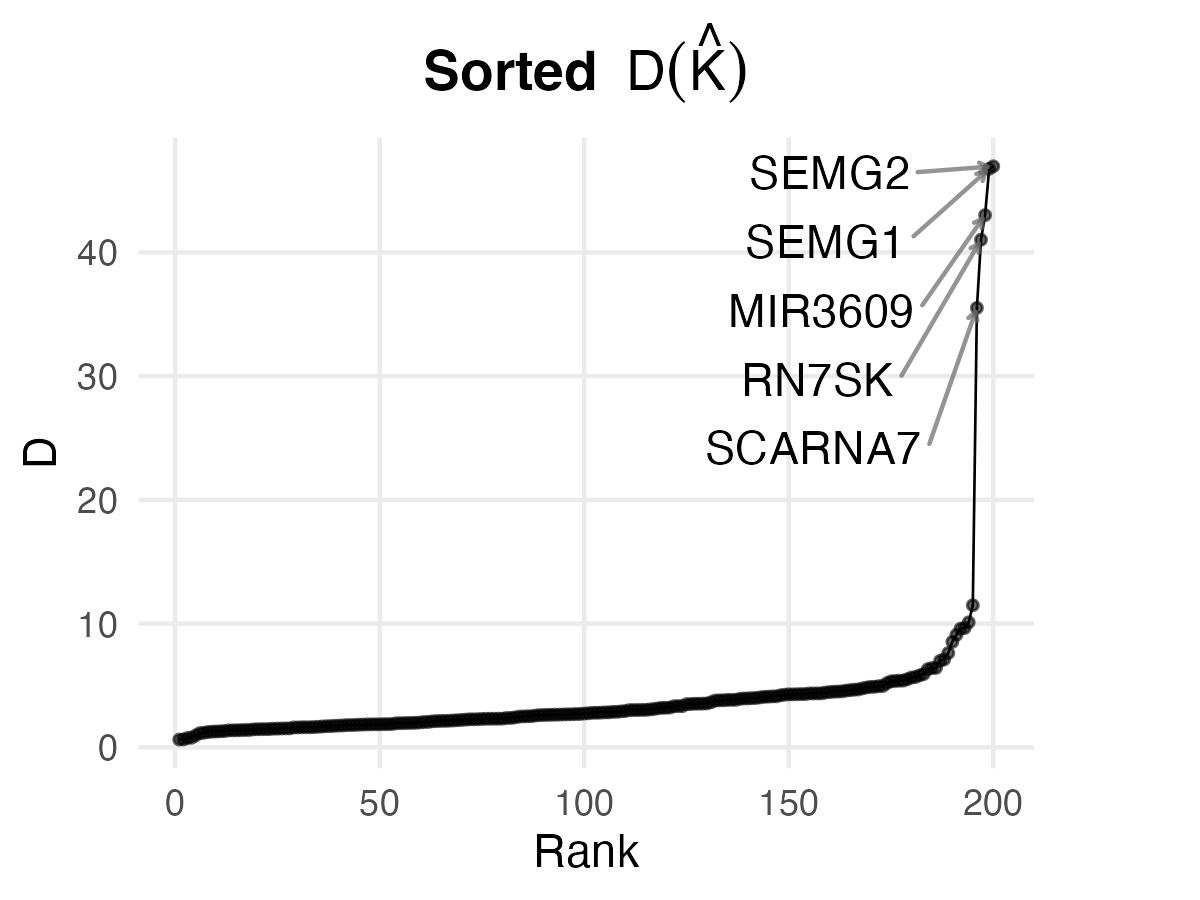}
  \caption{Prostate cancer RNA-seq network under PCGLASSO. (a) Estimated adjacency (nonzeros of $\hat K$). (b) Sorted $\alpha^{(1)}(\hat R)$ highlights one outlier \texttt{EEF2}. (c) Sorted $\alpha^{(2)}(\hat K)$ is dominated by near-duplicate pairs. (d) Large diagonal precision entries $\hat D=\mathrm{diag}(\hat K)^{1/2}$ pick out the five genes reported by \citet{HUB25}, but this reflects near-collinearity rather than hub-like connectivity.}
  \label{fig:cancer_panels}
\end{figure}

\subsection{Simulation study}\label{sec:simulations}
To validate the effectiveness of the proposed methods and benchmark them against existing approaches, we design a simulation study based on a covariance structure estimated from real data. Specifically, we estimate a covariance matrix $\Sigma$ from the gene dataset in Section~\ref{sec:real_ex} using PCGLASSO, with the regularization parameter chosen to yield a high level of sparsity. The corresponding precision matrix exhibits a clear hub structure, with a few highly connected nodes and many sparsely connected ones, as can be seen in Figure~\ref{fig:hub_structure}. A complementary simulation study based on a covariance which is taken from a GLASSO path is reported in Appendix~\ref{sec:appendix_applied}.

\begin{figure}[htbp]
    \centering
    \includegraphics[width=0.5\linewidth]{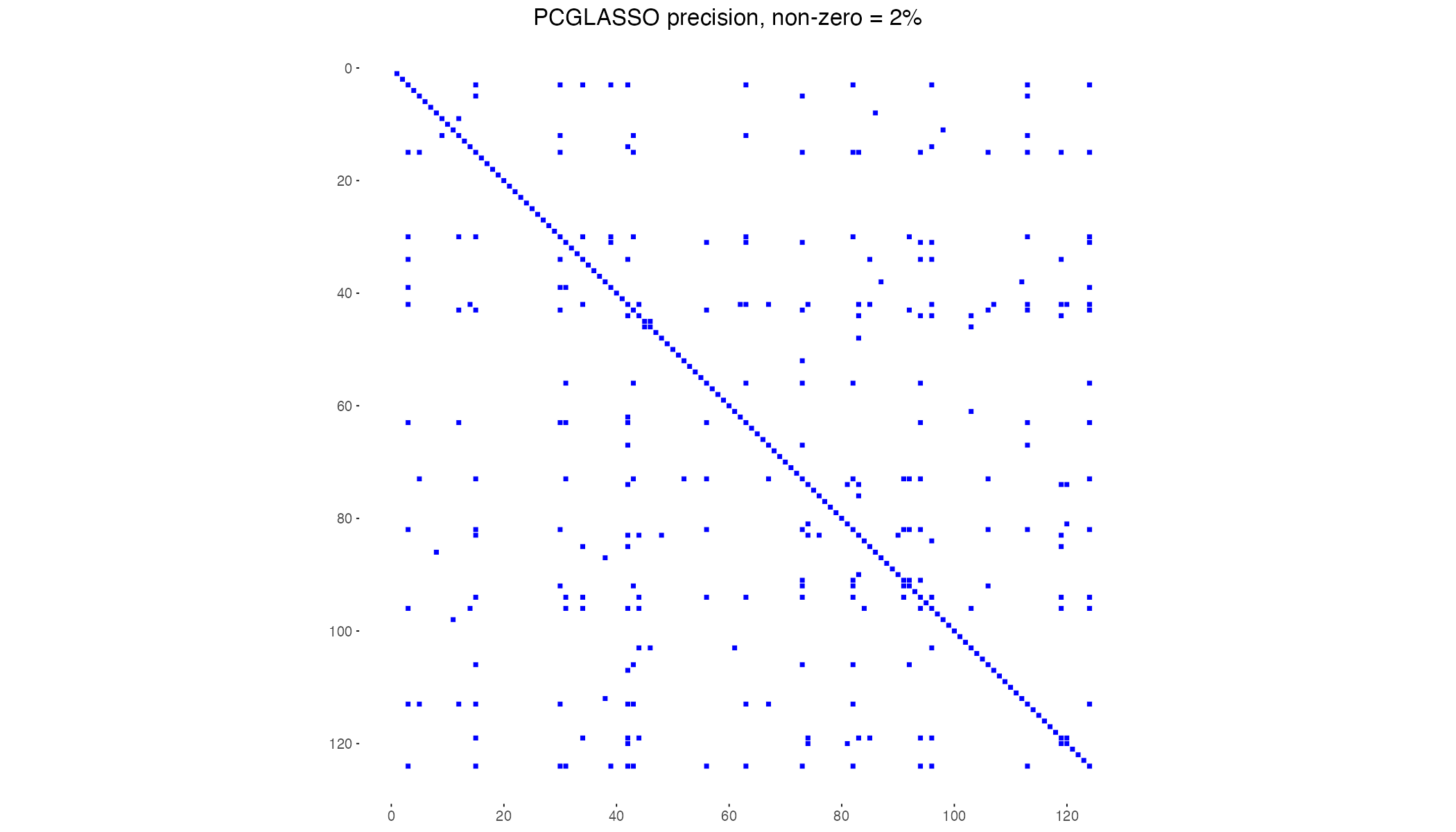}
    \caption{Nonzero pattern for the hub-structured precision matrix used in the simulation study.}
    \label{fig:hub_structure}
\end{figure}

We generate independent samples
$
X_i \sim \mathcal{N}_p(0,\Sigma)$ 
 for $i=1,\ldots,n,$
where $\Sigma$ is the PCGLASSO-based estimate.

We compare the performance of the following methods:
\begin{itemize}
    \item \textbf{GLASSO:} The GLASSO estimator.
    \item \textbf{Cor-GLASSO:} Estimation of the inverse correlation matrix via GLASSO.
    \item \textbf{SPACE:} The method proposed in \cite{SPACE23}, designed for sparse precision matrix estimation.
    \item \textbf{PCGLASSO:} The proposed Partial Correlation GLASSO method, implemented using either the \texttt{pcglassoFast\_Primal} or the \texttt{pcglassoFast\_Dual} algorithm.
\end{itemize}
For each method, hyperparameters are selected either by Bayesian Information Criterion (BIC) or by a single validation split (Val), with 70\% of the data used for training and 30\% for validation.

The main aim of this experiment is to evaluate the estimation error of PCGLASSO when the data are generated from a hub-structured precision matrix. We compare its performance with that of competing methods in terms of how accurately they estimate the underlying precision matrix, and we also assess their computational efficiency through timing comparisons.

We simulate datasets with sample sizes $n = 200$, $500$, $1\,000$, and $5\,000$. For each configuration, we compute the root mean squared error (RMSE) for the full matrix, the diagonal elements, and the nonzero off-diagonal elements. In addition, we record the computation time for each method. Each experiment is repeated $200$ times to assess the variability of the estimators.

The results are summarized in Tables~\ref{tab:hubrmse} and~\ref{tab:hubtime}. In the tables, CGL denotes Cor-GLASSO and GL denotes GLASSO. The PCGLASSO results are reported separately for the primal and dual algorithms. BIC indicates hyperparameter selection by BIC, while Val indicates selection by a single validation split. For the hub-structured precision matrix, PCGLASSO demonstrates the strongest overall performance in terms of RMSE, with SPACE performing competitively in some settings. The timing results show that SPACE is substantially slower than the other methods, especially as $n$ increases.

Overall, these simulations indicate that PCGLASSO performs very well in the hub-structured setting in terms of estimation error, while maintaining computational efficiency comparable to the other methods. The primal and dual implementations perform essentially equivalently.

\begin{longtable}[]{@{}llcccc@{}}
\caption{RMSE summary for each method and sample size.\label{tab:hubrmse}}\\
\toprule\noalign{}
Metric & Method & $n=200$ & $n=500$ & $n=1,000$ & $n=5,000$ \\
\midrule\noalign{}
\endfirsthead

\multicolumn{6}{c}{{\tablename\ \thetable{} -- continued from previous page}} \\
\toprule\noalign{}
Metric & Method & $n=200$ & $n=500$ & $n=1,000$ & $n=5,000$ \\
\midrule\noalign{}
\endhead

\midrule \multicolumn{6}{r}{{Continued on next page}} \\
\endfoot

\bottomrule
\endlastfoot

RMSE & CGL BIC & 1.50 & 1.37 & 1.27 & 0.96 \\
 & CGL Val & 1.39 & 1.24 & 1.11 & 0.73 \\
 & GL BIC & 1.63 & 1.53 & 1.38 & 0.99 \\
 & GL Val & 1.41 & 1.18 & 1.00 & 0.59 \\
 & pcglassoFast\_Primal BIC & \textbf{0.32} & 0.20 & 0.19 & 0.09 \\
 & pcglassoFast\_Primal Val & 0.37 & \textbf{0.18} & \textbf{0.13} & 0.06 \\
 & pcglassoFast\_Dual BIC & \textbf{0.32} & 0.21 & 0.18 & 0.07 \\
 & pcglassoFast\_Dual Val & 0.39 & 0.19 & 0.14 & \textbf{0.06} \\
 & SPACE BIC & 1.04 & 0.40 & 0.21 & 0.06 \\
 & SPACE Val & 0.60 & 0.27 & 0.17 & 0.07 \\

\midrule\noalign{}
Diag RMSE & CGL BIC & 12.30 & 11.20 & 10.30 & 7.84 \\
 & CGL Val & 11.30 & 10.10 & 8.98 & 5.92 \\
 & GL BIC & 13.10 & 12.50 & 11.30 & 8.08 \\
 & GL Val & 11.50 & 9.65 & 8.18 & 4.82 \\
 & pcglassoFast\_Primal BIC & 2.47 & 1.52 & 1.40 & 0.64 \\
 & pcglassoFast\_Primal Val & 2.91 & \textbf{1.38} & \textbf{0.99} & 0.46 \\
 & pcglassoFast\_Dual BIC & \textbf{2.46} & 1.55 & 1.33 & 0.49 \\
 & pcglassoFast\_Dual Val & 3.14 & 1.47 & 1.04 & 0.43 \\
 & SPACE BIC & 7.86 & 2.67 & 1.37 & \textbf{0.35} \\
 & SPACE Val & 4.19 & 1.82 & 1.12 & 0.41 \\

\midrule\noalign{}
Off-diag (NZ) & CGL BIC & 9.65 & 8.88 & 8.20 & 6.27 \\
RMSE & CGL Val & 8.97 & 8.04 & 7.17 & 4.75 \\
 & GL BIC & 10.70 & 9.89 & 8.90 & 6.41 \\
 & GL Val & 9.09 & 7.65 & 6.49 & 3.85 \\
 & pcglassoFast\_Primal BIC & \textbf{2.16} & 1.42 & 1.30 & 0.62 \\
 & pcglassoFast\_Primal Val & 2.35 & \textbf{1.17} & \textbf{0.86} & 0.42 \\
 & pcglassoFast\_Dual BIC & \textbf{2.16} & 1.44 & 1.25 & 0.51 \\
 & pcglassoFast\_Dual Val & 2.53 & 1.23 & 0.89 & \textbf{0.38} \\
 & SPACE BIC & 7.33 & 3.05 & 1.61 & 0.51 \\
 & SPACE Val & 4.36 & 1.91 & 1.15 & 0.47 \\

\end{longtable}

\begin{longtable}[]{@{}lcccc@{}}
\caption{Computation time (seconds) for each method and sample size. \label{tab:hubtime}}\\
\toprule
Method & $n=200$ & $n=500$ & $n=1,000$ & $n=5,000$ \\
\midrule
\endfirsthead

\multicolumn{5}{c}{{\tablename\ \thetable{} -- continued from previous page}} \\
\toprule
Method & $n=200$ & $n=500$ & $n=1,000$ & $n=5,000$ \\
\midrule
\endhead

\bottomrule
\endlastfoot

CGL BIC & 5.42 & 4.10 & 2.87 & 1.36 \\
CGL Val & 6.69 & 4.73 & 3.37 & 1.53 \\
GL BIC & \textbf{1.49} & \textbf{1.32} & \textbf{1.01} & \textbf{0.57} \\
GL Val & 1.78 & 1.52 & 1.12 & 0.62 \\
pcglassoFast\_Primal BIC & 3.77 & 2.74 & 2.46 & 2.17 \\
pcglassoFast\_Primal Val & 4.03 & 2.66 & 2.45 & 2.09 \\
pcglassoFast\_Dual BIC & 10.10 & 7.08 & 5.98 & 5.21 \\
pcglassoFast\_Dual Val & 11.40 & 8.25 & 6.16 & 4.92 \\
SPACE BIC & 10.40 & 28.10 & 55.40 & 307.00 \\
SPACE Val & 6.36 & 18.80 & 35.30 & 188.00 \\

\end{longtable}

\section*{Funding}
The research of BK and ACH was funded in part by National Science Centre, Poland, UMO-2022/45/B/ST1/00545. The research of MB, IH and JW was funded by the Swedish Research Council, grant no. 202005081.

This research was carried out with the support of the High Performance Computing Center at Faculty of Mathematics and Information Science Warsaw University of Technology.

\section{Data Availability Statement}\label{data-availability-statement}

No new primary data were collected for this study. The gene expression data
analyzed in Section~\ref{sec:real_ex} are publicly available from the NCBI Gene
Expression Omnibus under accession number GSE6536. The prostate cancer RNA-seq
data analyzed in Section~\ref{sec:real_ex2} are the data considered by
\citet{HUB25}; details on data access are provided in that reference. The
simulation results reported in this article can be reproduced using the
data-generating mechanisms and procedures described in Section~\ref{sec:simulations}
and in \url{https://github.com/PrzeChoj/pcglasso_article_code}. Code for implementing the proposed methods and
reproducing the numerical results is provided in \url{https://github.com/PrzeChoj/pcglassoFast}, \cite{chojeckiwallin2025pcglassofast}.

\appendix

\section{Comparison of PCGLASSO implementations under matched objective accuracy}
\label{app:comparison_carter}

In this appendix, we compare three implementations of the PCGLASSO optimization procedure:
\texttt{pcglasso}, which is reference implementation from \cite{carter_arxiv_2025};
\texttt{pcglassoFast\_Primal}, which implements the \texttt{pcglassoFast\_Primal} coordinate descent algorithm in $R$ described in Algorithm~\ref{alg:updateR};
and \texttt{pcglassoFast\_Dual}, which implements the \texttt{pcglassoFast\_Dual} coordinate descent algorithm in $W = R^{-1}$ described in Algorithm~\ref{alg:R}.
For each implementation, we also consider two starting points for $R$, denoted by \texttt{I} and \texttt{C}. Here \texttt{I} corresponds to the identity matrix, while \texttt{C} is defined as \texttt{cov2cor(solve(S))}.

The main goal of this comparison is to assess computational efficiency under a fair criterion.
A direct comparison of runtimes at default stopping rules may be misleading, since different algorithms may terminate at different objective values.
Therefore, instead of comparing raw runtimes, we compare the time needed to attain a given level of objective accuracy.

\subsection{Experimental setup}
We consider four graph structures:
\texttt{AR2}, \texttt{random}, \texttt{hub\_09}, and \texttt{hub\_1}.

For \texttt{AR2}, the nonzero off-diagonal entries are given by
\[
K^\ast_{i,i+1} = K^\ast_{i+1,i} = \frac12,
\qquad
K^\ast_{i,i+2} = K^\ast_{i+2,i} = \frac14,
\]
for all indices for which these entries are defined, and all remaining off-diagonal entries are zero.

For \texttt{random}, we start from the identity matrix and randomly generate off-diagonal entries with random signs and magnitudes uniformly distributed on $[0.4,1]$ until the number of nonzero off-diagonal entries is at least $3p/2$.
Next, in each column, the off-diagonal entries are rescaled by dividing them by $1.1$ times the sum of their absolute values.
The matrix is then symmetrized.
If the resulting matrix is not positive definite, the whole procedure is repeated until a positive definite matrix is obtained.

For the hub graphs, the nonzero off-diagonal entries adjacent to the hub are equal to $-1/\sqrt{p}$ for \texttt{hub\_1} and to $-0.9/\sqrt{p}$ for \texttt{hub\_09}.
The remaining off-diagonal entries are zero.

The simulations are performed for
$p \in \{50, 100, 150, 200\}$,
$n = 2p$,
$\lambda \in \{0.1, 0.2\}$,
and
$\alpha \in \{0, 0.5\}$.
Each configuration is repeated $M = 100$ times.

For each pair $(p, \mbox{graph structure})$, we generate data and compute the sample correlation matrix, denoted by $S$.
All methods are then run on this same matrix, so that they are compared on the same instance of the optimization problem.

For every combination of graph structure, dimension, regularization parameters, algorithm, starting point, and tolerance, we record the runtime and the final attained objective value, denoted by $f_{\mathrm{end}}$.
Runtimes are aggregated across repetitions using the median.
All three compared algorithms are deterministic: for every fixed configuration, all $M = 100$ values of $f_{\mathrm{end}}$ were identical, so we report the common attained value.
This is expected, since once the sample correlation matrix $S$ is fixed, the optimization routine is deterministic.

Next, for each tuple $(p, \mbox{graph structure}, \lambda, \alpha)$, we define a benchmark value $f_{\mathrm{best}}$ as the best value obtained using a very strict stopping criterion.
This gives a common reference level for all compared methods.

A detailed description of the experimental setup, including all implementation details, is available in the online repository:
\url{https://github.com/PrzeChoj/pcglasso_article_code/blob/main/experiments/Appendix_A/plan.md}.

\subsection{Accuracy--time trade-off}
Figure~\ref{fig:comparison_type1_all} shows, for each graph structure, a representative accuracy-versus-time plot for the case $p = 200$, $\lambda = 0.1$, and $\alpha = 0$.
On these plots, the horizontal axis is the median runtime, while the vertical axis is the difference between the attained objective value and the benchmark value $f_{\mathrm{best}}$, shown on a logarithmic scale.
Smaller values on both axes are better, so points located lower and further to the left correspond to better performance.
These four plots are selected as representative examples; the complete collection of $4 \times 4 \times 2 \times 2 = 64$ plots, corresponding to all considered combinations of dimension $p$, graph structure, and regularization parameters $(\lambda,\alpha)$, is available in the online repository:
\url{https://github.com/PrzeChoj/pcglasso_article_code/tree/main/experiments/Appendix_A/plots/type_1}.

\begin{figure}
    \centering
    \includegraphics[width=\textwidth]{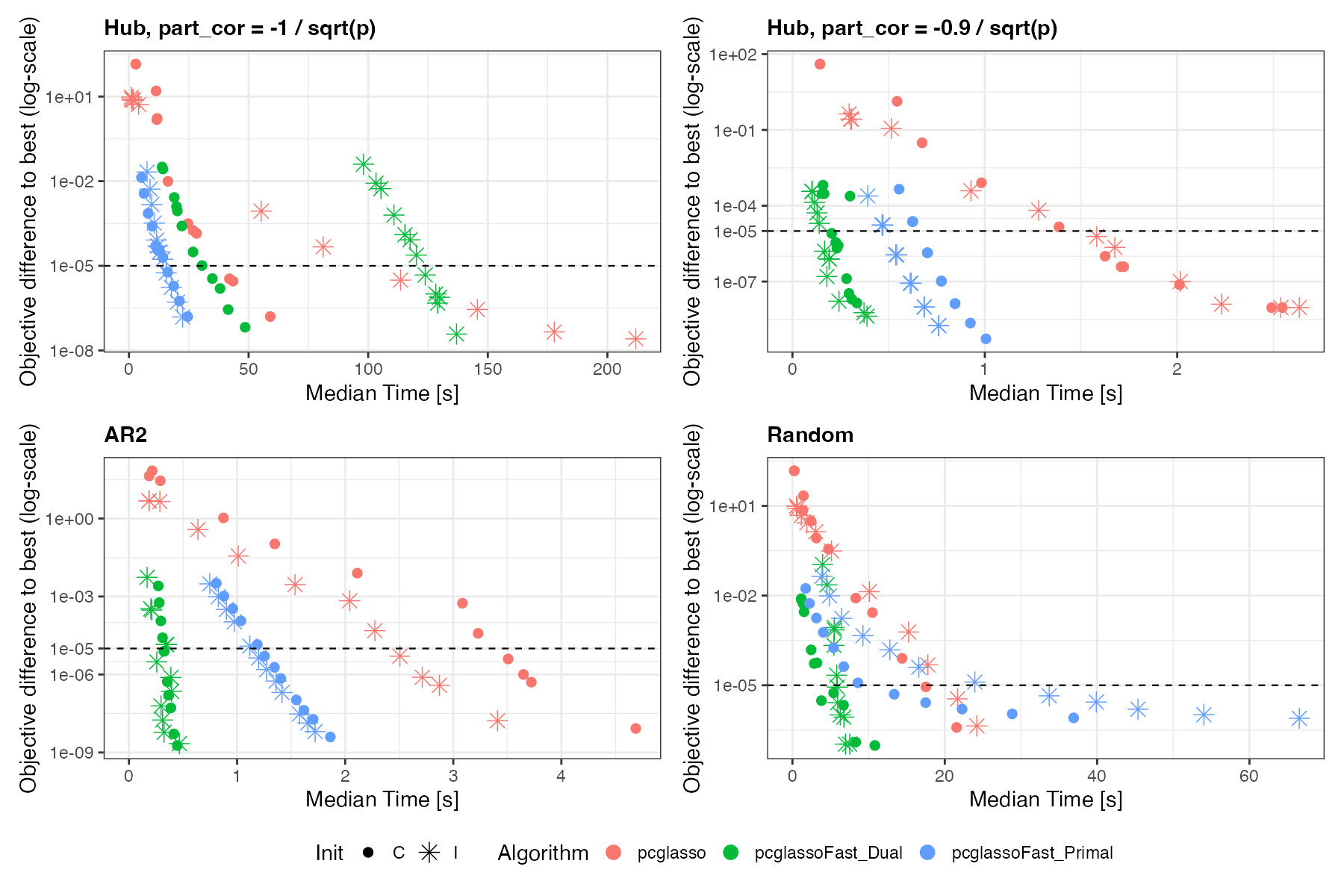}
    \caption{Accuracy-versus-time comparison for the four graph structures (\texttt{hub\_1}, \texttt{hub\_09}, \texttt{AR2}, \texttt{random}) for $p = 200$, $\lambda = 0.1$, and $\alpha = 0$.
    Each point corresponds to one combination of algorithm (\texttt{pcglasso}, \texttt{pcglassoFast\_Dual}, \texttt{pcglassoFast\_Primal}), starting point (\texttt{C} or \texttt{I}), and stopping tolerance.
    The horizontal axis shows the median runtime over $M = 100$ repetitions.
    The vertical axis shows the objective gap $f_{\mathrm{end}} - f_{\mathrm{best}}$ on a logarithmic scale.
    Colors distinguish algorithms, while point shapes distinguish starting points.
    The dashed horizontal line marks the accuracy threshold $10^{-5}$ used in the selection procedure for Figure~\ref{fig:comparison_type2_all}.
    Points closer to the lower-left corner indicate better performance.}
    \label{fig:comparison_type1_all}
\end{figure}

These plots reveal a clear ordering across several graph structures, in particular for \texttt{AR2} and \texttt{hub\_09}, where the method \texttt{pcglassoFast\_Dual} consistently reaches a given accuracy level fastest, followed by \texttt{pcglassoFast\_Primal}, while \texttt{pcglasso} is the slowest. In these representative examples, \texttt{pcglasso} can be roughly one order of magnitude slower than \texttt{pcglassoFast\_Dual}.
For \texttt{hub\_1}, \texttt{pcglassoFast\_Primal} appears to provide the best performance for both starting points.
The method \texttt{pcglassoFast\_Dual} with starting point \texttt{C} is slightly slower, while with starting point \texttt{I} it is noticeably slower.
For the \texttt{random} graph, no strong dominance is visible in these representative accuracy--time plots, and the relative performance varies across settings.

\subsection{Runtime comparison at matched accuracy}
To obtain a direct runtime comparison at matched optimization quality, we use the following procedure.
For each fixed tuple
$(p, \mbox{graph structure}, \lambda, \alpha, \mbox{solver}, \mbox{starting point})$,
we choose the loosest stopping criterion that still yields a solution within $10^{-5}$ of $f_{\mathrm{best}}$.
In other words, we select the largest tolerance such that
\[
f_{\mathrm{best}} \leq f_{\mathrm{end}} < f_{\mathrm{best}} + 10^{-5}.
\]
We then compare the runtimes obtained at these selected tolerances.

\begin{figure}
    \centering
    \includegraphics[width=\textwidth]{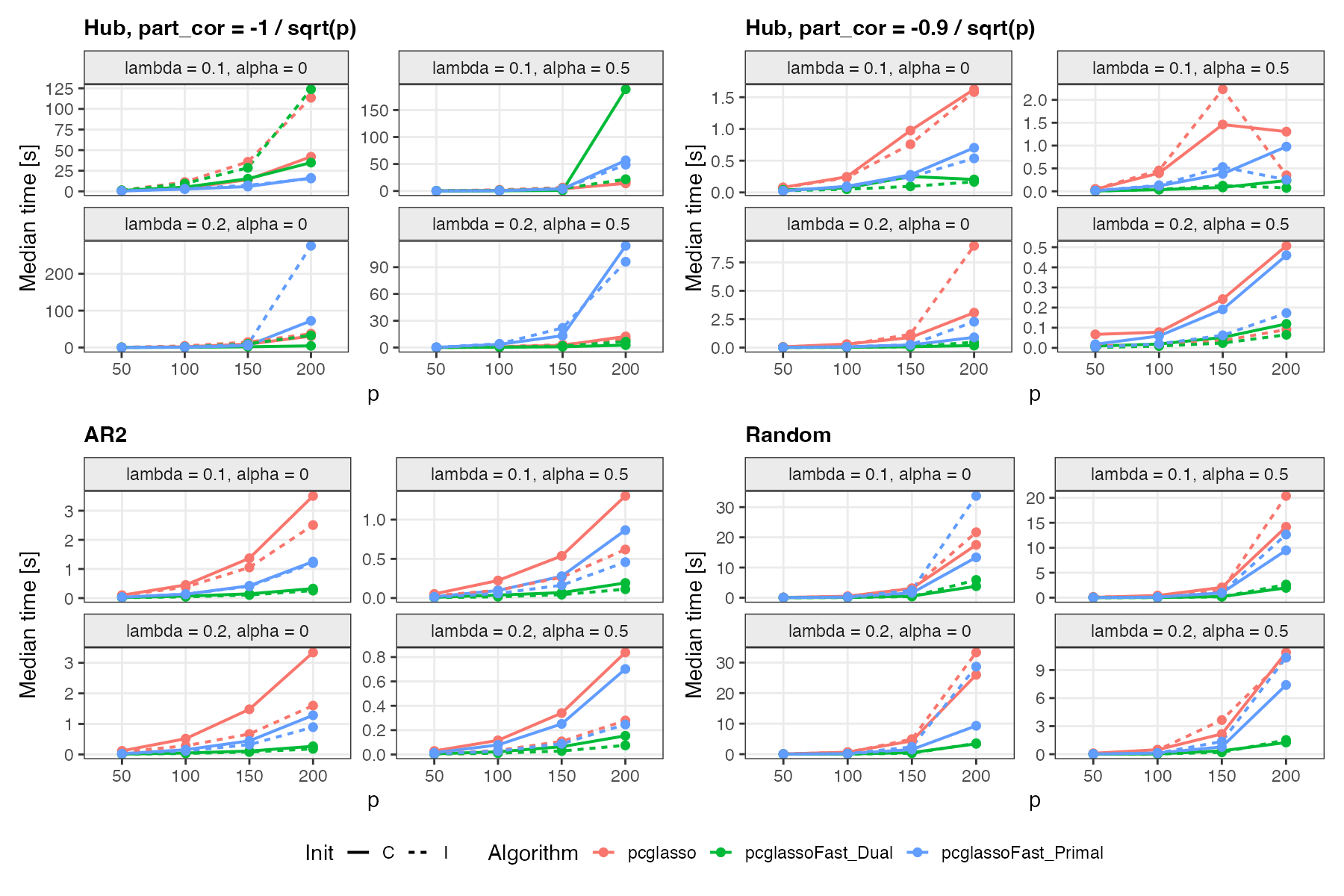}
    \caption{Median runtime as a function of $p$ after matching the optimization accuracy across methods.
    For each combination of graph structure, $(\lambda, \alpha)$, algorithm, and starting point, we select the largest stopping tolerance such that $f_{\mathrm{end}} - f_{\mathrm{best}} < 10^{-5}$.
    The reported runtime corresponds to this selected tolerance.
    Panels correspond to graph structures, and subplots within each panel correspond to $(\lambda, \alpha) \in \{0.1, 0.2\} \times \{0, 0.5\}$.
    Colors distinguish algorithms, and line types distinguish starting points.
    Smaller values indicate better computational efficiency at comparable objective accuracy.}
    \label{fig:comparison_type2_all}
\end{figure}

The resulting median runtimes as functions of $p$ are shown in Figure~\ref{fig:comparison_type2_all}.
For the graph structures \texttt{AR2}, \texttt{hub\_09}, and \texttt{random}, a consistent ordering is observed across most parameter settings and problem sizes:
\texttt{pcglassoFast\_Dual} is the fastest, \texttt{pcglassoFast\_\allowbreak Primal} is typically second, and \texttt{pcglasso} is the slowest.
The effect of initialization is not uniform: for some methods it leads to slight speedups (\texttt{AR2}), while for others it results in slightly longer runtimes (\texttt{random}).

The graph structure \texttt{hub\_1} behaves much less regularly, and no uniform ordering of the methods is visible across the four parameter settings.
For $(\lambda,\alpha)=(0.1,0)$, \texttt{pcglassoFast\_Primal} is the fastest.
For $(\lambda,\alpha)=(0.2,0)$ and $(0.2,0.5)$, however, \texttt{pcglassoFast\_Primal} becomes much slower, while \texttt{pcglassoFast\_Dual} is the fastest.
The case $(\lambda,\alpha)=(0.1,0.5)$ is particularly irregular: in this setting, \texttt{pcglasso} is the fastest for both initializations, while \texttt{pcglassoFast\_Dual} with initialization \texttt{I} performs comparably well.
At the same time, the same method with initialization \texttt{C} is noticeably the slowest.
Thus, for \texttt{hub\_1}, performance depends strongly on both the tuning parameters and the initialization.

\textbf{Limitations and implementation details.}
Firstly, the comparison is performed for individual optimization problems corresponding to fixed values of $(\lambda,\alpha)$.
In practice, many applications require solving a sequence of problems along a regularization path.
The relative performance of the compared methods can differ in such settings, and the conclusions drawn here may not necessarily transfer to path-wise optimization scenarios.

Secondly, this comparison is not exhaustive.
It is restricted to a finite grid of parameters and a limited collection of graph structures.
Extending the study to broader ranges of $(\lambda,\alpha)$, additional graph models, and larger dimensions may further refine the observed trends.

Finally, it should be noted that the compared implementations differ in their programming languages.
The method \texttt{pcglassoFast\_Dual} is implemented in Fortran, \texttt{pcglassoFast\_Primal} in \texttt{C++}, while the reference implementation \texttt{pcglasso} is written in \texttt{R}.
As a result, the observed differences reflect not only algorithmic efficiency but also the impact of implementation in compiled versus interpreted environments.

\subsection{Summary and practical recommendation}
Overall, the results indicate that \texttt{pcglassoFast\_Dual} provides the best computational efficiency across most graph structures and parameter settings, consistently achieving a given level of objective accuracy in the shortest time.
The method \texttt{pcglassoFast\_Primal} is typically the second best, while the reference implementation \texttt{pcglasso} is generally slower.
An exception is the \texttt{hub\_1} structure, where the relative performance depends strongly on $(\lambda,\alpha)$ and the choice of initialization, and no single method dominates uniformly.

The choice of initialization has a moderate and method-dependent impact, but does not alter the overall ranking in most cases.
In particular, the starting point \texttt{C} tends to provide more stable performance across settings.

Based on these findings, we adopt \texttt{pcglassoFast\_Dual} with initialization \texttt{C} as the default configuration in the \texttt{pcglassoFast} package, as it offers the best trade-off between speed and robustness across a wide range of scenarios.

\section{Applied examples and additional simulation results}\label{sec:appendix_applied}

\subsection{Prostate cancer: additional comparisons}\label{sec:prostateCont}
We present supplementary analyses for the prostate cancer RNA-seq data of \citet{HUB25} (see Section~\ref{sec:real_ex2}). Figure~\ref{fig:pathcancer} shows $\mathrm{EBIC}(\gamma=0.5)$ along the solution path for GLASSO, Cor-GLASSO, and PCGLASSO. The SPACE procedure was numerically unstable on these data and did not yield positive-definite precision-matrix estimates. For visual comparison, we display in Figure \ref{fig:adjacency_pcglasso}, for each method, the selected adjacency matrix at the EBIC-minimizing tuning parameter, together with a thresholded empirical partial-correlation graph matched to the same number of edges. Among the methods, the PCGLASSO estimator yields the most structured network and reveals several hub nodes in these data.

Finally, in Figure~\ref{fig:rows_cancer} we plot the rows of the empirical precision matrix corresponding to the largest values of \(D(\hat K)\) (see the bottom-right panel of Figure~\ref{fig:cancer_panels}). These rows show no evidence of hub-like structure; rather, they exhibit a small number of strong pairwise connections to specific genes within the group.

\begin{figure}[h]
  \centering
  \includegraphics[width=.6\linewidth]{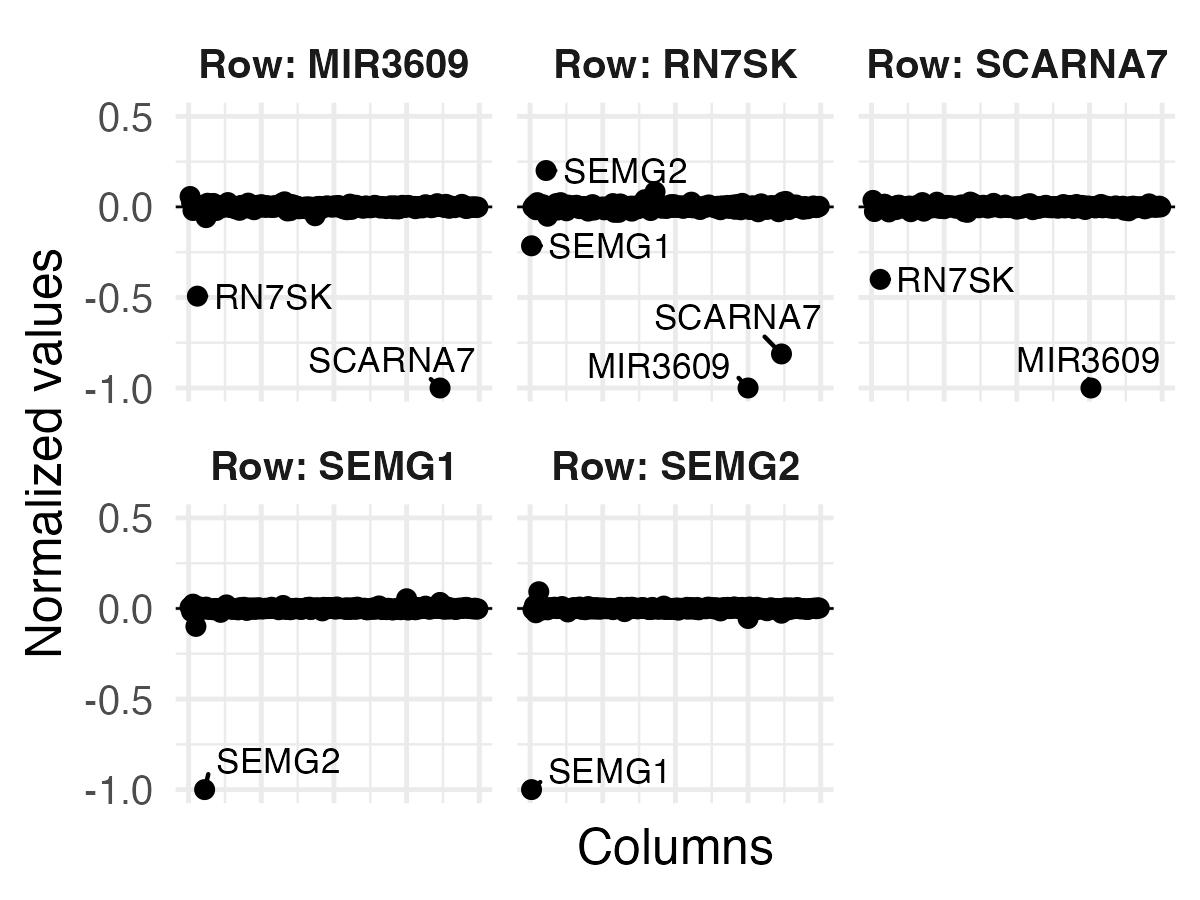}
  \caption{Rows of the empirical precision matrix (the inverse of the empirical covariance) for genes
  \(151=\texttt{MIR3609}\), \(174=\texttt{SCARNA7}\), \(1=\texttt{SEMG1}\), \(12=\texttt{SEMG2}\), and \(6=\texttt{RN7SK}\).
  For each row we removed the diagonal entry and standardized by dividing all entries by the maximum absolute value in that row (so the largest magnitude equals 1).}
  \label{fig:rows_cancer}
\end{figure}

\begin{figure}[htbp]
  \centering

  \begin{subfigure}{0.48\linewidth}
    \centering
    \includegraphics[width=\linewidth]{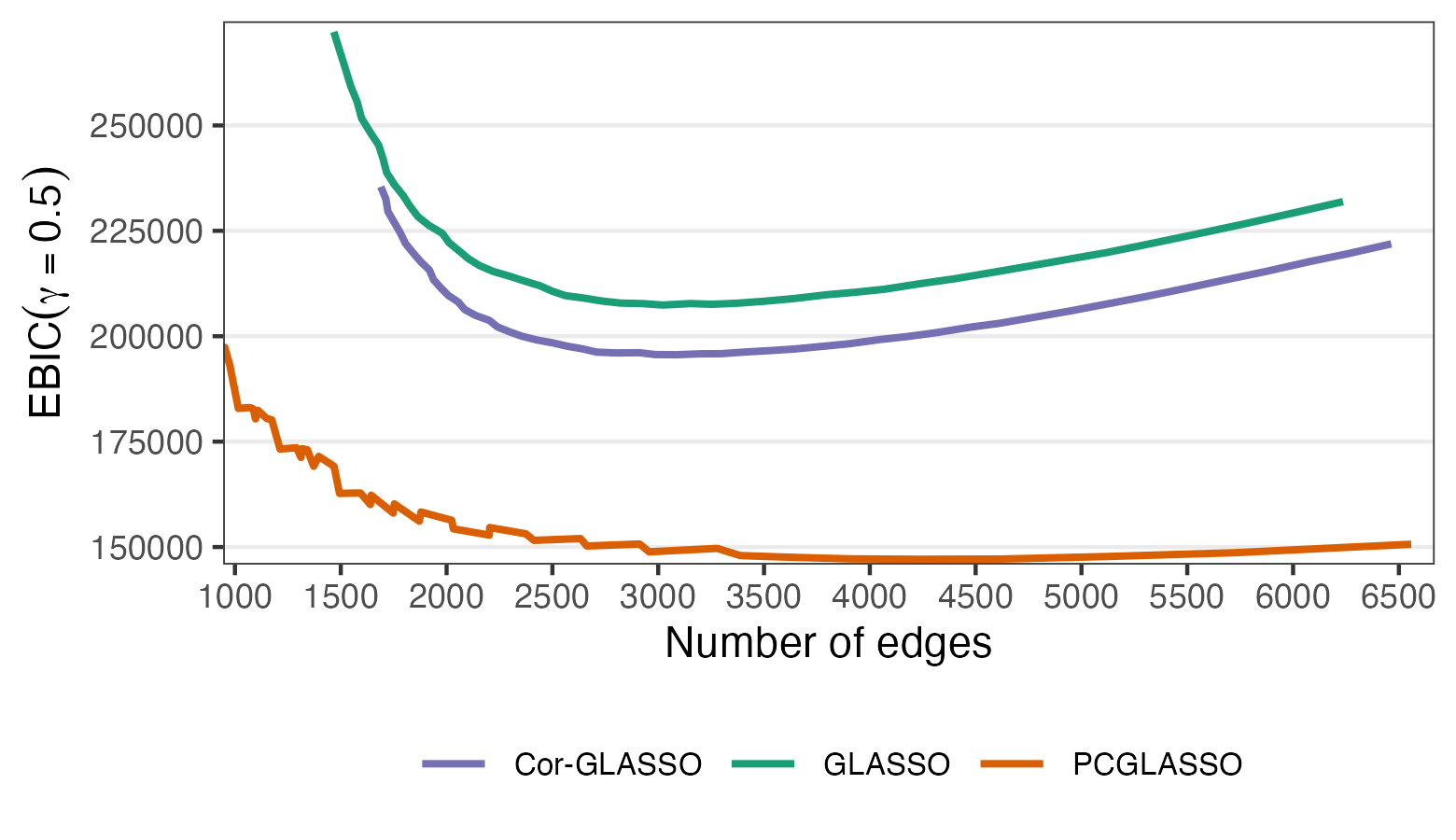}
    \caption{EBIC with $\gamma=0.5$ versus number of edges for GLASSO, Cor-GLASSO, and PCGLASSO; PCGLASSO attains the minimum.}
    \label{fig:pathcancer}
  \end{subfigure}
  \hfill
  \begin{subfigure}{0.48\linewidth}
    \centering
    \includegraphics[width=\linewidth]{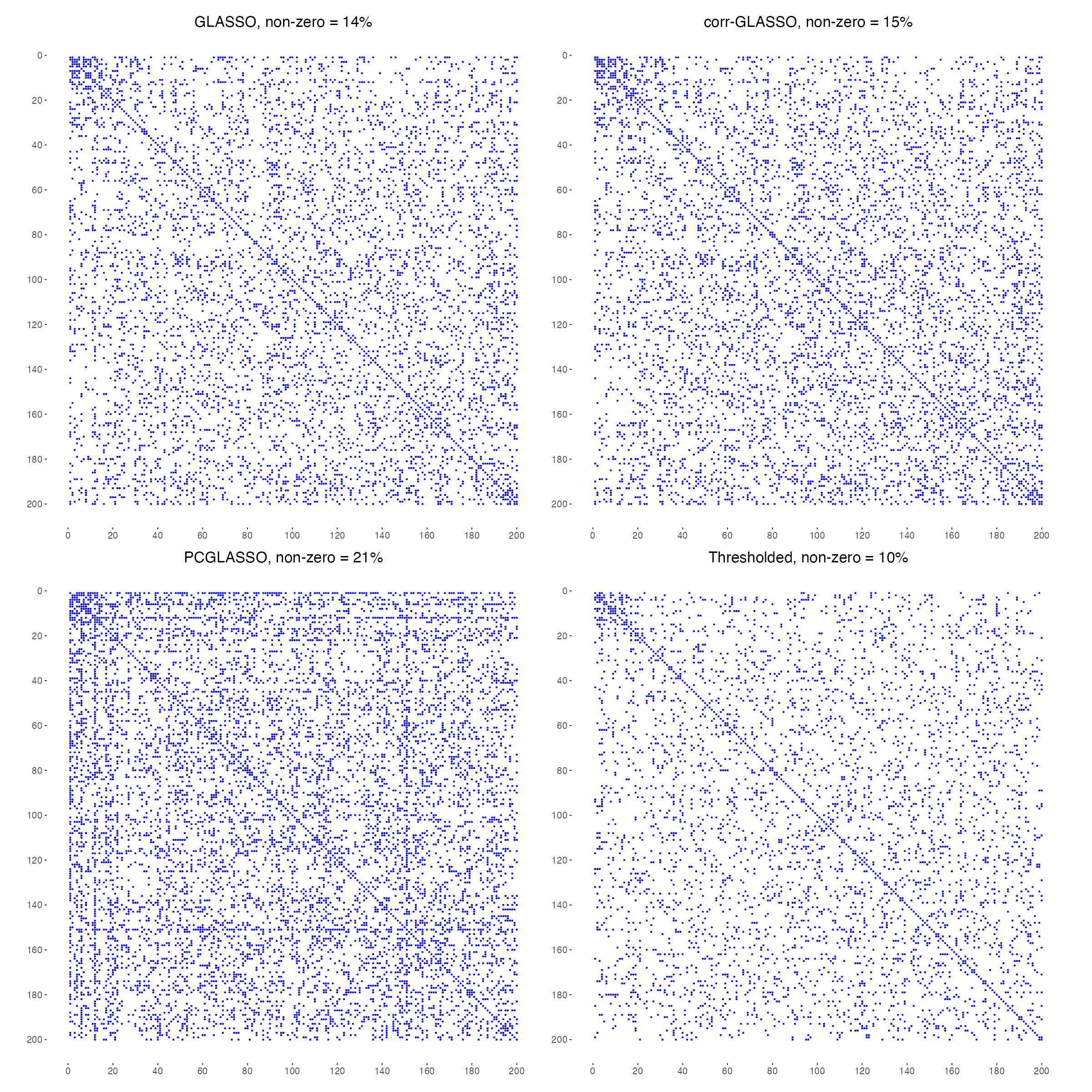}
    \caption{Selected adjacency at the EBIC minimum for GLASSO, Cor-GLASSO, PCGLASSO, and the thresholded empirical partial correlation.}
    \label{fig:adjacency_pcglasso}
  \end{subfigure}

  \caption{Cancer data analysis: model selection by EBIC and the corresponding selected adjacency matrices.}
  \label{fig:cancer_combined}
\end{figure}

\subsection{Simulation study: non-hub structure}\label{sec:simulations_nonhub_app}
For completeness, we also consider a simulation setting based on a covariance structure estimated from the same gene dataset. In this case, $\Sigma$ is an estimate from the GLASSO path, Figure~\ref{fig:nonhub_structure} displays the corresponding nonzero pattern.

\begin{figure}[htbp]
    \centering
    \includegraphics[width=0.7\linewidth]{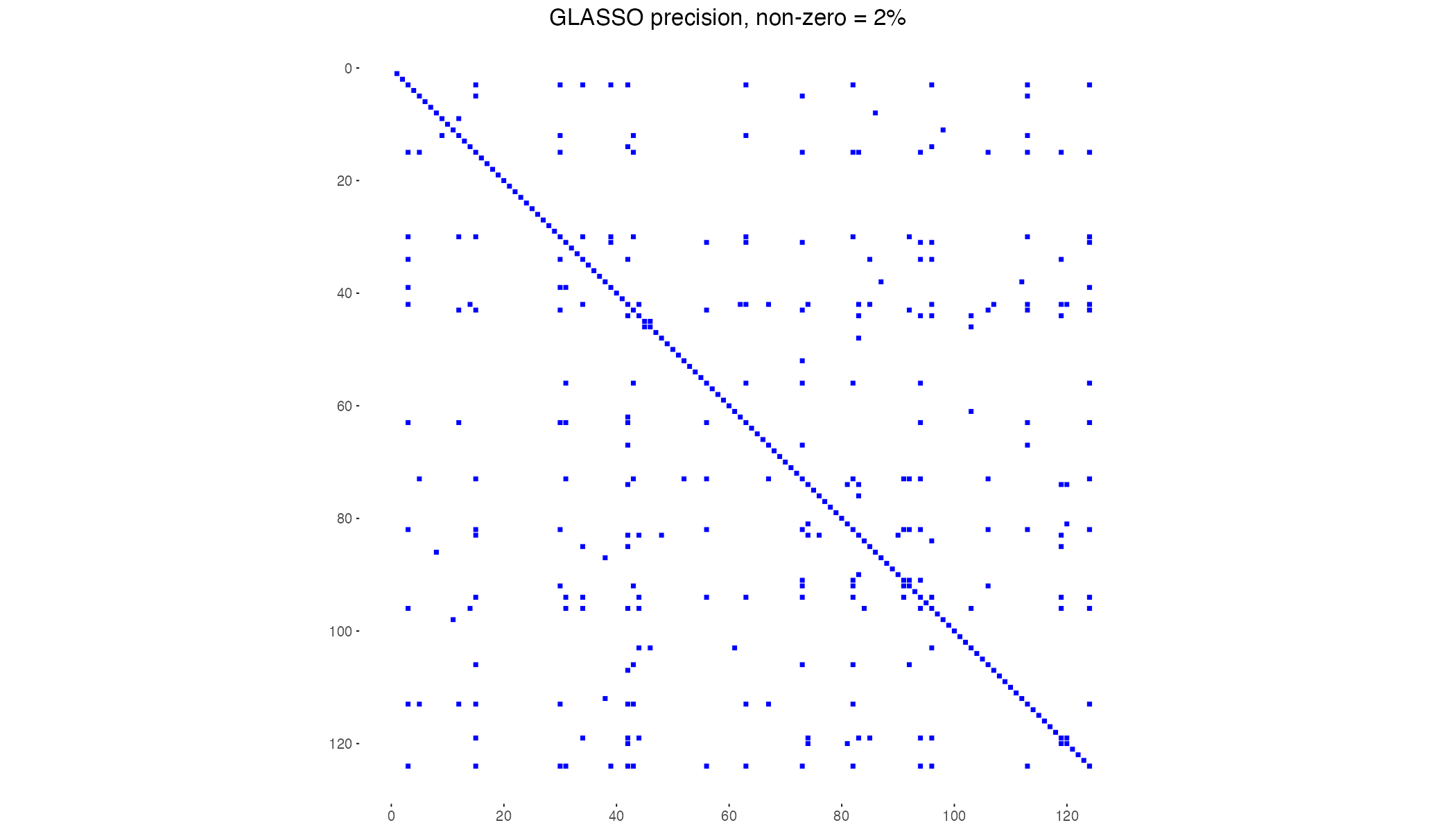}
    \caption{Nonzero pattern for the GLASSO generated precision matrix used in the supplementary simulation study.}
    \label{fig:nonhub_structure}
\end{figure}

We generate independent samples
$X_i \sim \mathcal{N}_p(0,\Sigma)$ where $i=1,\ldots,n,$
with sample sizes $n = 200$, $500$, $1\,000$, and $5\,000$. As in the main simulation study, we compare GLASSO, Cor-GLASSO, SPACE, and PCGLASSO, with hyperparameters selected either by BIC or by a single validation split (Val). For each configuration, we compute RMSE for the full matrix, the diagonal elements, and the nonzero off-diagonal elements, and repeat each experiment $200$ times.

This supplementary experiment assesses whether the methods perform well when the covariance structure is not generated by PCGLASSO. The results are summarized in Tables~\ref{tab:nonhubrmse} and~\ref{tab:nonhubtime}. In this setting, the methods perform more similarly overall, although SPACE often attains the lowest RMSE values for the diagonal entries. PCGLASSO performs best even though the covariance is taken from a regularization path produced by GLASSO. As in the PCGLASSO setting, SPACE is substantially slower than the competing methods.

\begin{longtable}[]{@{}lcccc@{}}
\caption{Computation time (seconds) for each method and sample size.\label{tab:nonhubtime}}\\
\toprule
Method & $n=200$ & $n=500$ & $n=1,000$ & $n=5,000$ \\
\midrule
\endfirsthead

\multicolumn{5}{c}{{\tablename\ \thetable{} -- continued from previous page}} \\
\toprule
Method & $n=200$ & $n=500$ & $n=1,000$ & $n=5,000$ \\
\midrule
\endhead

\midrule \multicolumn{5}{r}{{Continued on next page}} \\
\endfoot

\bottomrule
\endlastfoot

CGL BIC & 2.18 & 1.59 & 1.36 & 0.95 \\
CGL Val & 2.94 & 1.79 & 1.49 & 1.00 \\
GL BIC & 0.96 & 0.82 & 0.79 & 0.58 \\
GL Val & 1.22 & 0.91 & 0.86 & 0.61 \\
pcglassoFast\_Primal BIC & 1.80 & 1.49 & 1.24 & 0.93 \\
pcglassoFast\_Primal Val & 1.91 & 1.50 & 1.25 & 0.90 \\
pcglassoFast\_Dual BIC & 0.86 & 0.69 & 0.58 & \textbf{0.42} \\
pcglassoFast\_Dual Val & \textbf{0.85} & \textbf{0.64} & \textbf{0.55} & 0.44 \\
SPACE BIC & 5.15 & 10.20 & 18.50 & 84.90 \\
SPACE Val & 3.71 & 7.33 & 12.60 & 54.00 \\

\end{longtable}

\begin{longtable}[]{@{}llcccc@{}}
\caption{RMSE summary for each method and sample size.\label{tab:nonhubrmse}}\\
\toprule\noalign{}
Metric & Method & $n=200$ & $n=500$ & $n=1,000$ & $n=5,000$ \\
\midrule\noalign{}
\endfirsthead

\multicolumn{6}{c}{{\tablename\ \thetable{} -- continued from previous page}} \\
\toprule\noalign{}
Metric & Method & $n=200$ & $n=500$ & $n=1,000$ & $n=5,000$ \\
\midrule\noalign{}
\endhead

\midrule \multicolumn{6}{r}{{Continued on next page}} \\
\endfoot

\bottomrule
\endlastfoot

RMSE & CGL BIC & 0.20 & 0.16 & 0.13 & 0.07 \\
 & CGL Val & 0.19 & 0.14 & 0.11 & 0.05 \\
 & GL BIC & 0.21 & 0.20 & 0.20 & 0.16 \\
 & GL Val & 0.22 & 0.19 & 0.17 & 0.09 \\
 & pcglassoFast\_Primal BIC & 0.19 & 0.15 & 0.12 & 0.06 \\
 & pcglassoFast\_Primal Val & \textbf{0.18} & \textbf{0.13} & \textbf{0.10} & \textbf{0.05} \\
 & pcglassoFast\_Dual BIC & 0.19 & 0.15 & 0.12 & 0.06 \\
 & pcglassoFast\_Dual Val & \textbf{0.18} & \textbf{0.13} & \textbf{0.10} & \textbf{0.05} \\
 & SPACE BIC & \textbf{0.18} & 0.14 & 0.11 & 0.05 \\
 & SPACE Val & \textbf{0.18} & \textbf{0.13} & 0.10 & 0.05 \\

\midrule\noalign{}
Diag RMSE & CGL BIC & 1.27 & 0.97 & 0.78 & 0.40 \\
 & CGL Val & 1.35 & 0.92 & 0.69 & 0.33 \\
 & GL BIC & 1.29 & 1.08 & 1.00 & 0.80 \\
 & GL Val & 1.44 & 1.11 & 0.94 & 0.47 \\
 & pcglassoFast\_Primal BIC & 1.20 & 0.84 & 0.62 & 0.25 \\
 & pcglassoFast\_Primal Val & 1.27 & 0.79 & 0.56 & 0.25 \\
 & pcglassoFast\_Dual BIC & 1.20 & 0.84 & 0.61 & 0.25 \\
 & pcglassoFast\_Dual Val & 1.27 & 0.79 & 0.56 & 0.25 \\
 & SPACE BIC & \textbf{1.12} & \textbf{0.72} & \textbf{0.50} & \textbf{0.20} \\
 & SPACE Val & 1.27 & 0.79 & 0.55 & 0.24 \\

\midrule\noalign{}
Off-diag (NZ) & CGL BIC & 1.25 & 1.07 & 0.86 & 0.42 \\
RMSE & CGL Val & 1.07 & 0.82 & 0.65 & 0.33 \\
 & GL BIC & 1.35 & 1.35 & 1.34 & 1.12 \\
 & GL Val & 1.35 & 1.27 & 1.15 & 0.56 \\
 & pcglassoFast\_Primal BIC & 1.21 & 1.04 & 0.83 & 0.40 \\
 & pcglassoFast\_Primal Val & \textbf{1.06} & \textbf{0.80} & \textbf{0.62} & \textbf{0.32} \\
 & pcglassoFast\_Dual BIC & 1.21 & 1.04 & 0.83 & 0.40 \\
 & pcglassoFast\_Dual Val & \textbf{1.06} & \textbf{0.80} & \textbf{0.62} & \textbf{0.32} \\
 & SPACE BIC & 1.16 & 0.97 & 0.79 & 0.39 \\
 & SPACE Val & 1.10 & 0.84 & 0.66 & 0.33 \\

\end{longtable}

\section{Algorithms}\label{appendix:algs}

\subsection{Diagonal Newton Method for \texorpdfstring{$D$}{D} optimization}\label{sec:optD}

We seek a diagonal scaling matrix $D=\mathrm{diag}(d)\succ 0$ that satisfies \eqref{eq:scaling}.
Writing $d=(D_{11},\ldots,D_{pp})^\top\in(0,\infty)^p$, this is equivalent to the strictly convex optimization problem
\[
d^\ast \in \arg\min_{d\in(0,\infty)^p}
\left\{ f(d):=\frac12\, d^\top A d - \sum_{i=1}^p \log d_i \right\}.
\]
 We apply a diagonal Newton step with a backtracking
line search satisfying the Wolfe conditions \cite[Algorithm~3.5]{Nocedal2006}; see Algorithm~\ref{algo:diag_hess}. In Appendix~\ref{appendix_D_optim} we provide proof of convergence and the justification on using the diagonal approximation.

\begin{algorithm}[H]
\caption{Diagonal Newton Method for $D$ Optimization}
\label{algo:diag_hess}
\begin{algorithmic}[1]
\Require $A$: a $p\times p$ symmetric matrix, $k$: maximum number of iterations, $\eta_{\min}$: minimum step-size, $\mathrm{tol}$: objective-drop tolerance
\Ensure $d^\ast$ (approximation) \Procedure{DiagNewtonD}{$A$, $k$, $\eta_{\min}$, $\mathrm{tol}$}
\State Initialize $d \in \mathbb{R}^p_+$, $f_{\text{old}} \gets \infty$
\For{iter $= 1,\ldots,k$}
    \State $g \gets Ad - d^{-1}$ \Comment{Gradient, element-wise inverse}
    \State $h \gets a+d^{-2}$ \Comment{Hessian diagonal, $a=(A_{ii})_i$}
    \State $\Delta \gets g \oslash h$ \Comment{Element-wise division}
    \State Define $\phi(\eta) = f(d - \eta\Delta)$ for $\eta\in [0,\infty)$
    \State $\eta^* \gets \textbf{LineSearch}(\phi,\eta_{\min})$ 
    \State $d \gets d - \eta^*\Delta$
    \State $f_{\text{new}} \gets f(d)$
    \State $f_\delta \gets f_{\text{old}} - f_{\text{new}}$
    \State $f_{\text{old}} \gets f_{\text{new}}$
    \If{$f_\delta < \mathrm{tol}$} \Comment{early-exit test}
        \State \textbf{break} \Comment{tolerance satisfied}
    \EndIf
\EndFor
    \State \Return $d$
\EndProcedure
\end{algorithmic}
\end{algorithm}

\subsection{Coordinate descent algorithm for \texorpdfstring{$R$}{R} optimization - primal problem}\label{sec:optRA}

We solve the constrained optimization problem
\begin{align}\label{eq:r_problem}
\hat{R} \in \argmax_{R\in\Sppone}
\Big\{ \log\det(R) - \tr(RS) - \lambda\|R\|_{1,\mathrm{off}} \Big\},
\end{align}
where $S$ is a positive semidefinite matrix.

Fix an index $i\in\{1,\dots,p\}$ and permute coordinates (if needed) so that $i$ corresponds to the first coordinate.
Partition
\[
R=\begin{pmatrix}1 & r^\top\\ r & R_{11}\end{pmatrix},\qquad
S=\begin{pmatrix}S_{ii} & s^\top\\ s & S_{11}\end{pmatrix},
\qquad s:=S_{-i,i},
\]
where $r\in\mathbb{R}^{p-1}$ collects the off-diagonal entries of column $i$,
and $R_{11}=R_{-i,-i}\in\mathbb{S}^{p-1}_{++}$ is the principal submatrix obtained by removing row/column $i$.
By the block determinant identity,
\[
\det(R)=\det(R_{11})\bigl(1-r^\top R_{11}^{-1}r\bigr),
\]
and by symmetry,
\[
\tr(RS)=\tr(R_{11}S_{11})+S_{ii}+2s^\top r,\qquad
\|R\|_{1,\mathrm{off}}=\|R_{11}\|_{1,\mathrm{off}}+2\|r\|_1.
\]
Hence, for fixed $R_{11}$, maximizing \eqref{eq:r_problem} over $r$ is equivalent (up to an additive constant
and multiplication by $1/2$) to maximizing
\begin{equation}\label{eq:column_objective}
\ell(r)=\frac{1}{2}\log\bigl(1-r^\top R_{11}^{-1}r\bigr)-s^\top r-\lambda\|r\|_1+\text{const.},
\end{equation}
over the feasible set $\{r\in\R^{p-1}\colon 1-r^\top R_{11}^{-1}r>0\}$.

Let $Q:=R_{11}^{-1}$, fix $r_{-j}$ and update the single coordinate $r_j$.
Write the quadratic form as
\[
r^\top Q r = a_j r_j^2 + 2 b_j r_j + c_j,
\]
where
\[
a_j:=Q_{jj},\quad
b_j:=\sum_{k\neq j}Q_{jk}r_k = (Qr)_j - Q_{jj}r_j,\quad
c_j:=\sum_{k\neq j}\sum_{\ell\neq j}Q_{k\ell}r_k r_\ell
      = r^\top Q r - a_j r_j^2 - 2 b_j r_j.
\]
Note that $a_j>0$ and that $b_j,c_j$ depend only on the fixed vector $r_{-j}$ (hence are constants in the
one-dimensional update). Dropping constants, \eqref{eq:column_objective} gives the
scalar optimization problem
\begin{equation}\label{eq:rj}
\hat{r}_j \in \arg\max_{r_j\in K_j}
\left\{
\ell(r_j)=\frac12\log\left(1-a_j r_j^2-2b_j r_j-c_j\right)
- s_j r_j -\lambda |r_j|
\right\},
\end{equation}
where $K_j=\{r_j\colon 1-a_j r_j^2-2b_j r_j-c_j>0\}$ is the feasible interval induced by
positive definiteness.

The explicit solution is given by the following Theorem.
\begin{thm}
\label{thm:element_update}
Let $b,c,s\in\mathbb{R}$ and $a,\lambda>0$. Assume that $c<1+b^2/a$ and let 
\[
K = \{r\colon 1-a r^2-2br-c>0\}.
\] The solution for 
\begin{equation}
\label{eq:rj2}
\hat{r}= \argmax_{r\in K} \left\{ \ell(r) = \frac{1}{2}\log\left(1 - a r^2 - 2b r - c \right) - s r - \lambda |r| \right\},
\end{equation}
equals
\begin{equation}
\label{explicit solution}
   \hat{r}=\begin{cases}
        0 & \hspace{0.3cm} \text{ if } \vert\xi\vert\leq \lambda \mbox{ and }c<1,\\
        -\frac{b}{a}&\hspace{0.3cm}\text{ else if } \zeta=0,
        \\
          -\frac{\tilde{b}}{2\tilde{a}} +\sign(\tilde{a}) \sqrt{(\tilde{b}/2\tilde{a})^2 - (\tilde{c}/\tilde{a})} & \hspace{0.3cm} \text{ else} ,
    \end{cases}
\end{equation}
where 
\begin{equation*}
    \xi = \frac{-b}{1-c} - s, \quad \zeta=s+\lambda_s,\quad\lambda_s = 
    \begin{cases}
        \sign(\xi)\lambda, & \text{ if } c<1\\
        \sign(-b)\lambda, & \text{ if } c\geq1.
    \end{cases} 
\end{equation*}
and the coefficients $\tilde{a},\tilde{b},\tilde{c}$ given by:
\begin{align*}
    \tilde{a} = -\zeta a,\quad \tilde{b} = a - 2\zeta b, \quad    \tilde{c} = \zeta(1-c) + b.
\end{align*}
\end{thm}
\begin{proof}
First, note that
\begin{equation}\label{admissible domain}
K=\left( -\frac{b}{a} - \sqrt{\left(\frac{b}{a}\right)^2 + \frac{1-c}{a}},\ 
          -\frac{b}{a} + \sqrt{\left(\frac{b}{a}\right)^2 + \frac{1-c}{a}} \right).
\end{equation}
By the assumption $c<1+b^2/a$, this interval is non-empty. 

The function $\ell$ is strictly concave on $K$ (as a sum of a strictly concave log-term,
a linear term, and a concave penalty), and $\ell(r)\to -\infty$ as $r$ approaches the boundary of $K$. Hence the maximizer $\hat r$ exists and is unique.

We now determine the sign of the solution $\hat{r}$. In case $\hat{r}\neq 0$, we observe that due to the symmetry of the penalty term $-\lambda\vert r\vert$, it follows that $\sign(\hat{r})=\sign(\bar{r})$, where $\bar{r}$ maximizes the smooth unpenalized problem 
\begin{equation*}
    \ell_0(r)=\frac{1}{2}\log\left(1- a r^2 - 2b r -c  \right) - s r,
\end{equation*}
which is given by setting $\lambda=0$ in (\ref{eq:rj2}). We distinguish between two cases. If $c\geq 1$, the domain of $\ell_0(r)$ does not include $0$, and $\sign(\bar{r})=\sign(-b)$.  If $c<1$, the domain of $\ell_0(r)$ includes $0$, and the sign of the maximum $\bar{r}$ is determined by the derivative of $\ell_0$ at zero. The derivative is
\begin{equation*}
    \ell_0'(r)=\frac{ -ar -b}{1-a r^2 - 2b r -c} - s, \hspace{0.8cm} \ell_0'(0)=\frac{-b}{1-c}-s=\xi,
\end{equation*}
and we obtain that $\sign(\bar{r})=\sign(\xi)$. Moreover, for the original penalized problem, we observe that $\hat{r}=0$ whenever $c<1$ and $\vert \ell_0'(0)\vert=\vert \xi\vert\leq \lambda$. 
In summary,
\begin{equation}\label{eq:signs}
    \sign(\hat{r}) = 
    \begin{cases}
        0, & \text{ if } c<1 \text{ and } |\xi|\leq\lambda,\\
        \sign(\xi), & \text{ if } c<1 \text{ and } |\xi|>\lambda,\\
        \sign(-b), & \text{ if } c\geq1.
    \end{cases} 
\end{equation}

We now turn to the explicit solution when $\hat{r}\neq 0$. The optimum $\hat{r}$ solves
\begin{equation*}
\frac{ a\hat{r} +b}{1-a \hat{r}^2 - 2b \hat{r} -c} + s + \lambda_s=0,
\end{equation*}
where $\lambda_s=  \sign(\hat{r}) \lambda$ is determined by (\ref{eq:signs}). If $\zeta=(s + \lambda_s)=0$, then $a\hat{r}+b=0$, hence $\hat{r}=-b/a$. From now on, we assume that $\zeta=(s + \lambda_s)\neq 0$. Multiplying the above equation by the denominator and collecting the terms yields the quadratic equation
$$
\tilde{a} \hat{r}^2 + \tilde{b} \hat{r} + \tilde{c} = 0,
$$
where $\tilde{a}  = -\zeta a, \tilde{b} = a -2\zeta b, $ and $\tilde{c} =\zeta(1-c)+b $. The roots of the quadratic are
\begin{equation}\label{quadratic}
\hat{r}_{\pm} = -\frac{\tilde{b}}{2\tilde{a}} \pm\sqrt{\left(\frac{\tilde{b}}{2\tilde{a}}\right)^2 - \left(\frac{\tilde{c}}{\tilde{a}}\right)}.
\end{equation}

It remains to argue why $\sign(\tilde{a})$ specifies the correct candidate root in (\ref{quadratic}). 
By expanding (\ref{quadratic}) explicitly and simplifying, we obtain
    \begin{equation*}
        \hat{r}_{\pm}=-\dfrac{b}{a}+\dfrac{1}{2\zeta}\pm\sqrt{\left(\dfrac{b}{a}\right)^2+\dfrac{1-c}{a}+\dfrac{1}{4\zeta^2}}.
    \end{equation*}
We observe that the half-distance between the roots exceeds the half-length of the admissible domain
\eqref{admissible domain}; hence, at most one root can lie in $K$. Moreover, comparing the midpoint of the roots at $-(b/a)+1/(2\zeta)$ with the center of the admissible domain at $-(b/a)$, we see that the correct root is $\hat{r}=\hat{r}_+$ whenever $\zeta<0$ and $\hat{r}=\hat{r}_-$ when $\zeta>0$. Therefore, the root is determined by $\sign(-\zeta)=\sign(\tilde{a})$, which proves (\ref{explicit solution}). 
\end{proof}

\begin{algorithm}[H]
 \caption{Element-wise Coordinate Descent for Updating the Column Vector $r$}\label{alg:updater}
\begin{algorithmic}[1]
\Require $r\in\R^{p-1}$, $Q=R^{-1}_{11}\in\mathbb S^{p-1}_{++}$, $s\in\R^{p-1}$, $\lambda\in[0,\infty)$, $\tau$: convergence threshold
\Ensure Updated $r$ solving \eqref{eq:rj} approximately
\Procedure{UpdateColumn}{$r$, $Q$, $s$, $\lambda$, $\tau$}
    \Repeat
        \State $B \gets Qr,\quad c_0 \gets r^\top B$
        \State $\ell_{\mathrm{old}} \gets \frac12\log(1-c_0) - s^\top r - 2\lambda\|r\|_1$
        \For{$j=1,\ldots,p-1$}
            \State $a \gets Q_{jj},\quad b \gets B_j-a r_j,\quad c \gets c_0-a r_j^2-2 b r_j$
            \State $r^{\mathrm{new}}\gets \textsc{ElemUpdate}(a,b,c,s_j,\lambda)$ \Comment{Theorem~\ref{thm:element_update}}
            \State $\delta \gets r^{\mathrm{new}}-r_j$
            \If{$\delta\neq 0$}
                \State $r_j \gets r^{\mathrm{new}}$
                \State $c_0 \gets c_0 + 2\delta B_j + \delta^2 a$
                \State $B \gets B + \delta \, Q_{\cdot j}$
            \EndIf
        \EndFor
        \State $\ell_{\mathrm{new}} \gets \frac12\log(1-c_0) - s^\top r - 2\lambda\|r\|_1$
        \State $\Delta \ell \gets \ell_{\mathrm{new}}-\ell_{\mathrm{old}}$
    \Until{$\Delta \ell \le \tau$}
    \State \Return $r$
\EndProcedure
\end{algorithmic}
\end{algorithm}

\begin{algorithm}[H]
\caption{Coordinate Descent Algorithm for solving \eqref{eq:r_problem}}\label{alg:updateR}
\begin{algorithmic}[1]
\Require $R\in\Sppone$ and $W=R^{-1}$: initial iterates (warm start), $S$: a positive semidefinite matrix, $\lambda\in[0,\infty)$: tuning parameter, $\tau_{inner},\tau_{outer}$: convergence threshold
\Ensure Optimal $R$ and $W=R^{-1}$ from \eqref{eq:r_problem} 
\Procedure{pcglassoFast\_Primal}{$R$, $W$, $S$, $\lambda$, $\tau_{inner}$,$\tau_{outer}$}
    \State $\mathrm{obj} \gets \log\det(R) - \tr(RS) - \lambda \| R\|_{1,\mathrm{off}}$
    \Repeat
        \State $\mathrm{obj}_{\mathrm{old}} \gets \mathrm{obj}$
        \For{$i = 1$ \textbf{to} $p$}
            \State $r_{\mathrm{old}} \gets R_{-i,i}$ \Comment{current off-diagonal column $i$ (length $p-1$)}
            \State $Q \gets W_{-i,-i} - \frac{1}{W_{ii}}\, W_{-i,i} W_{i,-i}$ \Comment{$Q = R^{-1}_{11}$}
            \State $\mathrm{obj} \gets \mathrm{obj} + 2\, r_{\mathrm{old}}^\top S_{-i,i} + 2\lambda \|r_{\mathrm{old}}\|_1$
            \Comment{remove old column-$i$ terms}
            \State $r \gets {\textsc{UpdateColumn}}(r_{\mathrm{old}}, Q, S_{-i,i}, \lambda, \tau_{inner})$
            \State{$c_{\mathrm{old}}\gets r_{\mathrm{old}}^\top Q r_{\mathrm{old}},\quad c_{\mathrm{new}}\gets r^\top Qr$}
            \State $\mathrm{obj} \gets \mathrm{obj} - \log(1-c_{\mathrm{old}}) + \log(1-c_{\mathrm{new}})
                                    - 2\, r^\top S_{-i,i} - 2\lambda \|r\|_1$
            \Comment{update $\log\det$ and trace/penalty contributions}

            \State $R_{-i,i} \gets r,\quad R_{i,-i} \gets r^\top$ \Comment{enforce symmetry}

            \State $\beta \gets Q r$
            \State $\mathrm{Schur} \gets (1 - r^\top\beta)^{-1}$ \Comment{$\mathrm{Schur}=W_{ii}$}
            \State $W_{-i,-i} \gets Q + \mathrm{Schur}\cdot \beta\beta^\top$
            \State $W_{-i,i} \gets -\mathrm{Schur}\cdot \beta$
            \State $W_{i,-i} \gets W_{-i,i}^\top$
            \State $W_{ii} \gets \mathrm{Schur}$
        \EndFor
    \Until $|\mathrm{obj}-\mathrm{obj}_{\mathrm{old}}| < \tau_{outer}$
    \State \Return $(R,W,\mathrm{obj})$
\EndProcedure
\end{algorithmic}
\end{algorithm}

\subsection{Coordinate descent algorithm for \texorpdfstring{$R$}{R} optimization - dual problem}\label{sec:optRC}

Assume that $S$ is positive semidefinite. In our block coordinate descent algorithm, the subproblem for updating $R$ involves solving \eqref{eq:prime} below, where the matrix $S$ is given by $S=\hat{D}C\hat{D}$. 

We begin by considering the original GLASSO optimization problem with a general penalty $\lambda_{ij}=\lambda_{ji}\geq0$):
\[
\hat{K} = \argmin_{K\in\Spp} \left\{ -\log\det(K)+\tr(K S)+\sum_{i,j} \lambda_{ij} |K_{ij}|\right\}. 
\]
Because the $\ell_1$ regularization term is non-smooth, direct optimization is challenging. Consequently, many methods instead focus on the dual formulation:
\[
\hat{K}^{-1} = \argmax_{W\in\Spp}\{\log\det(W)\colon |W_{ij}-S_{ij}|\leq \lambda_{ij} \,\, \forall\,i,j\}.
\]
In \cite{Banerjee}, a block-coordinate descent method was proposed to solve this dual problem by iteratively updating one column and the corresponding row of $W$. They showed that each column-subproblem can be reformulated as a LASSO regression, which \cite{friedman2008sparse} later solved efficiently using coordinate descent.

Analogously, we consider the dual problem corresponding to the following $R$-op\-ti\-mi\-za\-tion:
\begin{align}\label{eq:prime}
\hat{R} = \argmin_{R\in\Sppone} \left\{ -\log\det(R)+\tr(R S)+\lambda \|R\|_{1,\mathrm{off}}\right\}.
\end{align}
The dual is given by the following lemma.
\begin{lemma}\label{lem:dual}
   The dual of \eqref{eq:prime} is 
\[
\hat{R}^{-1} = \argmax_{W\in\Spp}\left\{ \log\det(W)-\tr(W)\colon |W_{ij}-S_{ij}|\leq \lambda\, \,\forall\,i\neq j \right\}. 
\]
\end{lemma}
Following the approach in \cite{Banerjee}, we note that updating a single column of $W$ can also be reduced to a LASSO regression. This observation motivates an iterative algorithm that updates one column (and its corresponding row) of $W$ at a time.

To illustrate the update step, partition $W$ and $S$ as follows:
\[
W = \begin{pmatrix}
    W_{11} & w_{12} \\ w_{12}^\top & w_{22}
\end{pmatrix}\quad\mbox{and}\quad S = \begin{pmatrix}
    S_{11} & s_{12} \\ s_{12}^\top & s_{22}
\end{pmatrix},
\]
where $w_{12}\in\R^{p-1}$ comprises the off-diagonal elements for the column under update and $w_{22}\in\R$ is its diagonal element.
Using the Schur complement, the objective can be decomposed as
\[
\log\det(W) -\tr(W)=  \log\det(W_{11})+\log(w_{22}-w_{12}^\top W_{11}^{-1} w_{12}) -\tr(W_{11})-w_{22}. 
\]
To update $w_{12}$, we solve
$w_{12} = \argmin_{y\in \R^{p-1}} \{ y^\top W_{11}^{-1} y\colon \|y-s_{12}\|_\infty\leq \lambda \}$,
which mirrors the standard GLASSO update. As shown in \cite{Banerjee}, this problem is equivalent to the LASSO regression:
\[
\hat{\beta}:=W_{11}^{-1} w_{12} = \argmin_{\beta\in\R^{p-1}} \left\{ \frac12 \| W_{11}^{1/2} \beta - W_{11}^{-1/2}s_{12}\|_2^2 + \lambda \|\beta\|_1\right\}.
\]
We solve the above using coordinate descent.

Once $\hat{\beta}$ (and hence $w_{12}$) is obtained, the diagonal element $w_{22}$ is updated as 
\begin{align*}
w_{22} &= \argmax_{d\in\R}\left\{ \log(d-w_{12}^\top W_{11}^{-1} w_{12}) -d\right\}= 1+w_{12}^\top W_{11}^{-1}w_{12} = 1+w_{12}^\top\hat{\beta}.
\end{align*}
This update ensures that the corresponding diagonal entry of $R=W^{-1}$ equals exactly $1$. 
Finally, using the identity
\[
 \begin{pmatrix}
    W_{11} & w_{12} \\ w_{12}^\top & w_{22}
\end{pmatrix}  \begin{pmatrix}
    R_{11} & r_{12} \\ r_{12}^\top & r_{22}
\end{pmatrix} = \begin{pmatrix}
    I_{p-1} & 0 \\ 0^\top & 1
\end{pmatrix},
\]
one obtains $W_{11}r_{12}+w_{12} r_{22}=0\in\R^{p-1}$.
Since $r_{22}=1$, it follows that
$r_{12}=-W_{11}^{-1}w_{12}=-\hat{\beta}$. 

The following algorithm and the corresponding Fortran implementation is a minor adaptation of the \texttt{glassoFast} algorithm of \cite{glassoFAST}. The modifications are: (i) a new pre-processing step in line 2, (ii) PCGLASSO‑specific updates in lines 23–25, and (iii) a new post‑processing block in lines 27–30.
Up to line 26 of the pseudo-code, the off-diagonal entries of the $j$th column of $R$ (denoted by $R_{\cdot j}$) contain the corresponding $\hat{\beta}$ vector.
Recall the soft-threshold function 
$\mathrm{soft}(x,\lambda) = \sign(x)(|x|-\lambda)_+$.

\begin{algorithm}[H]
\caption{Coordinate Descent Algorithm for solving \eqref{eq:prime}; An adaptation of the glassoFast algorithm by \cite{glassoFAST}}\label{alg:R}
\begin{algorithmic}[1]
\Require $S$: a $p\times p$ positive semidefinite matrix, $\lambda\in[0,\infty)$: tuning parameter, $\tau$: convergence threshold\Comment{Input}
\Ensure Optimal $R$ and $W=R^{-1}$ from \eqref{eq:prime}\Comment{Output}
\Procedure{pcglassoFast\_Dual}{$S$, $\lambda$, $\tau$}
\State Initialize $R \gets 0\in\R^{p\times p}$, $W\gets I_p$ 
\Repeat
    \State $\Delta_{\max} \gets 0$
    \For{$j = 1,\ldots,p$}
        \State $v \gets W R\, e_j$ \Comment{Compute the $j$th column of $W R$}
        \Repeat
            \State $\delta_{\max} \gets 0$
            \For{$i=1,\ldots,p$} 
                \If{$i\neq j$} \Comment{LASSO update}
                    \State $c\gets\mathrm{soft}(S_{ij}-v_i+W_{ii} R_{ij}, \lambda)/W_{ii}$\Comment{Apply soft-threshold}
                    \State $\delta\gets c-R_{ij}$
                    \If{$\delta\neq0$}
                        \State $R_{ij} \gets c$
                        \State $v\gets v+\delta\cdot W_{\cdot i}$\Comment{$W_{\cdot i}$ is the $i$th column of $W$}
                        \State $\delta_{\max} \gets \max\{\delta_{\max}, |\delta|\}$
                    \EndIf
                \EndIf
            \EndFor
        \Until $\delta_{\max}\cdot p < \tau$ \Comment{LASSO convergence test}
        \State $\Delta_{\max} \gets \max\{\Delta_{\max}, \|W_{\cdot j}-v\|_1\}$
        \State $W_{\cdot j}\gets v$, $W_{j \cdot}\gets v^\top$\Comment{Update $j$th column and $j$th row of $W$}
        \State $\Delta_{\max} \gets \max\{\Delta_{\max}, |1+W_{\cdot j}^\top R_{\cdot j} - W_{jj}|\}$
        \State $W_{jj}\gets 1+W_{\cdot j}^\top R_{\cdot j}$ \Comment{Update the diagonal of $W$}
    \EndFor
\Until $\Delta_{\max} < \tau$ \Comment{Convergence test}
\State $R\gets -R$
\For{$i=1,\ldots,p$}
    \State $R_{ii}\gets 1$
\EndFor
\State \Return $(R,W)$
\EndProcedure
\end{algorithmic}
\end{algorithm}
For a warm‑start initialization, substitute line 2 of Algorithm \ref{alg:R} with\\
2: $R \gets -R_0, \diag{R} \gets 0$, $W\gets W_0$.

\section{Justification for the diagonal Hessian approximation}\label{appendix_D_optim}

In Section~\ref{sec:algD}, we presented the optimization scheme for estimating $D$ given $R$. As the underlying problem is convex, employing a standard Newton-Raphson algorithm is suitable (iteration with Equation~(\ref{eq:NR})). However, given the computational cost of each iteration of Newton's method, we considered a diagonal Hessian approximation as a potential simplification.

To assess the practical usefulness of this approximation, we implemented both the exact Newton method and its diagonal version and compared their runtimes empirically. Figure~\ref{fig:appendix_diag} reports the average runtime of the $D$-update as a function of the dimension $p$, together with 95\% confidence intervals. The experiment was carried out on subproblems derived from the \texttt{stockdata} dataset from the R package \texttt{huge}; implementation details are available in our code repository.

\begin{figure}[h]
    \centering
    \includegraphics[width=0.9\linewidth]{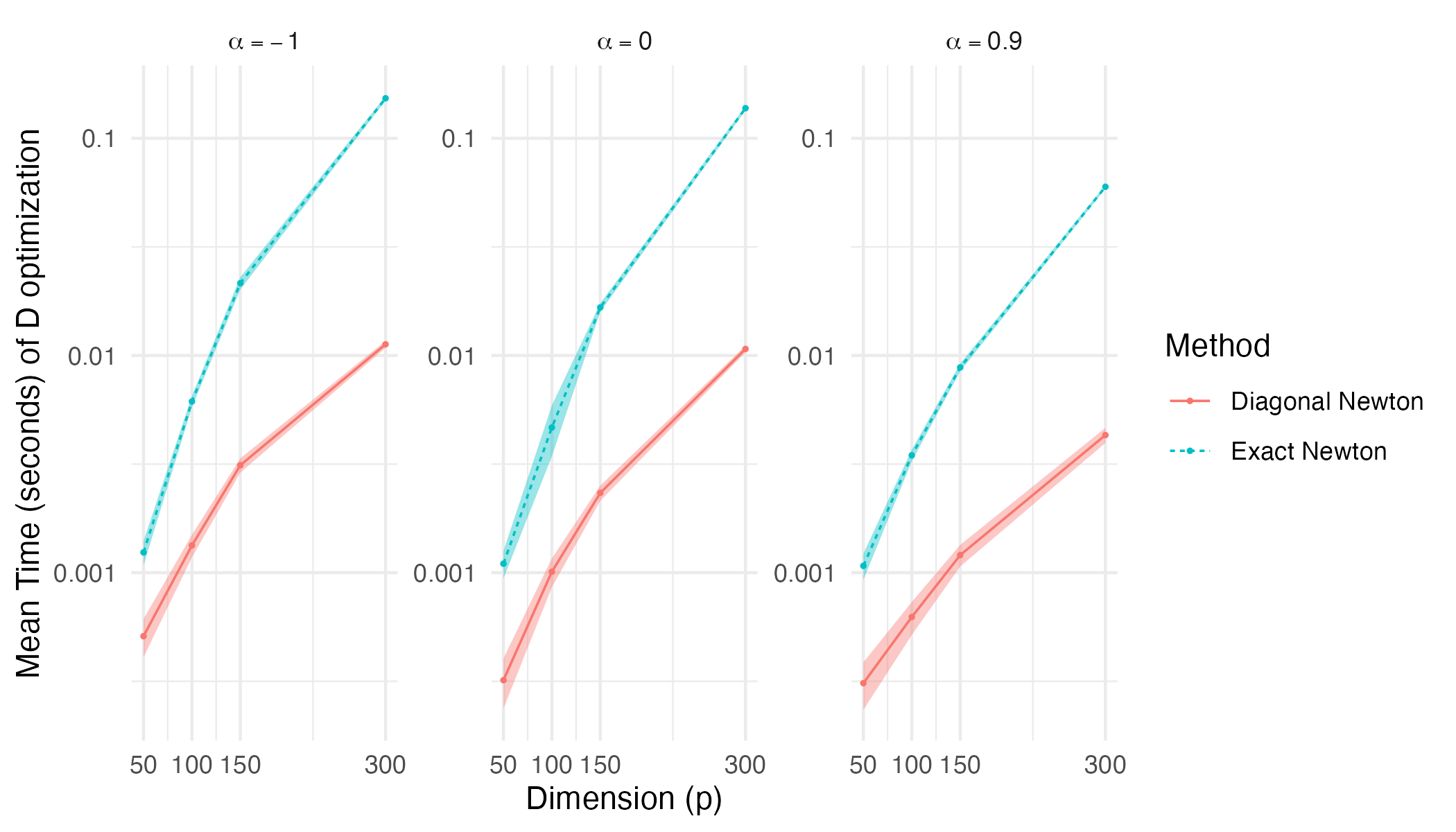}
    \caption{Mean runtime comparison for optimizing $D$ given $R$ between the diagonal Newton approximation (red solid line) and the exact Newton algorithm (blue dashed line). Shaded areas represent the 95\% confidence intervals for the mean runtime.}
    \label{fig:appendix_diag}
\end{figure}

The results clearly show that the diagonal Hessian approximation substantially reduces computation time. For the largest dimensions considered, it is about ten times faster than the exact Newton method. This provides strong empirical support for using the diagonal approximation in practice.

Moreover, the following classical result from numerical optimization guarantees the convergence of the diagonal Newton approximation in our setting. The theorem is stated and proven, for instance, in \cite[Theorem 3.2 and Eq. (3.20)]{Nocedal2006}.

\begin{thm}\label{thm:Nocedal}
Let $f\colon \mathbb{R}^p \to \mathbb{R}$ be a function that is bounded below and continuously differentiable on an open set $\mathcal{N}$ containing the level set $\mathcal{L} = \{ d \in \mathbb{R}^p \colon f(d) \leq f(d^0) \}$, where $d^0$ is the initial point of the iteration. Consider the iterative scheme $d^{k+1} = d^k + \alpha_k p_k$, where $\alpha_k$ satisfies the Wolfe conditions and $p_k = -B_k^{-1} \nabla f(d^k)$ for some symmetric and positive definite matrices $B_k$. Assume the following:
\begin{enumerate}
    \item The condition numbers of $B_k$ are uniformly bounded, i.e., there exists a constant $M \in (0,\infty)$ such that for all $k \geq 0$, $\kappa(B_k) =  \frac{\lambda_{\max}(B_k)}{\lambda_{\min}(B_k)} \leq M$, where $\lambda_{\max}(B_k)$ and $\lambda_{\min}(B_k)$ are maximum and minimum eigenvalues of $B_k$.
    \item The gradient $\nabla f$ is Lipschitz continuous on $\mathcal{N}$.
\end{enumerate}

Then,
\[\lim_{k \to \infty} \| \nabla f(d^k) \| = 0.\]
\end{thm}

Note that, in general, the result guarantees convergence to a stationary point, which could be a saddle point, but in our case the optimization problem is convex, so convergence is to the global optimum. Note also that the proof of Theorem \ref{thm:Nocedal} never uses the values of $f$ outside the set $\mathcal{N}$ and therefore it can be generalized to any $f$ defined on a subset of $\R^p$.

Let us now verify that the assumptions of Theorem \ref{thm:Nocedal} are satisfied in our setting. The function $f\colon (0,\infty)^p\to\R$ given by $f(d) = \frac12 d^\top Ad-\sum_{i=1}^p \log(d_i)$ is bounded below and continuously differentiable. Fix $d^0\in (0,\infty)^p$ and define an open set $\mathcal{N} = \{ d \in (0,\infty)^p\colon f(d) < f(d^0)+1\}$ so that $\mathcal{L} = \{ d \in (0,\infty)^p\colon f(d) \leq f(d^0) \}\subset\mathcal{N}$. Our line search enforces the Wolfe conditions.

It is left to check the assumptions (1) and (2). By coercivity of $f$, the closure $\overline{\mathcal{N}}$ is compact so that there exist $\varepsilon\in(0,1)$ such that $\mathcal{L} \subset \mathcal{N}\subset \overline{\mathcal{N}} \subset [\varepsilon, \varepsilon^{-1}]^p$. Note that for every $k\in\mathbb{N},\ d^{k}\in\mathcal{L}$. In our case, we have $B_k = \diag{d^{k}}^{-2}+\frac{1}{1-\alpha}I_p$ and it is easy to see we have 
\[
\kappa(B_k)  \leq \max_{d\in\mathcal{L}} \left\{\kappa\left( \diag{d}^{-2}+\frac{1}{1-\alpha}I_p\right)\right\} \leq 1 + \varepsilon^{-2}(1-\alpha) \eqqcolon M < \infty. 
\]
It remains to show that the gradient $\nabla f(d)$ is Lipschitz continuous on $\mathcal{N}$. It follows from the fact that its Jacobian is bounded. The Jacobian of $\nabla f(d)$ is the Hessian $\nabla^2 f(d) = \diag{d}^{-2}+A$. We have:
\[
\|\nabla^2 f(d)\|_2 = \|\diag{d}^{-2}+A\|_2 \leq \|\diag{d}^{-2}\|_2 + \|A\|_2 \leq \frac{1}{\varepsilon^2} + \|A\|_2,
\]
which establishes the result.

\section{Proofs}\label{sec:proofs}

\subsection{Proof of Theorem \ref{thm:Dregion}}

We start with a simple result that will be used in the proof of Theorem~\ref{thm:Dregion}.  
\begin{lemma}\label{lem:Lmin}
Assume that $A$ and $B$ are are positive semidefinite. Then,
\[
\lambda_{\min}(A\odot B) \geq \max\{ \lambda_{\min}(A),\, \lambda_{\min}(B)\}.
\]
\end{lemma}
\begin{proof}[Proof of Lemma~\ref{lem:Lmin}]
Since $A$ and $B$ are positive semidefinite, their smallest eigenvalues, $\lambda_{\min}(A)$ and $\lambda_{\min}(B)$, are nonnegative.  Define
$\alpha = -\lambda_{\min}(A)$ and $\beta = -\lambda_{\min}(B)$.
Then, the matrices
\[
A+\alpha I_p \quad \text{and} \quad B+\beta I_p
\]
are positive semidefinite. By the Schur product theorem, \cite[Section 7.5]{Horn13}, the Hadamard product
\[
(A+\alpha I_p)\odot (B+\beta I_p)
\]
is also positive semidefinite. Moreover, since $A$ and $B$ have unit diagonals, we have
\[
(A+\alpha I_p)\odot (B+\beta I_p) = A\odot B + \Bigl((1+\alpha)(1+\beta)-1\Bigr) I_p.
\]
Because the above matrix is positive semidefinite, its smallest eigenvalue is nonnegative. Therefore,
\[
\lambda_{\min}\Bigl(A\odot B + \Bigl((1+\alpha)(1+\beta)-1\Bigr) I_p\Bigr) \ge 0,
\]
which implies
\[
\lambda_{\min}(A\odot B) \ge 1 - (1+\alpha)(1+\beta).
\]
Expanding the right-hand side yields
\[
1 - (1+\alpha)(1+\beta) = 1 - (1+\alpha+\beta+\alpha\beta) = -\alpha - \beta - \alpha\beta.
\]
Substituting back $\alpha = -\lambda_{\min}(A)$ and $\beta = -\lambda_{\min}(B)$, we obtain
\[
\lambda_{\min}(A\odot B) \ge \lambda_{\min}(A) + \lambda_{\min}(B) - \lambda_{\min}(A)\lambda_{\min}(B).
\]
Since both $A$ and $B$ have unit diagonals, it follows that $\lambda_{\min}(A) \le 1$ and $\lambda_{\min}(B) \le 1$. Consequently,
\[
\lambda_{\min}(A) + \lambda_{\min}(B) - \lambda_{\min}(A)\lambda_{\min}(B) \ge \max\{\lambda_{\min}(A),\,\lambda_{\min}(B)\}.
\]
This completes the proof.
\end{proof}

\begin{lemma}\label{lem:A>0}
    For any $\alpha<1$, $R\in\Sppone$ and correlation matrix $C$, the matrix $A=(R\odot C)/(1-\alpha)$ is positive definite and 
    \[
    \lambda_{\min}(A) \geq \frac{\lambda_{\min}(C)}{1-\alpha}. 
    \]
\end{lemma}
\begin{proof}[Proof of Lemma~\ref{lem:A>0}]
Since $R$ is positive definite and $C$ is positive semidefinite, the matrix $
A = \frac{1}{1-\alpha}\,R\odot C$
is positive definite. Indeed, it is well known that the Hadamard product of two positive semidefinite matrices is itself positive semidefinite. Therefore, it suffices to show that \(R\odot C\) is nonsingular. By Oppenheim's inequality (see \cite{Opp30}), we have
\[
\det(R\odot C) \ge \det(R) \left(\prod_{i=1}^p C_{ii}\right) = \det(R) > 0.
\]
The inequality for $\lambda_{\min}(A)$ follows directly from Lemma~\ref{lem:Lmin}. 
\end{proof}

The following result is proved in \cite{Kal92}.
\begin{lemma}\label{lem:D}
Assume that $A$ is positive semidefinite. Then, there exists a solution to \eqref{eq:scaling} if and only if 
\[
\mu(A) = \min_{y\in[0,\infty)^p\setminus\{0\}} \left\{ \frac{y^\top  A y}{y^\top y} \right\} > 0.  
\]
If $\mu(A)>0$, then the solution is unique and satisfies
\begin{align}\label{eq:ineqD2}
\tr(D^2) \leq \frac{p}{\mu(A)}.
\end{align}
\end{lemma}

\begin{proof}[Proof of Theorem~\ref{thm:Dregion}]
The existence and uniqueness of a solution to \eqref{eq:scaling} is established by Lemmas \ref{lem:A>0} and \ref{lem:D}. Indeed, we have
\[
\mu(A) \geq \min_{y\in \R^p\setminus\{0\}} \left\{ \frac{y^\top A y}{y^\top y}\right\} = \lambda_{\min}(A) > 0. 
\]

Suppose that $C$ is positive definite. 
By Lemma \ref{lem:A>0} we arrive at
\[
\mu(A) \geq \frac{\lambda_{\min}(C)}{1-\alpha}. 
\]
Since $\tr(D^2) = \sum_{i=1}^p D_{ii}^2$, we have by Lemma \ref{lem:D}, for any $i\in\{1,\ldots,p\}$,
\[
D_{ii}\leq \sqrt{\tr(D^2)} \leq \sqrt{\frac{p}{\mu(A)}} \leq \sqrt{\frac{p (1-\alpha)}{\lambda_{\min}(C)}}.
\]
Since $|A_{ij}| = \frac{1}{1-\alpha}|\hat{R}_{ij}C_{ij}|\leq \frac{1}{1-\alpha}$, we have
\[
\frac{1}{D_{ii}} = \sum_{j=1}^p D_{jj} A_{ij} \leq \sqrt{\tr(D^2) \sum_{j=1}^p A_{ij}^2} 
\leq \sqrt{\tr(D^2) \frac{p}{(1-\alpha)^2}}. 
\]
Thus, by \eqref{eq:ineqD2}, we get 
\[
\frac{1}{D_{ii}} \leq \sqrt{\frac{(1-\alpha)p}{\lambda_{\min}(C)} \frac{p}{(1-\alpha)^2}} = \frac{p}{\sqrt{(1-\alpha)\lambda_{\min}(C)}}. 
\]
\end{proof}

\subsection{Proof of Lemma \ref{lem:dual}}

\begin{proof}[Proof of Lemma \ref{lem:dual}]
First, observe that the off-diagonal penalty may be written using its dual norm representation. Specifically, we have
\[
\lambda \sum_{i\neq j}|R_{ij}| = \max_{\substack{|Z_{ij}|\leq \lambda\\ i\neq j}} \sum_{i\neq j} Z_{ij} R_{ij}.
\]
We enforce the constraints $R_{ii}=1$, $i=1,\ldots,p$, by introducing Lagrange multipliers $Z_{ii}$. In this way, the Lagrangian for the primal problem becomes 
\begin{align*}
\mathcal{L}(R, Z) &= \log\det(R) - \tr(S R) - \sum_{i\neq j} Z_{ij} R_{ij} - \sum_{i=1}^p Z_{ii}(R_{ii}-1) \\
&= \log\det(R)-\tr((S+Z)R) + \tr(Z).
\end{align*}
Setting $W=S+Z$, we have $\mathcal{L}(R, Z) = \log\det(R)-\tr(WR)+\tr(W-S)$. 

Stationarity with respect to $R$ gives 
$R^{-1} = W$.
Since $R$ is positive definite, so is $W$.

We express the Lagrangian solely in terms of the dual variable $W=R^{-1}$ to obtain the dual objective (to be minimized)
\begin{align*}
\mathcal{L}(W^{-1}, Z) = -\log\det(W) + \tr(W) \quad (\,+\,\mbox{constant terms}\,)
\end{align*}
with the constraint $W\in\Spp$ and
\[
|W_{ij}-S_{ij}|\leq \lambda\quad\forall\, i\neq j.
\]
Under the strict concavity of $\log\det$ and the affine equality constraints $R_{ii}=1$, strong duality holds. This guarantees that the optimal value of the primal problem coincides with that of the dual problem.
\end{proof}

\subsection{Proof of Lemma \ref{lem:biconvex}}
\begin{proof}[Proof of Lemma \ref{lem:biconvex}]
Let $f$ denote the objective function in \eqref{eq:main_problem}. It is clear that $f(\cdot,R)$ is strictly convex in $R$ and this fact was already noted in \cite[Proposition 4]{Carter}. Fix $R\in\Sppone$. We have 
\[
f(R,D) = -2(1-\alpha)\sum_{i=1}^p \log(d_{i}) + d^\top (R\odot C) d + [R\mbox{-terms}], 
\]
where $d=(D_{ii})_{i=1}^p \in\R^p$. 
By Lemma \ref{lem:A>0}, the matrix $R\odot C$ is positive definite. Hence, $f(R,\cdot)$ is a sum of strictly convex functions, making it strictly convex.
\end{proof}

\subsection{Proof of Lemma \ref{lem:charact}}

Since PCGLASSO optimization program is non-convex, we must employ concepts beyond the standard subgradient to analyze its properties, \cite{nonconvex}.
For a locally Lipschitz function $f\colon \R^n\to\R$, we define the generalized directional derivative of $f$ at a point $x$ in the direction $v$ by 
\[
f^\circ(x,v) = \limsup_{y\to x,\, h\downarrow0} \frac{f(y+h v)-f(y)}{h}.
\]
The Clarke subgradient of $f$ at $x$ is then given by
\[
\partial_C f(x) = \{ \xi\in\R^n\colon\xi^\top v \leq f^\circ(x,v)\mbox{ for all }v\in\R^n\}. 
\]
If $f$ is convex, the Clarke subgradient coincides with the usual subgradient of $f$. Moreover, if $f$ is differentiable at $x$, then $\partial_C f(x) = \{\nabla f(x)\}$.  Suppose that $f$ is differentiable and that $g$ is convex. Then, 
\[
\partial_C(f+g)(x)= \{ \nabla f(x)\} + \partial g(x). 
\]
Finally, the condition $0\in\partial_C f(x)$ is necessary for $x$ to be a local extremum. 

In the following lemma, we present a condition under which the (Clarke) subgradient of the objective function in \eqref{eq:main_problem} vanishes. Since the objective in \eqref{eq:main_problem} is biconvex, all critical points correspond to coordinate-wise minimizers. Recall that the operations $\diag{\cdot}$ and $\odiag{\cdot}$ as well as the matrix $J_p'$ are defined in the Section \ref{sec:notation}.

\begin{proof}[Proof of Lemma~\ref{lem:charact}]
Let $f$ be the unpenalized objective of \eqref{eq:main_problem}, $f\colon \Sppone\times\Diagp\to \R$ defined by
\[  
f(R,D) = -\log\det(R)-2(1-\alpha)\log \det(D) + \tr(C D R D).
\]  
Let $D^1 f$ and $D^2f$ denote the differentials of $f$ with respect to its first and second arguments, respectively. 

\textbf{Differentiation with respect to $R$:} 
Consider the directional derivative of $f(\cdot, D)$  in the direction of matrix $M\in\Szero$:
\begin{align*}
\left< D^{1} f(R,D)| M\right> &= \lim_{\varepsilon\to 0} \frac{1}{\varepsilon}(f(R+\varepsilon M,D)-f(R,D)) \\
&=\tr(M DC D)-\tr(R^{-1} M)\\
&= \left< \odiag{DC D -R^{-1}}|M\right>.
\end{align*}

\textbf{Differentiation with respect to $D$:} 
Next, we differentiate $f$ with respect to $D$ in the direction $H \in \Diag$:
\begin{align*}
\left< D^{2} f(R,D)| H\right> & = \lim_{\varepsilon\to 0} \frac{1}{\varepsilon}(f(R,D+\varepsilon H)-f(R,D)) \\
&=  2\,\tr(R D C H) - 2(1-\alpha)\, \tr(D^{-1} H)\\
&= 2\, \left< \diag{R D C}-(1-\alpha) D^{-1}|H\right>.
\end{align*}
Setting this derivative equal to zero (i.e., for optimality in the $D$-direction) for all $H \in \Diag$ yields
$(1-\alpha)D^{-1}=\diag{R D C}$, which is equivalent to  
\begin{align}\label{eq:diffD}
\diag{R D C D}= (1-\alpha)I_p.
\end{align}

Incorporating the non-smooth term $\lambda \|R\|_{1,\mathrm{off}}$ into the optimization, we obtain that
 $0\in\partial_C^{(R)}  \left(f(R, D)+\lambda\|R\|_{1,\mathrm{off}} \right)$  if and only if
\begin{align}\label{eq:diffR}
\odiag{R^{-1}-D C D} = \lambda\, \Pi,
\end{align}
where $\Pi$ belongs to the subgradient  $\partial \|R\|_{1,\mathrm{off}}$. 

Since $\diag{R}=I_p=\diag{C}$, it follows that by \eqref{eq:diffD} and \eqref{eq:diffR},
\begin{align}\label{eq:D-Ri}
\begin{split}
    (1-\alpha)I_p &=  \diag{R D C D} = \diag{R\, \diag{D C D}} + \diag{R\, \odiag{D C D}} \\
    &= D^2 + \diag{R  ( \odiag{R^{-1}} - \lambda \Pi )}\\
    &= D^2 + \diag{R  R^{-1}} - \diag{R \, \diag{R^{-1}}}-\lambda\, \diag{R \,\Pi} \\
    &= D^2 + I_p - \diag{R^{-1}} - \lambda\,\diag{R\, \Pi}.
    \end{split}
\end{align}
We have $\Pi\in\partial \|R\|_{1,\mathrm{off}}$if and only if $\diag{\Pi}=0$ and
\[
\begin{cases}
	\Pi_{ij}=\sign(R_{ij}), & R_{ij}\neq 0, i\neq j \\
	\Pi_{ij}\in[-1,1], & R_{ij}=0.
\end{cases}
\]
In particular, we have $\diag{R\,\Pi} = \diag{J_p' |R|}$, where $|R|=(|R_{ij}|)_{i,j}$. Indeed, one may verify that
\[
(R\,\Pi)_{ii} = \sum_{k=1}^p R_{k i}\Pi_{k i}  = \sum_{\substack{k=1,\ldots,p\\ k\neq i}} |R_{ki}| =(J_p' |R|)_{ii}.
\]
Finally, by \eqref{eq:D-Ri}, we deduce that
\begin{align*}
R^{-1}-DCD -\lambda\Pi &= \diag{R^{-1}-DCD} 
=\diag{R^{-1}} - D^2 \\
&=\alpha I_p - \lambda\,\diag{J_p' |R|}, 
\end{align*}
which is \eqref{eq:R-DCD}.

\end{proof}

\subsection{Proof of Lemma \ref{lem:cons}}

\begin{proof}[Proof of Lemma~\ref{lem:cons}]
    By \eqref{eq:R-DCD}, we have 
    \[
\hat{R}^{-1}-\hat{D} C \hat{D} =  \lambda\Pi + \alpha I_p  - \lambda\,\diag{J_p' |\hat{R}|}.
    \]
Since $\|\Pi\|_\infty\leq 1$, and $\|\diag{J_p' |\hat{R}|}\|_\infty\leq p-1$, we obtain 
\[
\|\hat{R}^{-1}-\hat{D}C\hat{D}\|_\infty\leq \lambda+|\alpha|+\lambda(p-1). 
\]
Further, if $C$ is positive definite, then 
\[
\| \hat{D}^{-1}(\hat{R}^{-1}-\hat{D}C\hat{D})\hat{D}^{-1}\|_\infty \leq \|\hat{D}^{-1}\|_\infty^2 \|\hat{R}^{-1}-\hat{D}C\hat{D}\|_\infty.
\]
Thus, the result follows from Theorem \ref{thm:Dregion}. 
\end{proof}

\subsection{Proof of Theorem \ref{thm:unique}}

We start with a couple of lemmas. 

\begin{lemma}\label{lem:convex}
Define the function $f\colon \Sppone\times \Diagp\to\R$ by
\begin{align}\label{eq:fun}
f(R,D)= -\log\det(R)-2(1-\alpha)\log \det(D) + \tr(R D C D).
\end{align}
Then, $f$ is convex at a point $(R,D)\in \Sppone\times \Diagp$ if and only if
\begin{align}\label{eq:ispositive}
 \tr(M R^{-1} M R^{-1}) + 4\, \tr(D C H M) 
+2(1-\alpha)\, \tr(D^{-2}H^2)+2\,\tr(R H C H)\geq0
\end{align}
for all $M\in \Szero$ and $H\in\Diag$.
\end{lemma}
\begin{remark}\label{rem:not_globally_convex}
If $C\neq I_p$, then the function in \eqref{eq:fun} is not globally convex. Indeed, if $C_{ij}\neq 0$ for some $i\neq j$, one may take $M\in \Szero$ defined by 
\[
M_{ij}=M_{ji}=-\sign(C_{ij}) \quad \text{and} \quad M_{kl}=0 \quad \text{for all } (k,l)\notin\{(i,j),(j,i)\}.
\]
Then,
\[
\tr(D C H M)  = -|C_{ij}| (D_{ii} H_{jj}+D_{jj}H_{ii}),
\]
which is negative whenever $D_{ii},D_{jj}$ and $H_{ii},H_{jj}$ are positive. Hence, by increasing the entries of $D\in \Diagp$, this negative term will eventually dominate the other terms in \eqref{eq:ispositive}, showing that the inequality fails and that $f$ is not convex on the entire domain.
\end{remark}

\begin{proof}[Proof of Lemma~\ref{lem:convex}]

Let $f$ denote the function \eqref{eq:fun}. Since $f$ is twice differentiable, it is convex at a given point if and only if its Hessian is semi-positive definite in that point. 

We express the Hessian of $f$ in block form:
\[
H(R,D) = \begin{pmatrix}
    D^{1,1} f(R,D) & D^{1,2} f(R,D) \\
    (D^{1,2} f(R,D))^\top  & D^{2,2} f(R,D)
\end{pmatrix},
\]
where $D^{i,j}$ denotes the second order differential in $i$th and $j$th variable, $i,j=1,2$. Then, $H(R,D)$ is positive semidefinite if and only if for all $M\in \Szero$ and $H\in \Diag$ the following inequality holds:
\begin{multline}\label{eq:HPD}
\left< D^{1,1} f(R,D) M\mid M\right>_{\Szero}+\left< D^{2,2} f(R,D) H\mid H\right>_{\Diag} \\
+ \left< D^{1,2} f(R,D) M\mid H\right>_{\Diag}+
\left< (D^{1,2} f(R,D))^\top H\mid M\right>_{\Szero}
\geq 0,
\end{multline}
where $\left<\cdot\mid\cdot\right>_{\Szero}$ and $\left<\cdot\mid\cdot\right>_{\Diag}$ denote the trace inner product on $\Szero$ and $\Diag$, respectively.

For $M_1,M_2\in \Szero$, we have 
\begin{align*}
\left<D^{1,1}f(R,D)M_1\mid M_2\right>_{\Szero}&=\frac{d^2}{d\varepsilon_1\,d\varepsilon_2}f(R+\varepsilon_1 M_1+\varepsilon_2 M_2,D)\mid_{\varepsilon_1=\varepsilon_2=0} \\
&=\frac{d}{d\varepsilon_1} \tr( (-(R+\varepsilon_1 M_1)^{-1}+D C D)\cdot M_2) \\
&= \tr( R^{-1} M_1 R^{-1} M_2). 
\end{align*}

For $H_1, H_2\in \Diag$, we obtain
\begin{align*}
\left<D^{2,2}f(R,D)H_1\mid H_2\right>_{\Diag} &=\frac{d^2}{d\varepsilon_1\,d\varepsilon_2}f(R,D+\varepsilon_1 H_1+\varepsilon_2 H_2)\mid_{\varepsilon_1=\varepsilon_2=0} \\
&= \frac{d}{d \varepsilon_1} \left( -2(1-\alpha)\, \tr( (D+\varepsilon H_1)^{-1})+2\,\tr(R (D+\varepsilon H_1) C H_2)\right)\\
&= 2(1-\alpha)\,\mathrm{tr}( D^{-1}H_1 D^{-1}H_2)+2\,\mathrm{tr}(R H_1 C H_2).
\end{align*}

For $M\in \Szero$ and $H\in \Diag$, we have
\begin{align*}
\left<D^{1,2}f(R,D)M\mid H\right>_{\Diag} &=  \frac{d^2}{d\varepsilon_1\,d\varepsilon_2}f(R+\varepsilon_1 M, D+\varepsilon_2 H)\mid_{\varepsilon_1=\varepsilon_2=0} \\
&=\frac{d}{d\varepsilon_1} 2\, \tr( (- D^{-1}+(R+\varepsilon_1 M) D C) H) \\
&=2\,\tr(M D C H) = \left< (D^{1,2} f(R,D))^\top H\mid M\right>_{\Szero}.
\end{align*}
Substituting these expressions into \eqref{eq:HPD} yields the inequality in \eqref{eq:ispositive}. This completes the proof.
\end{proof}

\begin{lemma}\label{lem:ineq}
    Suppose that $A$ is positive definite and $B$ is symmetric. Then,
\begin{align*}
     \tr(A B A B)\geq \lambda_{\min}(A)^2 \tr(B^2). 
\end{align*}
\end{lemma}
\begin{proof}
First, assume that $A$ is diagonal with positive entries. Then,
\[
\tr(A B A B) = \sum_{i,j} A_{ii}A_{jj}B_{ij}^2 \geq \min_i\{A_{ii}^2\} \sum_{i,j}B_{ij}^2 = \lambda_{\min}(A)^2 \tr(B^2). 
\]
If $A$ is not diagonal, write its spectral decomposition as $A=U \Lambda U^\top$, where $U$ is orthogonal and $\Lambda$ is the diagonal matrix of eigenvalues of $A$. Define $\tilde{B}=U^\top B U$, which is symmetric. Then,
\[
\tr(A B A B) = \tr(\Lambda \tilde{B} \Lambda \tilde{B})\geq \min_i\{ \Lambda_{ii}^2\} \tr(\tilde{B}^2)=\lambda_{\min}(A)^2 \tr(B^2),
\]
where the last equality follows from the invariance of the trace under orthogonal transformations.
\end{proof}

\begin{lemma}\label{lem:ineq2}
Fix $\gamma>0$. For any $H\in\Diag$ and any $M\in\Szero$, we have 
\[
\left|\tr(C H M)\right| \leq \frac{\gamma}{2} \opnorm{C-I_p}\, \tr(H^2)+\frac{1}{2\gamma} \|C-I_p\|_\infty \,\tr(M^2).
\]
\end{lemma}
\begin{proof}
Since $H$ is diagonal and $M\in\Szero$ (so that $M_{ii}=0$ for all $i$), a short calculation shows that
  \[
\tr(C H M) = \sum_{i\neq j} (H_{ii}+H_{jj})M_{ij} C_{ij}.
\]
Applying the inequality $h m \leq \frac12( h^2 \gamma + m^2/\gamma)$ for positive $\gamma$, we obtain
\begin{align*}
\left|\tr(C H M)\right| & \leq \sum_{i\neq j} |H_{ii}| |M_{ij}| |C_{ij}| \leq \sum_{i\neq j}  |C_{ij}| \frac1{2} \left( H_{ii}^2\gamma +  M_{ij}^2/\gamma \right) \\
&= \frac{\gamma}{2} \sum_{i} \left(\sum_{j\neq i}|C_{ij}|\right) H_{ii}^2 + \frac {1}{2\gamma} \sum_{i\neq j}  |C_{ij}| M_{ij}^2 \\
&\leq \frac{\gamma}{2} \opnorm{C-I_p} \,\tr(H^2)+\frac{1}{2\gamma} \|C-I_p\|_\infty \,\tr(M^2). 
\end{align*}  
\end{proof}

\begin{proof}[Proof of Theorem~\ref{thm:unique} (i) ]
Let $f$ denote the function in \eqref{eq:fun}. Since the penalty $R\mapsto \|R\|_{1,\mathrm{off}}$ is piecewise linear and continuous, its Hessian is a.e. zero. Thus, the function $(R,D)\mapsto f(R,D)+\lambda \|R\|_{1,\mathrm{off}}$ shares the same region of convexity as $f$.

Fix $\alpha<1$ and recall that $e=(1,\ldots,1)^\top\in\R^p$. For any $R\in\Sppone$ define function $\D(R)$ as the unique solution $D\in\Diagp$ to (cf. Eq. \eqref{eq:scaling})
\[
D \left( R\odot C \right) D e = (1-\alpha) e.
\]
By Theorem \ref{thm:Dregion}, such $D$ exists and is unique. 

We note that $(\hat{R},\hat{D})$ satisfies \eqref{eq:main_problem}, i.e.,
\[
(\hat{R},\hat{D}) \in\Argmin_{R,D} \left\{ f(R,D)+\lambda \|R\|_{1,\mathrm{off}} \right\}
\]
if and only if $\hat{D} = \D(\hat{R})$, where 
\[
\hat{R} \in\Argmin_{R\in \Sppone}\left\{ f(R, \D(R)) + \lambda\|R\|_{1,\mathrm{off}}\right\}
\]
We will show that the function $R\mapsto f(R,\D(R))+\lambda\|R\|_{1,\mathrm{off}}$ is convex on $\Sppone$, which will imply that there is only a unique global minimum to \eqref{eq:main_problem}.

The function $R\mapsto f(R,\D(R))+\lambda\|R\|_{1,\mathrm{off}}$ is convex at a point $R\in\Sppone$ if \eqref{eq:ispositive} holds with $D=\D(R)$ for all $M\in \Szero$ and $H\in\Diag$. For notational simplicity, we write $\D$ instead of $\D(R)$.

Perform the change of variables $H\mapsto \D H\in\Diag$ and $M\mapsto \D^{-1}M \D^{-1}\in\Szero$ in \eqref{eq:ispositive}. With these substitutions, inequality \eqref{eq:ispositive} becomes
\begin{align}\label{eq:ispositive2}
 \tr(M (\D R\D)^{-1} M (\D R\D)^{-1}) + 4\, \tr(C H M)
+2(1-\alpha)\, \tr(H^2)+2\,\tr(R H \D C\D H)\geq0.
\end{align}
We aim to ensure that the positive quadratic terms dominate the indefinite cross-term $\tr(C HM)$. 
Let $A=\frac{1}{1-\alpha} R\odot C$. By Theorem~\ref{thm:Dregion}, we have the bound
\[
\tr(\D^2)\leq  \frac{1-\alpha}{\lambda_{\min}(C)}p.
\]
Moreover, since $\lambda_{\max}(\D R\D) \le \tr(\D R\D) = \tr(D^2)$, we deduce that
\begin{equation}\label{eq:lmaxDRD}
\lambda_{\max}(\D R\D) \le \frac{1-\alpha}{\lambda_{\min}(C)}p.
\end{equation}
By Lemma~\ref{lem:ineq}, it follows that
\[
\tr(M (\D R\D)^{-1} M (\D R\D)^{-1}) \geq \lambda_{\min}((\D R\D)^{-1})^2 \tr(M^2) = \frac{1}{\lambda_{\max}(\D R\D)^2} \tr(M^2). 
\]
Also, note that
\[
\tr(R H \D C\D H)\geq 0. 
\]
Application of Lemma~\ref{lem:ineq2} (with $\tilde{C}:=C-I_p=\odiag{C}$) to bound the cross-term yields
\begin{align*}
\left|\tr(C H M)\right| \leq \frac{\gamma}{2} \opnorm{\tilde{C}} \,\tr(H^2)+\frac{1}{2\gamma} \|\tilde{C}\|_\infty\, \tr(M^2).
\end{align*}
Hence, inequality \eqref{eq:ispositive2} holds if 
\begin{align*}
\frac{1}{\lambda_{\max}(\D R\D)^2} \tr(M^2)+2(1-\alpha)\, \tr(H^2)
\geq 2\gamma \opnorm{\tilde{C}} \tr(H^2)+\frac{2}{\gamma} \|\tilde{C}\|_\infty \tr(M^2)
\end{align*}
holds for some $\gamma>0$ and 
for all $H\in\Diag$ and $M\in\Szero$. This inequality holds for all such $H$ and $M$ if and only if 
\[
2 \lambda_{\max}(\D R\D)^2 \|\tilde{C}\|_\infty  \leq \gamma \leq \frac{1-\alpha}{\opnorm{\tilde{C}}}.
\]
In view of the bound \eqref{eq:lmaxDRD}, the inequality \eqref{eq:ispositive2} holds for some $\gamma>0$ if
\begin{align}\label{eq:lastineq}
2 \left( \frac{p(1-\alpha)}{\lambda_{\min}(C)}
\right)^2 \|\tilde{C}\|_\infty \leq \frac{1-\alpha}{\opnorm{\tilde{C}}}.
\end{align}
We will show that the above inequality holds true under the assumption
\begin{align}\label{eq:ass_bound}
\| \tilde{C}\|_\infty\leq \frac{1}{\sqrt{2(1-\alpha)p^3}},
\end{align}
We have $\opnorm{\tilde{C}} \leq (p-1) \| \tilde{C}\|_\infty$ and by the  Gershgorin circle theorem,
\[
\lambda_{\min}(C)\geq 1- \opnorm{\tilde{C}} \geq 1-(p-1)\|\tilde{C}\|_\infty. 
\]
Thus,
\[
 \frac{\|\tilde{C}\|_\infty\opnorm{\tilde{C}} }{\lambda_{\min}(C)^2} 
  \leq (p-1)\frac{\|\tilde{C}\|_\infty^2}{(1-(p-1)\|\tilde{C}\|_\infty)^2}
\]
and direct computation shows that, under \eqref{eq:ass_bound}, the right hand side above is bounded by $(2p^2(1-\alpha))^{-1}$, which implies \eqref{eq:lastineq}. 
This completes the proof.
\end{proof}

\begin{proof}[Proof of Theorem~\ref{thm:unique} (ii)]
 For $K\in\Spp$, $\lambda>0$ and $\alpha<1$, define 
\[
f_{\lambda,\alpha}(K) = -\log\det(K)+\tr(C K) + \lambda\, p(K)+ \alpha \log \det(\diag{K}),
\]
where we denote $p(K) = \| \diag{K}^{-1/2}K\diag{K}^{-1/2}\|_{1,\mathrm{off}}$.

By Lemma \ref{lem:cons}, all critical points $K=\hat{K}$ of $f_{\lambda,\alpha}$ must satisfy
\begin{align}\label{eq:bound1}
\| K^{-1}-C\|_\infty  \leq  \frac{(\lambda p +|\alpha|)p^2}{(1-\alpha)\lambda_{\min}(C)}=:m_1.
\end{align}
Moreover, by Theorem \ref{thm:Dregion}, we have 
\begin{align}\label{eq:bound2}
\|K\|_\infty \leq \|\hat{D}\|_\infty^2 \leq \frac{p(1-\alpha)}{\lambda_{\min}(C)}=:m_2.
\end{align}
Define a convex subset $\mathcal{K}_{\lambda,\alpha}$ of $\Spp$ defined by
\[
\mathcal{K}_{\lambda,\alpha} = \mathrm{conv}\{ K\in\Spp\colon  \eqref{eq:bound1}\mbox{ and }\eqref{eq:bound2} \mbox{ hold true}\}.
\]
We note that under \eqref{eq:bound1} and \eqref{eq:bound2}, we have 
\begin{align*}
\frac{1}{p(m_1+1)} \leq \lambda_{\min}(K)\leq \lambda_{\max}(K)\leq p \,m_2\quad\mbox{and}\quad \frac{1}{1+m_1}\leq K_{ii}\leq m_2. 
\end{align*}
Indeed, by Gershgorin's circle theorem, we obtain
\[
\lambda_{\min}(K) = \frac{1}{\lambda_{\max}(K^{-1})} \geq \frac{1}{\max_i \sum_{j=1}^p |(K^{-1})_{ij}|} \geq \frac{1}{p (\max_{i,j} |(K^{-1}-C)_{ij}| + |C_{ij}|)}. 
\]
The upper bound on $\lambda_{\max}$ follows from the same argument and \eqref{eq:bound2}. 
The upper bound on $K_{ii}$ follows directly from \eqref{eq:bound2}, while the lower is based on the inequality 
\[
K_{ii} \geq 1/(K^{-1})_{ii} \geq 1/(C_{ii}+m_1). 
\]

Clearly, these bounds also hold for all $K\in\mathcal{K}_{\lambda,\alpha}$.

We will show that for sufficiently small $\lambda$ and $\alpha$, the restriction $f_{\lambda,\alpha}\big|_{\mathcal{K}_{\lambda,\alpha}}$ is convex; this establishes the uniqueness of the  minimizer. 
To ease notation, we further write $f$ for $f_{\lambda,\alpha}$ and $\mathcal{K}$ for $\mathcal{K}_{\lambda,\alpha}$. 

Since $f$ is continuous, to establish convexity, it is enough to show that $f( (A+B)/2)\leq (f(A)+f(B))/2$ for all $A,B\in\mathcal{K}$. 

Denote $g(K) = \alpha\log\det(\diag{K}) = \alpha\sum_{i} \log K_{ii}$.  Using the fact that for $a,b>0$
\[
0<\log\left(\frac{a+b}{2}\right)-\frac{\log(a)+\log(b)}{2}\leq \frac{(a-b)^2}{8\min\{a^2,b^2\}}, 
\]
we obtain for $A,B\in\mathcal{K}$,
\[
g\left(\frac{A+B}{2}\right) - \frac{g(A)+g(B)}{2} \leq \max\{\alpha,0\} (1+m_1)^2 \frac{\|A-B\|_F^2}{8},
\]
where $\|A\|_F = \sqrt{\tr(A^2)}$ is the Frobenius norm. 
Similarly, by \cite[Lemma 15]{INEQ}, we have for $A,B\in\Spp$,
\[
-\log\det\left(\frac{A+B}{2}\right) + \frac{\log\det(A)+\log\det(B)}{2}\leq -\frac{\|A-B\|_F^2}{8\max\{\lambda_{\max}(A)^2,\lambda_{\max}(B)^2\}}.
\]

We therefore obtain for $A,B\in\mathcal{K}$,
\begin{multline*}
    f\left( \frac{A+B}{2}\right)- \frac{f(A)+f(B)}{2} \\
    \leq -\frac{\|A-B\|_F^2}{8\, (p\, m_2)^2} - \frac{\lambda}{2} \left(p(A)+p(B)-2\, p\left(\frac{A+B}{2}\right)\right) \\
    + \max\{\alpha,0\}(1+m_1)^2\frac{\|A-B\|_F^2}{8}.
\end{multline*}
We write $M=(A+B)/2$ and $\Delta = (A-B)/2$. Then $f$ is convex if 
\begin{align}\label{eq:concavef}
\left(\frac{1}{p^2 m_2^2} - \max\{\alpha,0\}(1+m_1)^2 \right)\|\Delta\|_F^2 + \lambda\left(p(M+\Delta)+p(M-\Delta)-2\, p\left(M\right)\right)\geq 0.
\end{align}
We write $p(M)$ as $\sum_{i\neq j} p_{ij}(M)$, where
$p_{ij}(M) = |M_{ij}|/\sqrt{M_{ii}M_{jj}}$. 

For any convex function $f$, we have 
\[
f(x)+f(y)\geq 2 f\left(\frac{x+y}{2}\right)\geq 2 f(u)+2 f'(u)\left(\frac{x+y}{2}-u\right),
\]
which implies 
\[
-2 f(u) \geq -f(x)-f(y) +2 f'(u)\left(\frac{x+y}{2}-u\right).
\]
Applying this inequality to $f(z)=z^{-1/2}$ and 
\begin{gather*}
    x=(M_{ii}-\Delta_{ii})(M_{jj}-\Delta_{jj}),\quad y=(M_{ii}+\Delta_{ii})(M_{jj}+\Delta_{jj}), \quad 
    u = M_{ii} M_{jj},
\end{gather*}
we obtain
\[
-2 p_{ij}(M) \geq -\frac{|M_{ij}|}{\sqrt{x}} - \frac{|M_{ij}|}{\sqrt{y}} - \frac{p_{ij}(M)}{M_{ii}M_{jj}}\Delta_{ii}\Delta_{jj}.
\]
Thus,
\begin{align*}
 I_{ij}:=&   p_{ij}(M-\Delta)+p_{ij}(M+\Delta)-2p_{ij}(M) \\
        &\geq
    \frac{|M_{ij}-\Delta_{ij}|-|M_{ij}|}{\sqrt{x}} + \frac{|M_{ij}+\Delta_{ij}|-|M_{ij}|}{\sqrt{y}}-(1+m_1)^2|\Delta_{ii}\Delta_{jj}|,
\end{align*}
where we used the fact that on $\mathcal{K}$ ($M\in\mathcal{K}$ by convexity of $\mathcal{K}$) we have 
\[
\frac{p_{ij}(M)}{M_{ii}M_{jj}} \leq (1+m_1)^2.
\]

We consider the following complementary cases
\begin{itemize}
    \item[(I)] $|M_{ij}|\leq |\Delta_{ij}|/2$ or $\Delta_{ij}=0$, 
    \item[(II)] $|M_{ij}|> |\Delta_{ij}|/2>0$  
\end{itemize}
In (I), we have $|M_{ij}-\Delta_{ij}|-|M_{ij}|\geq 0$ and $|M_{ij}+\Delta_{ij}|-|M_{ij}|\geq0$, which implies that 
\[
I_{ij}\geq -(1+m_1)^2|\Delta_{ii}\Delta_{jj}|.
\]
In (II), we have $|M_{ij}-\Delta_{ij}|-|M_{ij}|< 0$ or $|M_{ij}+\Delta_{ij}|-|M_{ij}|<0$, but both cannot hold simultaneously.  Suppose that $|M_{ij}-\Delta_{ij}|-|M_{ij}|< 0$, so we necessarily have $|M_{ij}+\Delta_{ij}|-|M_{ij}|>0$. Since $y=x + 2(\Delta_{ii}M_{jj}+\Delta_{jj}M_{ii})$, we have 
\[
\frac{1}{\sqrt{y}} \geq \frac{1}{\sqrt{x}}-\frac{1}{x^{3/2}}(\Delta_{ii}M_{jj}+\Delta_{jj}M_{ii}). 
\]
Thus,
\begin{align*}
I_{ij}& \geq \frac{|M_{ij}-\Delta_{ij}|+|M_{ij}+\Delta_{ij}|-2|M_{ij}|}{\sqrt{x}}\\
& \qquad-\frac{|M_{ij}+\Delta_{ij}|-|M_{ij}|}{x^{3/2}}(\Delta_{ii}M_{jj}+\Delta_{jj}M_{ii}) 
- \frac{p_{ij}(M)}{M_{ii}M_{jj}}\Delta_{ii}\Delta_{jj}\\
& \geq -(1+m_1)^3 m_2 |\Delta_{ij}|(|\Delta_{ii}|+|\Delta_{jj}|)  - (1+m_1)^2 |\Delta_{ii}\Delta_{jj}|, \end{align*}
where we used the triangle inequality and the fact that $B=M-\Delta$ and $M$ belong to $\mathcal{K}$ (so that $x\geq (1+m_1)^{-2}$). We obtain the same bound in the case $|M_{ij}+\Delta_{ij}|-|M_{ij}|<0$. 
Therefore, we obtain
\begin{multline*}
p(M+\Delta)   +     p(M-\Delta) - 2\, p(M) = \sum_{i\neq j}I_{ij}\\
\geq - (1+m_1)^3 m_2  \sum_{i\neq j} |\Delta_{ij}|(|\Delta_{ii}|+|\Delta_{jj}|)  - (1+m_1)^2 \sum_{i\neq j} |\Delta_{ii}\Delta_{jj}| -C \|\Delta\|_F^2,
\end{multline*}
with
\[
C = p(1+m_1)^2(1+m_2(1+m_1)).
\]
Thus, \eqref{eq:concavef} holds if 
\[
\frac{\lambda_{\min}(C)^2}{p^4 (1-\alpha)^2}=\frac{1}{p^2 m_2^2} \geq \max\{\alpha,0\}(1+m_1)^2 + \lambda\, C.
\]
If $(\lambda,\alpha)\to(0,0)$, the right hand side converges to $0$, while the left has strictly positive limit. Thus, this inequality holds for sufficiently small $\lambda$ and $\alpha$. 
\end{proof}

\subsection{Proof of Theorem \ref{thm:conv_in_dist}}

\begin{lemma}\label{Lemma g - derivative} 
    Let $K=DRD$. The directional derivative of 
    \[
    g\colon \Spp\ni K\mapsto R\in \Sppone
    \]
    in a direction $U\in \Sym$ is given by 
    \[
    g'(K;U) =  D^{-1} U D^{-1}-\frac{1}{2} R\, \diag{U} D^{-2}-\frac{1}{2} D^{-2}\diag{U} R,
    \]
or equivalently, 
\begin{align}\label{g - derivative} 
    \operatorname{vec}(g'(K;U)) = M_R^\top (D^{-1}\otimes D^{-1}) \operatorname{vec}(U),
\end{align}
where $M_R$ is defined by \begin{align}\label{eq:MR}
M_{R}=I_{p^2}-\frac12 \mathrm{P}_{\mathrm{diag}}((I_p\otimes R) + (R\otimes I_p)).
\end{align}
\end{lemma}
\begin{proof}[Proof of Lemma~\ref{Lemma g - derivative}]
        First, observe that for a fixed $a>0$, expansion of the function $\varepsilon\mapsto(a + \varepsilon)^{-1/2}$ around $0$, gives $a^{-1/2} -   \frac12 a^{-3/2} \varepsilon + o(\varepsilon)$. Thus, 
        \[
   \diag{K+\varepsilon U}^{-1/2}  =     (D^2+\diag{U} \varepsilon)^{-1/2} = D^{-1}-\varepsilon\frac12 D^{-3}\diag{U}+o(\varepsilon) I_p.
        \]
        Therefore
\begin{align*}
\varepsilon^{-1}(g(K&+\varepsilon U) - g(K))  = \varepsilon^{-1} \left(\diag{K+\varepsilon U})^{-1/2} (K+\varepsilon U) \diag{K+\varepsilon U}^{-1/2}-R\right)\\ 
&=\varepsilon^{-1} \left((D^{-1}-\varepsilon\frac12 D^{-3}\diag{U})(K+\varepsilon U) (D^{-1}-\varepsilon\frac12 D^{-3}\diag{U})-R +o(\varepsilon)I_p\right)\\
    &=  D^{-1} U D^{-1}-\frac{1}{2} D^{-1}K \diag{U} D^{-3}-\frac{1}{2} D^{-3}\diag{U}K D^{-1}  +  o(1)I_p\\
    &= D^{-1} U D^{-1}-\frac{1}{2} R\, \diag{U} D^{-2}-\frac{1}{2} D^{-2}\diag{U} R  +   o(1)I_p,
    \end{align*}
    where we have used the fact that $\diag{U}$ and $D$ commute. 
Thus, 
\begin{align*}
\operatorname{vec}(g'(K;U)) & = \operatorname{vec}(D^{-1} U D^{-1}-\frac{1}{2} R\, \diag{U} D^{-2}-\frac{1}{2} D^{-2}\diag{U} R).
\intertext{On the other hand, we have}
M_R^\top (D^{-1}\otimes D^{-1})\operatorname{vec}(U)
&= \left( I_{p^2}-\frac12 ((I_p\otimes R)+(R\otimes I_p))\mathrm{P}_{\mathrm{diag}}\right)\! \operatorname{vec}(D^{-1}UD^{-1}) \\
&= \operatorname{vec}(D^{-1}U D^{-1}\!- \!\frac{1}{2} R \,\diag{ D^{-1}U D^{-1}}\!-\!\frac{1}{2} \diag{D^{-1}U D^{-1}}R ). 
\end{align*}
Since $\diag{D^{-1}U D^{-1}}= D^{-2}\diag{U}$, we obtain \eqref{g - derivative}. 
\end{proof}

\begin{proof}[Proof of Theorem~\ref{thm:conv_in_dist}]

The statement follows from  \cite[Corollary 3.2 and Corollary A.1]{hejny2025asymptotic}. It suffices to verify that the loss and the penalty satisfy the corresponding assumptions. 
First, we check the conditions for the loss
\begin{equation*}
    \ell(X,K)=-\log\det(K)+\tr(KXX^\top).
\end{equation*}
This is a smooth map on the parameter space $\Theta=\Spp$ for every fixed $X\in\mathbb{R}^p$. The derivatives are
\begin{align*}
\nabla_K \ell(X, K) = -K^{-1} + XX^\top\quad\mbox{and}\quad 
\nabla^2_K \ell(X, K) = K^{-1} \otimes K^{-1}.
\end{align*}
The expected loss is
\begin{equation*}
    G(K)=\mathbb{E}[\ell(X,K)]=-\log(\det(K))+\tr(K\Sigma^\ast),
\end{equation*}
where $\Sigma^\ast=\mathbb{E}[XX^\top]$. Let $U$ be an open neighborhood of $K^\ast=(\Sigma^\ast)^{-1}$ in $\Spp$ of the form
\begin{equation*}
    U=\{K\in \Spp: c_1<\lambda_{\min}(K), \lambda_{\max}(K)<c_2\},
\end{equation*}
where $c_2\geq c_1>0$ are positive constants satisfying $c_1<\lambda_{\min}(K^\ast)$, $c_2>\lambda_{\max}(K^\ast)$, and where $\lambda_{\min},\lambda_{\max}$ denote the smallest and largest eigenvalues, respectively.
We need to check that
 \begin{align*}
    i) &\hspace{0.2cm}  \|\nabla^2_K \ell(X, K)\Vert\leq M(X) \text{ for } K\in U, \text{ for some function }M \text{ with } \mathbb{E}[M(X)^2]<\infty. \\
    ii) &\hspace{0.3cm} G(K)\text{ is }C^3 \text{ on } U \text{ and } C=\nabla^2 G(K)|_{K=K^\ast}=\Sigma^\ast\otimes\Sigma^\ast=\Gamma^\ast \text{ is positive definite}. \\
    iii) &\hspace{0.2cm} \text{\scalebox{0.97}{$\displaystyle \mathbb{E}[\nabla_K \ell(X, K)]\big|_{K=K^\ast}=0 \text{ and } C_{\triangle} =\mathbb{E}[\nabla_{\operatorname{vec}(K)}\ell(X,K)(\nabla_{\operatorname{vec}(K)}\ell(X,K))^{\top}]\big|_{K=K^\ast}<\infty.$}} \\
    iv) &\hspace{0.2cm} 
    \text{The sequence } (\hat{K}_n)_{n\geq 1} \text{ is uniformly tight}.\\
    v) &\hspace{0.2cm} \text{\scalebox{0.97}{For every compact $\mathcal{K}\subset\Spp$; $\sup_{K\in \mathcal{K}}|\ell(X,K)|\leq L(X)$ for some $L$ with $\mathbb{E}[L(X)]<\infty$.}}
\end{align*}
First, note that the closure of $U$ is a compact subset of positive definite matrices. Therefore, condition $i)$ follows from continuity of $\nabla^2_K \ell(X, K)=K^{-1}\otimes K^{-1}$. Condition $ii)$ is clear. For $iii)$, the expectation of $\nabla_K \ell(X, K)\big|_{K=K^\ast} = -(K^\ast)^{-1} + XX^\top$ is zero. Moreover, $C_{\triangle}=\mathrm{Cov}(\mathrm{vec}(XX^\top))<\infty$, by the finiteness of the fourth moment $\mathbb{E}[\|X\|^4]<\infty$. 

To argue uniform tightness in $iv)$, note that as $n\to \infty$, the estimator $\|\hat{K}_n^{-1}-\Sigma^\ast\|_{\infty}\leq\|\hat{K}_n^{-1}-S_n\|_{\infty}+\|S_n-\Sigma^\ast\|_{\infty}\stackrel{a.s.}{\longrightarrow
}0$. The first term goes to zero by Lemma \ref{lem:cons} with $\lambda = \gamma\,n^{-1/2},  \alpha = o(n^{-1/2})$, after renormalization by $\mathrm{diag}^{-1/2}(S_n)$. The second term goes to zero by consistency of the empirical covariance $S_n$.  
Therefore $\hat{K}_n^{-1}\stackrel{a.s.}{\longrightarrow
}\Sigma^\ast$, and by continuity of the inverse map at $\Sigma^\ast$, also $\hat{K}_n\stackrel{a.s.}{\longrightarrow
} (\Sigma^\ast)^{-1}=K^\ast$. Consistency of $\hat{K}_n$ implies uniform tightness. 

To obtain a uniform envelope in $v)$, observe that the trace can be bounded as $\tr(K XX^\top)=X^\top K X\leq \lambda_{\max}(K) \|X\|_2^2$. Given a compact set $\mathcal{K}\subset \Spp$, by continuity and compactness there exist positive constants $A_{\mathcal{K}}, B_{\mathcal{K}}$, such that $|\log(\det(K))|\leq A_{\mathcal{K}}$ and $\lambda_{\max}(K)\leq B_{\mathcal{K}}$ for every $K\in\mathcal{K}$. Thus $\sup_{K\in\mathcal{K}}|\ell(X,K)|\leq A_{\mathcal{K}}+B_{\mathcal{K}}\|X\|_2^2$, which is an integrable envelope of the loss, because $\mathbb{E}[\|X\|_2^2]<\infty$.
Consequently, the loss $\ell$ satisfies all regularity conditions required in \cite[Corollary 3.2]{hejny2025asymptotic}. 

Finally, note that the penalty\footnote{ We omit the negligible $\alpha=o(n^{-1/2})$ penalization term. } $\mathrm{Pen}(K)=f(g(K))$ is not a polyhedral gauge, but a composition of the polyhedral GLASSO norm $f(M)=\|M\|_{1, \text{off}}$ and the smooth map $g(K)=\diag{K}^{-1/2}K\diag{K}^{-1/2}$. Therefore, in order to conclude the proof, we verify the assumptions of \cite[Corollary A.1]{hejny2025asymptotic}. 
Precisely, we want to verify
that for any $U_1, U_2 \in \Sym$, such that $\sign(U_1)=\sign(U_2)$, we have  \begin{equation*}
    \sign(g(K^\ast)+\varepsilon\, g'(K^\ast;U_1))=\sign(g(K^\ast)+\varepsilon\, g'(K^\ast;U_2)),
\end{equation*}
for sufficiently small $\varepsilon>0$. Write $K^\ast$ as $DRD$, where $D\in\Diagp$ and $R\in\Sppone$. By Lemma~\ref{Lemma g - derivative}, the derivative of $g$ is
\begin{equation*}
g'(K^\ast;U) =  D^{-1} U D^{-1}-\frac{1}{2} R\, \mathrm{diag}(U) D^{-2}-\frac{1}{2} D^{-2}\mathrm{diag}(U) R.
\end{equation*}
If $R_{ij}\neq0$, then the sign of $g(K^\ast)_{ij}=R_{ij}$ is not changed by small perturbations. If $R_{ij}=0$, then $\sign(g'({K^\ast};U)_{ij})=\sign(( D^{-1} U D^{-1})_{ij})=\sign(U_{ij})$, hence the above holds since $\sign(U_1)_{ij}=\sign(U_2)_{ij}$ by assumption. \cite[Corollary A.1]{hejny2025asymptotic} completes the proof.
\end{proof}

\subsection{Proof of Theorem \ref{thm:pattern_recovery}}

\begin{lemma}\label{lem:MRtilde}\ 
\begin{itemize}
    \item[(i)]

If $R\in\Sppone$, then the matrix 
\[
\tilde{M}_R = M_R+\mathrm{P}_{\mathrm{diag}} 
\]
is invertible with the inverse given by
    \[
\tilde{M}_R^{-1}=\mathrm{P}_{\mathrm{diag}}^\bot+\frac12\mathrm{P}_{\mathrm{diag}}((I_p\otimes R)+(R\otimes I_p)).
    \]
    \item[(ii)] Let
\begin{align}\label{eq:Gammatilde}
    \tilde{\Gamma}&=\tilde{M}_{R^\ast}^{-1}((R^\ast)^{-1}\otimes (R^\ast)^{-1}).
\end{align}
    We have 
    \[
    \tilde{\Gamma}=\mathrm{P}_{\mathrm{diag}}^\bot((R^\ast)^{-1}\otimes(R^\ast)^{-1})+\frac12\mathrm{P}_{\mathrm{diag}}(((R^\ast)^{-1}\otimes I_p)+(I_p\otimes (R^\ast)^{-1})).
\]
Moreover, the matrix $\tilde{\Gamma}_{\Supp\Supp}$ is invertible.
\end{itemize}
\end{lemma}
\begin{proof}[Proof of Lemma~\ref{lem:MRtilde}]
(i)    Denote $O_R=(I_p\otimes R)+(R\otimes I_p)$ and $N_R = \mathrm{P}_{\mathrm{diag}}^\bot+\frac12\mathrm{P}_{\mathrm{diag}} O_R$. First, observe that for any $X\in \R^{p\times p}$, we have 
    \begin{align*}
    \frac{1}{2}\mathrm{P}_{\mathrm{diag}} O_R\mathrm{P}_{\mathrm{diag}} \operatorname{vec}(X)&=\frac{1}{2}\mathrm{P}_{\mathrm{diag}}\operatorname{vec}(R\,\diag{X}+\diag{X}R))=\operatorname{vec}(\diag{X})\\
    &=\mathrm{P}_{\mathrm{diag}}\operatorname{vec}(X),
    \end{align*}
    which implies that $\frac{1}{2}\mathrm{P}_{\mathrm{diag}} O_R\mathrm{P}_{\mathrm{diag}}=\mathrm{P}_{\mathrm{diag}} $ on $\operatorname{vec}(\R^{p\times p})=\R^{p^2}$. 
    
 We have    \begin{align*}
    N_R\tilde{M}_R&= \left(\mathrm{P}_{\mathrm{diag}}^\bot+\frac12\mathrm{P}_{\mathrm{diag}} O_R\right)\left(I_{p^2}-\frac12 \mathrm{P}_{\mathrm{diag}} O_R+\mathrm{P}_{\mathrm{diag}}\right)\\
        &=\mathrm{P}_{\mathrm{diag}}^\bot + \frac12\mathrm{P}_{\mathrm{diag}}  O_R - \frac14 \mathrm{P}_{\mathrm{diag}}  O_R \mathrm{P}_{\mathrm{diag}}  O_R+ \frac12 \mathrm{P}_{\mathrm{diag}} O_R\mathrm{P}_{\mathrm{diag}} \\
        &=\mathrm{P}_{\mathrm{diag}}^\bot + \frac12\mathrm{P}_{\mathrm{diag}}  O_R - \frac12\mathrm{P}_{\mathrm{diag}}  O_R+ \mathrm{P}_{\mathrm{diag}} =\mathrm{P}_{\mathrm{diag}}^\bot+\mathrm{P}_{\mathrm{diag}} = I_{p^2},
    \end{align*}
which implies that $\tilde{M}_R^{-1}=N_R$. 

(ii) The formula for $\tilde{\Gamma}$ follows directly from (i). \\
We show invertibility of $\tilde{\Gamma}_{\Supp\Supp}$. Assume $\tilde{\Gamma}_{\Supp\Supp}u_{\Supp}=0$ for some $u_{\Supp}\in\mathbb{R}^{|\Supp|}$. Our aim is to show that $u_{\Supp}=0$. Consider $U\in\R^{p\times p}$ such that $\operatorname{vec}(U)_{\Supp}=u_{\Supp}$ and $\operatorname{vec}(U)_{\Supp^c}=0$. We have
\begin{align*}
0=\tilde{\Gamma}_{\Supp\Supp}u_{\Supp}&=(\tilde{\Gamma}\operatorname{vec}(U))_{\Supp}=\operatorname{vec}(\odiag X)_{\Supp}+\frac{1}{2}\operatorname{vec}(\diag{XR^\ast+R^\ast X})_{\Supp},
\end{align*}
where we denoted $X=(R^\ast)^{-1}U(R^\ast)^{-1}$. In particular, for all $(i,j)\in \Supp$ with $i\neq j$, we have $X_{ij}=0$. On the other hand, by definition of $S$, we have $R^\ast_{ij}=0$ for $(i,j)\in \Supp^c$. Thus, 
\[
\frac12(XR^\ast+R^\ast X)_{ii} = \sum_{j} X_{ij}R_{ji}^\ast =X_{ii}. 
\]
This implies that 
\begin{align*}
\operatorname{vec}(\odiag X)_{\Supp}&+\frac{1}{2}\operatorname{vec}(\diag{XR^\ast+R^\ast X})_{\Supp} = \operatorname{vec}(X)_{\Supp} = ((R^\ast)^{-1}\otimes (R^\ast)^{-1}\operatorname{vec}(U) )_{\Supp} \\
&=((R^\ast)^{-1}\otimes (R^\ast)^{-1})_{\Supp\Supp} \,u_{\Supp}.
\end{align*}
Positive definiteness of $R^\ast$ implies positive definiteness of $((R^\ast)^{-1}\otimes (R^\ast)^{-1})_{\Supp\Supp}$. Thus, we obtain $u_{\Supp}=0$ and the proof is complete.
\end{proof}

\begin{lemma}\label{lem:chainrule}
   For a convex function $\psi\colon \Sym\to \R$ and a linear map $L\colon \Sym\to \Sym$, 
\[
\operatorname{vec}(\partial(\psi\circ L)(x)) = A^\top \operatorname{vec}(\partial \psi(L x)),
\]
where $A$ is defined via $\operatorname{vec}(L v) = A \operatorname{vec}(v)$ for any $v\in\Sym$.  
\end{lemma}
 
\begin{lemma}\label{lem:subgradients}
Let $f\colon \Spp\to \R$ be defined by $f(M)=\|M\|_{1,\mathrm{off}}$.
If $\sign(U)=\sign(R)$, then
\[
\partial_U f'(R;U)=\partial f(R).
\]
\end{lemma}

\begin{proof}
For fixed $R$, define
\[
g(U):=f'(R;U).
\]
By the directional derivative formula for $f$,
\[
g(U)=\sum_{i\neq j:\,R_{ij}\neq 0}\sign(R_{ij})\,U_{ij}
+\sum_{i\neq j:\,R_{ij}=0}|U_{ij}|.
\]
Now assume $\sign(U)=\sign(R)$. Then for every $i\neq j$ such that $R_{ij}=0$, we also have $U_{ij}=0$.

Hence, the subdifferential of $g$ with respect to $U$ is given entrywise by
\[
(\partial_U g(U))_{ij}=
\begin{cases}
\{\sign(R_{ij})\}, & R_{ij}\neq 0,\\[1mm]
[-1,1], & R_{ij}=0,\ i\neq j,\\[1mm]
\{0\}, & i=j.
\end{cases}
\]
But this is exactly the subdifferential of the off-diagonal $\ell_1$ norm at $R$.
Therefore,
\[
\partial_U f'(R;U)=\partial f(R).
\]
\end{proof}

For a non-empty set $B$ define the parallel space 
by
\[
\mathrm{par}(B) = \mathrm{span}\{b-b'\colon b,b'\in B\}.
\]
Then, for any $b_0\in B$, 
\[
\mathrm{aff}(B) = b_0 + \mathrm{par}(B)
\]
is the affine hull of $B$, i.e., the smallest affine space containing $B$. 

\begin{lemma}\label{lem:ri}
Let $V$ be a finite‑dimensional real vector space, $A\subset V$ a linear subspace, and $B\subset V$ a non‑empty compact convex set. Assume that $A\cap \mathrm{par}(B) =\{0\}$. Then,  
\[
A+\mathrm{cone}(B) = V
\quad\iff\quad
        A \cap \mathrm{ri}(B) \neq \emptyset,
\]
where $\mathrm{cone}(B) = \{ \lambda b\colon b\in B, \lambda> 0\}$ and $\mathrm{ri}$ is the interior of $B$ relative to the affine hull of $B$. 
\end{lemma}
\begin{proof}[Proof of Lemma \ref{lem:ri}]
 Decompose $V=A\oplus A^{\perp}$ and let \(P\colon V\to A^\perp\) denote the orthogonal projection onto the complement
  \(A^\perp\). Since \(A\cap \mathrm{par}(B)=\{0\}\), the restriction
  \[
    P\bigl|_{\mathrm{aff}(B)} \colon \mathrm{aff}(B)\longrightarrow A^\perp
  \]
  is injective, hence affine‐bijective onto its image. Indeed, pick $x,y\in \mathrm{aff}(B)$ and assume that $P(x)=P(y)$. By linearity of $P$, we have $x-y\in \ker P=A$. Moreover, we have also $x-y\in \mathrm{par}(B)$, so that the assumption forces $x=y$, proving injectivity. An injective affine map is automatically a bijection onto its image.

  In particular, we obtain  $ P(\mathrm{ri}(B)) =\mathrm{ri}(P(B))$ so that 
  \[    0\in\mathrm{ri}(P(B))
    \iff
    \exists\,b\in\mathrm{ri}(B)\text{ with }P(b)=0
    \iff
    A\cap\mathrm{ri}(B)\neq\emptyset.
  \]
  Next, observe that
  \[
    A + \mathrm{cone}(B) = V
    \iff
    P(A+\mathrm{cone}(B)) = P(V)=A^\perp
    \iff
    \mathrm{cone}(P(B))=A^\perp,
  \]
  since \(P\) is linear and \(P(A)=\{0\}\).
  Finally we invoke:
    If \(K\subset W\) is a nonempty compact convex subset of a real vector space \(W\),
    then
    \[
      \mathrm{cone}(K)=W
      \quad\iff\quad
      0\in\mathrm{ri}(K).
    \]
  Applying this result to \(K=P(B)\subset A^\perp\) gives
  \(\mathrm{cone}(P(B))=A^\perp\iff 0\in\mathrm{ri}(P(B))\).  Chaining all
  the equivalences,
  \[
    A+\mathrm{cone}(B)=V
    \iff
    \mathrm{cone}(P(B))=A^\perp
    \iff
    0\in\mathrm{ri}(P(B))
    \iff
    A\cap\mathrm{ri}(B)\neq\emptyset,
  \]
  proving the theorem.
\end{proof}

We are now ready to prove the main result of Section \ref{sec:recovery}. 
\begin{proof}[Proof of Theorem~\ref{thm:pattern_recovery}]
The proof is constructive and shows how one can derive the irrepresentability condition \eqref{new irrepresentability condition} from the asymptotic distribution \eqref{vec(U) objective}.
For the PCGLASSO, the penalty in \eqref{vec(U) objective} is $\operatorname{Pen}(K)=\| R\|_{1,\mathrm{off}}$, 
  which can be written as $\operatorname{Pen}(K)= f(g(K))$, where 
   \[
   f(M)=\|M\|_{1,\mathrm{off}}\quad\mbox{and}\quad g(K)=\diag{K}^{-1/2}K \diag{K}^{-1/2}.  
   \]
For notational simplicity, we omit the $o(1)$ penalization term for finite $n$. This term will not matter in the limit. Also, to ease notation, we write $K^\ast$ as $DRD$ instead of $D^\ast R^\ast D^\ast$. 
   The directional derivative of $\operatorname{Pen}$ in a direction $U\in\Sym$ is 
   \begin{equation*}
      \operatorname{Pen}'(K^\ast; U) = f'(g(K^\ast);g'(K^\ast;U))=f'(R ;g'(K^\ast;U)).
   \end{equation*}
Since the objective in  \eqref{vec(U) objective}  is strictly convex, the minimizer $\hat{U}$ satisfies 
\begin{align*}
    0 \in \Gamma^\ast\operatorname{vec}(\hat{U})-W + \gamma \operatorname{vec}\left( \partial_{U}\left( f'(R;\cdot) \circ g'(K^\ast;\cdot)\right)(\hat{U})\right),
\end{align*}
where 
\[
\Gamma^\ast = (K^\ast)^{-1}\otimes (K^\ast)^{-1} = (D^{-1}\otimes D^{-1}) (R^{-1}\otimes R^{-1})(D^{-1}\otimes D^{-1}).
\]
The directional derivative of $g$ is computed in Lemma \ref{Lemma g - derivative}:
\begin{align*}
    \operatorname{vec}(g'(K^\ast;U)) &= M_R^\top (D^{-1}\otimes D^{-1})\operatorname{vec}(U).
\end{align*}
Thus, by the subgradient chain rule (see Lemma \ref{lem:chainrule}), we obtain
\begin{align*}
    W \in \Gamma^\ast\operatorname{vec}(\hat{U}) +  \gamma  (D^{-1}\otimes D^{-1})M_{R} \, \operatorname{vec}\left(\partial_U f'(R;\cdot)(g'(K;\hat{U})) \right).
\end{align*}
Let $\langle U_{K^\ast}\rangle=\mathrm{span}\{U\in \Sym\colon \sign(U)=\sign(K^\ast)\}$ be the pattern space of $K^\ast$; i.e. the subspace of matrices of the same sparsity structure as $K^\ast$. Clearly $\langle U_R\rangle=\langle U_{K^\ast}\rangle$. Importantly, we see that $g'(K^\ast;\cdot)$ preserves  the pattern space, i.e., 
\begin{align}\label{eq:preserv}
g'(K^\ast;\langle U_{K^\ast}\rangle)\subset \langle U_{K^\ast}\rangle    
\end{align}
Indeed, suppose that  $U\in\Sym$ with $\sign(U)=\sign(K^\ast)$. Then, by Lemma \ref{Lemma g - derivative} 
\[
 g'(K;U) =  D^{-1} U D^{-1}-\frac{1}{2} R\, \diag{U} D^{-2}-\frac{1}{2} D^{-2}\diag{U} R.
\]
It is now clear that $K_{ij}^\ast=0=U_{ij}=R_{ij}$ implies that $g'(K;U)_{ij}=0$ and thus $g'(K;U)\in \langle U_{K^\ast}\rangle$. By linearity, we obtain \eqref{eq:preserv}. 

By the above fact and Lemma \ref{lem:subgradients}, we obtain for any $U\in\langle U_{R}\rangle$, 
\[
\partial_U f'(R;\cdot)(g'(K;U)) = \partial f(R).
\]
Now, we can express the limiting probability of support recovery as
\begin{align}\label{limiting probability in proof}
\lim_{n\rightarrow \infty}&\mathbb{P}\left(\sign(\hat{K}_n)=\sign(K^\ast)\right) =\mathbb{P}\left( \hat{U}\in\langle U_{K^\ast}\rangle\right)\nonumber \\ &=\mathbb{P} \left(W \in \Gamma^\ast\operatorname{vec}(\langle U_{K^\ast}\rangle) + \gamma(D^{-1}\otimes D^{-1})M_R   \operatorname{vec}(\partial f( R))\right)\nonumber\\
&=\mathbb{P} \left(\tilde{W} \in (R^{-1}\otimes R^{-1})\operatorname{vec}(\langle U_R\rangle) +\gamma M_R \operatorname{vec}(\partial f(R))\right),
\end{align}
where we denoted $\tilde{W}=(D\otimes D)W$ and used the fact that $D^{-1}\langle U_R\rangle D^{-1}=\langle U_R\rangle$ for any diagonal matrix $D$ with positive diagonal entries. 
Since $\tilde{W}$ is Gaussian, the probability of the pattern recovery goes to $1$ as $\gamma\to\infty$ if and only if the set 
\[
(R^{-1}\otimes R^{-1})\operatorname{vec}(\langle U_R\rangle) +\gamma M_R \operatorname{vec}(\partial f(R))
\]``fills out'' the whole space as $\gamma\to\infty$. Since $\mathrm{P}_{\mathrm{diag}} \operatorname{vec}(\partial f(R)) = \operatorname{vec}(\diag{\partial f(R)})=0$, we have $M_R \operatorname{vec}(\partial f(R))=\tilde{M}_R \operatorname{vec}(\partial f(R))$. Equivalently, after multiplying by $\tilde{M}_R^{-1}$, we need to show that $\cup_{\gamma>0}( A+\gamma B)=\operatorname{vec}(\Sym)=:V$ with (recalling \eqref{eq:Gammatilde})
\[
A = \tilde{\Gamma}\operatorname{vec}(\langle U_R\rangle)\quad\mbox{ and }\quad B=\operatorname{vec}(\partial f(R)).
\]
We note that $A$ is a linear subspace of $V$, while $B$ is a compact convex set in $V$. 
We will first show that $A\cap\mathrm{par}(B)=\{0\}$. Suppose that $v\in A\cap\mathrm{par}(B)$. Then, there exists $u\in \operatorname{vec}(\langle U_R\rangle)$ such that 
\begin{align}\label{eq:uv}
\tilde{\Gamma} u = v \in \mathrm{par}(B).
\end{align}
Denote by $S$ the support of $K^\ast$, i.e., $S=\{(i,j)\in \{1,\ldots,p\}^2\colon  K^\ast_{ij}\neq 0\}$.
We have that $u\in \operatorname{vec}(\langle U_R\rangle)$ if and only if $u_{\Supp^c}=0$ and $v \in \mathrm{par}(B)$ if and only if $v_{\Supp}=0$. Thus, \eqref{eq:uv} implies that 
$\tilde\Gamma_{\Supp\Supp} u_{\Supp} =0$ and $\tilde\Gamma_{\Supp^c\Supp} u_{\Supp} = v_{\Supp^c}$.  
Since $\tilde\Gamma_{\Supp\Supp}$ is invertible, we obtain $u=0$, which further implies that $v=0$. Thus, $A\cap\mathrm{par}(B)=\{0\}$ as claimed. 

By Lemma \ref{lem:ri}, we have 
\[
\bigcup_{\gamma>0}(A+\gamma B) = A + \mathrm{cone}(B) = V  
\]
if and only if $A\cap \mathrm{ri}(B)\neq \emptyset$, i.e.
\begin{equation}\label{irrepresentability PCGLASSO with subdiff}
    \tilde\Gamma \operatorname{vec}(\langle U_{R}\rangle)\cap \operatorname{vec}(\mathrm{ri}(\partial f(R)))\neq \emptyset.
\end{equation}
Moreover, if \eqref{irrepresentability PCGLASSO with subdiff} holds, then by Gaussianity of $\tilde{W}$, there is $c>0$ such that the limiting probability can be bounded from below by $1-e^{-c\gamma^2}$, for all $\gamma>0$. 

It remains to argue that \eqref{irrepresentability PCGLASSO with subdiff} is equivalent to the irrepresentability condition \eqref{new irrepresentability condition}. Denote $\pi=\operatorname{vec}(\odiag{K^\ast})$ and observe that 
\begin{align*}
\operatorname{vec}(\langle U_R\rangle)
&=\{u\in\operatorname{vec}(\Sym)\colon u_{\Supp^c}=0\},\\
\operatorname{vec}(\mathrm{ri}(\partial f(R)))
&=\{z\in\operatorname{vec}(\odiag{\Sym})\colon z_{\Supp}=\operatorname{vec}(\pi)_{\Supp},\;\|z_{\Supp^c}\|_\infty<1\}.
\end{align*}
Partition any vector $u\in\operatorname{vec}(\Sym)$ as $u^\top=(u_{\Supp}^\top ,u_{\Supp^c}^\top)$ and write
\[
\tilde\Gamma
=\begin{pmatrix}\tilde\Gamma_{\Supp\Supp}&\tilde\Gamma_{\Supp\Supp^c}\\
                \tilde\Gamma_{\Supp^c\Supp}&\tilde\Gamma_{\Supp^c\Supp^c}\end{pmatrix}.
\]
Suppose \eqref{irrepresentability PCGLASSO with subdiff}, so that there exists a vector $u$ such that 
\[
u_{\Supp^c}=0,\qquad \tilde\Gamma u=z,\qquad z_{\Supp}=\operatorname{vec}(\pi)_{\Supp},\qquad \|z_{\Supp^c}\|_\infty <1.
\]
In particular, 
\[
\tilde\Gamma_{\Supp\Supp}\,u_{\Supp}=\operatorname{vec}(\pi)_{\Supp}\qquad\mbox{and}\qquad
\tilde\Gamma_{\Supp^c\Supp}\,u_{\Supp}=z_{\Supp^c},
\]
so $u_{\Supp}=(\tilde\Gamma_{\Supp\Supp})^{-1}\operatorname{vec}(\pi)_{\Supp}$. Hence
\[
z_{\Supp^c}
=\tilde\Gamma_{\Supp^c\Supp}(\tilde\Gamma_{\Supp\Supp})^{-1}\operatorname{vec}(\pi)_{\Supp}
\]
and condition $\|z_{\Supp^c}\|_\infty<1$ gives exactly \eqref{new irrepresentability condition}. 

Now, suppose \eqref{new irrepresentability condition} and let  $z_{\Supp}=\operatorname{vec}(\pi)_{\Supp}$  with $\|z_{\Supp^c}\|_\infty<1$. Then, $u=\tilde{\Gamma}^{-1} z$ belongs to $\operatorname{vec}(\langle U_R\rangle)$, which completes the proof of the first part. 

If \eqref{new irrepresentability condition} is violated, then \eqref{irrepresentability PCGLASSO with subdiff} also does not hold. As a result, the intersection
\begin{equation*}
\tilde\Gamma \operatorname{vec}(\langle U_{R}\rangle) \cap \operatorname{vec}(\mathrm{aff}(\partial f(R)))
\end{equation*}
contains exactly one element, say $v_0$, such that $v_0\notin\operatorname{vec}(\mathrm{ri}(\partial f(R)))$.  (Note that the uniqueness of $v_0$ follows from the fact that \(A \cap \mathrm{par}(B) = \{0\}\), established above.)
We now consider the limiting probability \eqref{limiting probability in proof}, which can be expressed as 
\begin{equation*}
   \lim_{n\rightarrow \infty}\mathbb{P}\left(\sign(\hat{K}_n)=\sign(K^\ast)\right) = \mathbb{P} \left(\tilde{M}_R^{-1}\tilde{W} \in \mathcal{K}_\gamma\right),
\end{equation*}
where 
\begin{align*}
\mathcal{K}_\gamma &= \tilde{\Gamma}\operatorname{vec}(\langle U_R\rangle) + \gamma \operatorname{vec}(\partial f(R))\\
&=\tilde{\Gamma}\operatorname{vec}(\langle U_R\rangle) + \gamma (\operatorname{vec}(\partial f(R)) - v_0).
\end{align*}
Fix any $\gamma > 0$. Since $ 0 \notin \gamma(\operatorname{vec}(\mathrm{ri}(\partial f(R))) - v_0)$,
we also have $0 \notin \mathrm{ri}(\mathcal{K}_{\gamma})$. By convexity, the set $\mathcal{K}_{\gamma}$ must lie entirely on one side of some separating hyperplane through the origin. As a result, by symmetry, the centered Gaussian vector $\tilde{M}_R^{-1}\tilde{W}$ satisfies
\begin{equation*}
\mathbb{P}(\tilde{M}_R^{-1}\tilde{W} \in \mathcal{K}_{\gamma}) \leq \frac{1}{2}.
\end{equation*}
This completes the proof.
\end{proof}

\bibliographystyle{plainnat}
\bibliography{bibliography}

\end{document}